\documentclass{article}

\usepackage{graphicx}
\usepackage{amsmath}
\usepackage{amssymb}
\usepackage{mathrsfs}
\usepackage{IEEEtrantools}
\usepackage{amsthm}
\usepackage{multirow}
\usepackage{esint}
\usepackage{mathtools}
\usepackage{fullpage}
\usepackage{subcaption}

\newcommand{\RR}{\mathbb{R}}
\newcommand{\TT}{\mathbb{T}}

\newtheoremstyle{plain}
{\topsep}
{1.5\topsep}
{\itshape}
{0pt}
{\bfseries}
{.}
{5pt plus 1pt minus 1pt}
{}

\newtheoremstyle{definition}
{\topsep}
{1.5\topsep}
{\normalfont}
{0pt}
{\bfseries}
{.}
{5pt plus 1pt minus 1pt}
{}

\newtheoremstyle{remark}
{0.5\topsep}
{0.8\topsep}
{\normalfont}
{0pt}
{\itshape}
{.}
{5pt plus 1pt minus 1pt}
{}

\theoremstyle{plain} \newtheorem{lemma}{Lemma}
\theoremstyle{plain} \newtheorem{thm}[lemma]{Theorem}
\theoremstyle{plain} 
\theoremstyle{plain} \newtheorem{lem}[lemma]{Lemma}
\theoremstyle{definition} 
\theoremstyle{remark} \newtheorem*{rmk}{Remark}
\theoremstyle{plain} \newtheorem{cor}[lemma]{Corollary}
\newcommand{\dd}{\,\mathrm{d}}

\begin{document}
\allowdisplaybreaks
\title{On compressible magnetic relaxation in planar symmetry}
\date{}
\author{Taehun Kim}
\maketitle
\pagestyle{plain}

\begin{abstract}
We consider the compressible Magnetic Relaxation Equations on the three-dimensional torus $\TT^{3}$. The system is derived from compressible magnetohydrodynamics (MHD) by replacing the acceleration term with a Darcy-type friction. Under planar symmetry, we establish three main results: (1) local well-posedness for smooth initial data, (2) magnetic relaxation for smooth perturbations of constant steady states, and (3) the absence of vacuum states or implosions prior to and at the time of a potential singularity.
\end{abstract}

\section{Introduction}\label{sec1}

Steady states of the incompressible Euler equations are important because they represent the long-term dynamics of incompressible fluids (see Arnold and Khesin \cite{AK99}; Bedrossian and Masmoudi \cite{BM15}; Ionescu and Jia \cite{IJ21, IJ20}; and a recent review paper by Drivas and Elgindi \cite{DE22}). Among the many mathematical studies of Euler steady states, one area of interest is the construction of steady states with remarkable geometric properties. This has been achieved in various works. For instance, Enciso and Peralta-Salas \cite{EPS12} proved that there exists a smooth diffeomorphism from an arbitrarily given locally finite link to some steady state of the incompressible Euler equation. Furthermore, recent works such as Gavrilov \cite{Gavrilov19} and Constantin, La, and Vicol \cite{CLV19} have established the existence of smooth steady incompressible flows with compact support. Finally, there have been geometric approaches that led to the development of topological hydrodynamics; for a comprehensive survey of this field, see Moffatt's recent review \cite{Moffatt21} and references therein.

To address the problem of finding steady flows with a given topology, Arnold \cite{Arnold74} proposed a dynamic approach, later refined by Moffatt \cite{Moffatt85}. This approach exploits the fact that the incompressible Euler equations and the magnetohydrostatic (MHS) equations share the same steady states. The key idea is to use a variant of the MHD equations that is dissipative (with respect to a suitable energy) and preserves the topology of the magnetic field lines throughout the evolution, a process called magnetic relaxation. One may then ask whether a global solution exists and determine whether the velocity field $u$, and the magnetic field $B$ converge in the infinite-time limit to an MHS steady state $u=0$, $B=B_{\text{steady}}$. This magnetic relaxation procedure has also been used for numerically computing MHS steady states, as demonstrated by codes such as PIES \cite{PIESpaper}, HINT \cite{HINTpaper2, HINTpaper1}, SIESTA \cite{SIESTApaper}, and the more recent MRX \cite{MRXpaper}.

Among various magnetic relaxation frameworks, the one that has received the most mathematical attention is Moffatt's Magnetic Relaxation Equations (MRE). This system is derived by replacing the acceleration term $(\partial_{t}+u\cdot \nabla)u$ in the standard MHD equations with a friction term $(-\Delta)^{\alpha}u$: \begin{eqnarray} \label{20260218eq1}
\begin{cases}
(-\Delta)^{\alpha}u = -\nabla p + B \cdot \nabla B, \\
\partial_{t}B + u \cdot \nabla B = B \cdot \nabla u, \\
\nabla \cdot u = \nabla \cdot B =0.
\end{cases}
\end{eqnarray} Here, $\alpha \ge 0$ is a regularization parameter. For the case $\alpha=0$, Brenier \cite{Brenier14} introduced the concept of dissipative solutions---resembling Lions' dissipative weak solutions for the Euler equations \cite{Lions96}---and proved their global existence in $2$D. The local existence of $H^{s}(\TT^{d})$ solutions was established by Beekie, Friedlander, and Vicol \cite{BFV22} for $\alpha \ge 0$ and $s>\frac{d}{2}+1$, results which were later refined by Bae, Kwon, and Shin \cite{BKS25}. Regarding global existence, \cite{BFV22} proved the case for $\alpha, s > \frac{d}{2}+1$ and showed that the velocity field $u$ converges to zero in the Lipschitz norm. Subsequently, \cite{BKS25} extended these global results to the regime $\alpha \ge \frac{d}{2}+1$ and $s \ge 1$. See also Tan \cite{Tan26} for the global existence and velocity relaxation of weak solutions with $L^{2+\epsilon}$ initial data when $\alpha\ge \frac{d}{2}+1$.

However, the relaxation of the velocity field $u$ does not necessarily guarantee the convergence of the magnetic field $B$ to a steady state, and indeed \cite{BFV22} provided an example where the $H^{1}$ norm of $B$ diverges as $t \to \infty$. Nevertheless, some affirmative magnetic relaxation results exist. Paper \cite{BFV22} proves the asymptotic stability of the $2$D constant steady state $(u, B)=(0, e_{1})$ under high Sobolev perturbations. Mohammadkhani and Nguyen \cite{MN25} proved similar relaxation results for non-constant shear flows on $\TT \times (-1, 1)$ under Neumann-type boundary conditions for both $u$ and $B$, as well as for certain 2.5D solutions.

\subsection{Compressible Magnetic Relaxation}

In numerical implementations for finding MHS steady states, using magnetic relaxation, a compressible velocity field is sometimes used. Enforcing incompressibility requires computing the inverse Laplacian for each step, as in PIES \cite{PIESpaper}. But this expensive computation could be avoided if one uses a compressible setup. Indeed, SIESTA \cite{SIESTApaper} employs the Kulsrud-Kruskal MHD energy minimization principle \cite{KK58}, which assesses plasma stability through an ideal MHD energy functional which includes compressional effects.

In analogy with the incompressible MRE, an isentropic compressible MRE model (corresponding to $\alpha=0$ in \eqref{20260218eq1}) can be formulated by replacing the acceleration term in ideal inviscid 3D compressible MHD with Darcy-type friction:
\begin{eqnarray}\label{20260110eq1}
\begin{cases}
\partial_{t}\rho + \nabla \cdot (\rho u)=0, \\
\rho u = -\nabla p + (\nabla \times B)\times B, \\
\partial_{t}B + \nabla \cdot (u \otimes B) = B \cdot \nabla u, \\
\nabla \cdot B=0,
\end{cases}
\end{eqnarray} where the pressure is given by the ideal gas law with constant entropy: \[
p=\frac{1}{\gamma}\rho^{\gamma}.
\] Here, $\gamma>1$ is the adiabatic gas constant. For complex, non-monatomic fluids, $\gamma$ approaches $1$, reflecting a departure from simpler ideal gas models. It is worth noting that neither system \eqref{20260110eq1} nor the ideal inviscid compressible MHD exhibit any explicit diffusion (they are topology preserving with respect to magnetic streamlines), which makes it difficult to prove any relaxation results. 

To the author's knowledge, fine analysis of ideal inviscid compressible MHD has only been done under certain symmetry assumptions, which reduce the problem to a 1D system (see Chen, Young, and Zhang \cite{CYZ13}; Kang \cite{Kang18}; An, Chen, and Yin \cite{ACY22}; and Ding and Yin \cite{DY25}). Such reduction is standard for the analysis of hyperbolic systems (going back to John \cite{John74} and even earlier), and is the state of the art for compressible MHD. Accordingly, we will continue our analysis under planar symmetry, by considering the following ansatz: \[
B=(B_{0}, 0, B^{z}(t, x)), \quad u=(u^{x}(t, x), u^{y}(t, x), u^{z}(t, x)), \quad \rho=\rho(t, x) \qquad (t \ge 0, x \in \TT),
\] where $B_{0} \neq 0$ is a constant. In this paper, we denote the initial data of $B$ by $B_{t=0}$ instead of $B_{0}$. Then the system \eqref{20260110eq1} reads as
\begin{eqnarray} \label{20260130eq1}
\begin{cases}
\partial_{t}\rho + \partial_{x}(\rho u^{x}) =0, \\
\rho u^{x} = -\rho^{\gamma-1}\partial_{x}\rho -B^{z}\partial_{x}B^{z}, \\
\rho u^{y} = 0, \\
\rho u^{z} = B_{0}\partial_{x}B^{z}, \\
\partial_{t}B^{z} + B^{z}\partial_{x}u^{x} + u^{x}\partial_{x}B^{z} = B_{0}\partial_{x}u^{z}.
\end{cases}
\end{eqnarray}
By substituting the expressions for $u^{x}$ and $u^{z}$, we derive the following system:
\begin{eqnarray}\label{20260110eq2}
\begin{cases}
\partial_{t}\rho = \partial_{x}^{2}\left ( \frac{1}{\gamma}\rho^{\gamma} + \frac{1}{2}B^{2}\right ), \\
\partial_{t}B = \partial_{x}\left ( \frac{B}{\rho}\partial_{x}\left ( \frac{1}{\gamma}\rho^{\gamma}+\frac{1}{2}B^{2}\right ) + \frac{B_{0}^{2}}{\rho}\partial_{x}B\right ),
\end{cases}
\end{eqnarray} where, for simplicity, we have replaced $B^{z}$ with $B$. Note that when $B_{0} \neq 0$, the diffusion in the $B$-equation is guaranteed to be nondegenerate if the density $\rho$ remains finite.

The goal of this paper is to prove that the 1D compressible MRE model \eqref{20260110eq2} is locally well-posed, and that magnetic relaxation occurs in the vicinity of constant steady states. These results are summarized in the following two theorems.

\begin{thm} \label{20260110thm1}
Assume that $1<\gamma<2$ and $B_{0}\neq 0$. Consider smooth initial data $\rho_{t=0}, B_{t=0} \in C^{\infty}(\TT)$. Then there exists a maximal existence time $T=T(\gamma, B_{0}, \|\rho_{t=0}\|_{H^{3}}, \|B_{t=0}\|_{H^{3}}, \|\rho_{t=0}^{-1}\|_{L^{\infty}})>0$ and unique smooth solutions $\rho, B \in C^{\infty}([0, T) \times \TT)$ to the system \eqref{20260110eq2}, such that if $T<\infty$ then $\lim_{t \to T-} \left (\|\rho\|_{H^{3}}+\|B\|_{H^{3}}\right ) = \infty$.
 
Moreover, there exists a constant $M=M(\gamma, B_{0}, \|\rho_{t=0}\|_{L^{\infty}}, \|B_{t=0}\|_{L^{\infty}}, \|\rho_{t=0}^{-1}\|_{L^{\infty}})>0$, such that $\rho(t, x)$, $|B(t, x)|$, and $\rho^{-1}(t, x)$ are bounded from above by $M$ for all $0 \le t < T$ and $x \in \TT$. In particular, $\rho \ge M^{-1} > 0$ on $[0, T)\times \TT$.
\end{thm}

\begin{thm} \label{20260110thm2}
Assume that $1<\gamma <2$ and $B_{0} \neq 0$. Let $\overline{\rho}>0$ and $\overline{B} \in \RR$ be two constants. Let $s \ge 6$ be an integer. Then there exists a constant $\epsilon=\epsilon(\gamma, B_{0}, \overline{\rho}, \overline{B})$ such that if initial data $\rho_{t=0}, B_{t=0} \in C^{\infty}(\TT)$ satisfy \[
\|\rho_{t=0}-\overline{\rho}\|_{H^{s}} + \|B_{t=0}-\overline{B}\|_{H^{s}} \le \epsilon, \quad \fint_{\TT} \rho_{t=0} = \overline{\rho}, \quad  \fint_{\TT} B_{t=0} = \overline{B},
\] then the solution $(\rho, B)$ attained from Theorem \ref{20260110thm1} is global in time, and satisfies \[
(\rho(t), B(t)) \xrightarrow[]{t\to\infty} (\overline{\rho}, \overline{B}) \text{ in } H^{m}(\TT)
\] for any nonnegative integer $m < s$. In particular, $u(t) \to 0$ in $H^{m-1}(\TT)$, where $u$ is related to $(\rho, B)$ via \eqref{20260130eq1}.
\end{thm}

\begin{rmk}
From the equation \eqref{20260110eq2}, it is clear that $\int_{\TT} \rho$ and $\int_{\TT} B$ are constant in time.
\end{rmk}
\begin{rmk}
The result is still true without the assumption $\gamma<2$. This redundant assumption is added to simplify the proof using the result in Theorem \ref{20260110thm1}, that $\rho$ is bounded away from zero.
\end{rmk}

There is a vast literature on stabilization effects in MHD near a constant magnetic field. Lin, Xu, and Zhang \cite{LXZ15} proved the global well-posedness of $2$D viscous incompressible MHD near a constant steady state, a result later refined by Ren, Wu, Xiang, and Zhang \cite{RWXZ14} by removing initial data admissibility conditions and including asymptotic stability results. Additionally, G\'erard-Varet and Prestipino \cite{GP17} derived boundary layer models for 3D incompressible MHD around a constant magnetic field and studied their linear stability. Our second result (Theorem \ref{20260110thm2}) establishes the relaxation to a constant steady state for the ideal inviscid compressible MRE under planar symmetry. This result provides a rigorous mathematical verification of magnetic relaxation, a topic that remains less explored in the compressible setting. %This result contributes to the limited literature in which magnetic relaxation has been rigorously proven. 
Moreover, it may serve as a first step toward the mathematical justification of numerical schemes that utilize the compressible MRE framework.

\subsection{Idea of the proof}\label{subsec1.2}

In this section, we discuss briefly the main idea to prove Theorem \ref{20260110thm1}. The key step is to prove that vacuum does not happen before potential singularities. The idea is to use a maximum principle, but the main obstacle is that the system \eqref{20260110eq2} has a mixed diffusion term: \[
\begin{cases}
\partial_{t}\rho = \rho^{\gamma-1}\partial_{x}^{2}\rho + B\partial_{x}^{2}B + (\text{lower order terms}), \\
\partial_{t}B = \rho^{\gamma-2}B \partial_{x}^{2}\rho + \frac{B^{2}+B_{0}^{2}}{\rho}\partial_{x}^{2}B + (\text{lower order terms}).
\end{cases}
\] Hence at the point where $\rho$ is minimized (or maximized), we have no information on the sign of $\partial_{x}^{2}B$, and vice versa. To remedy this, we try to search for a new variable $F(\rho, B)$ such that the equation satisfied by $F$ has the form \[
\partial_{t}F = c(\rho, B) \partial_{x}^{2}F + (\text{lower order terms}).
\] Ignoring all the lower order terms, we have \[
\partial_{t}F \approx \partial_{x}^{2}\rho \left ( \rho^{\gamma-1}\partial_{\rho}F + \rho^{\gamma-2}B \partial_{B}F\right ) + \partial_{x}^{2}B \left ( B \partial_{\rho}F + \frac{B^{2}+B_{0}^{2}}{\rho}\partial_{B}F \right ).
\] We want these second order terms to appear exactly in the expansion of $\partial_{x}^{2}F$, \[
\partial_{x}^{2}F \approx \partial_{x}^{2}\rho \partial_{\rho}F + \partial_{x}^{2}B \partial_{B}F. 
\] Hence we would want \[
\frac{\rho^{\gamma-1}\partial_{\rho}F + \rho^{\gamma-2}B\partial_{B}F}{B\partial_{\rho}F + \frac{B^{2}+B_{0}^{2}}{\rho}\partial_{B}F} = \frac{\partial_{\rho}F}{\partial_{B}F},
\] which gives \begin{eqnarray} \label{20260218eq2}
\partial_{\rho}F = \partial_{B}F \cdot \frac{-\big(B^{2}+B_{0}^{2}-\rho^{\gamma}\big ) \pm \sqrt{\big(B^{2}+B_{0}^{2}-\rho^{\gamma}\big)^{2}+4B^{2}\rho^{\gamma}}}{2B\rho}.
\end{eqnarray} The linearized version of this same relation is used to prove the relaxation results in Section \ref{subsec3.3} and \ref{subsec3.4}. Solving \eqref{20260218eq2} (without linearization) formally, using the method of characteristics, gives two different family of solutions for $F$:
\begin{IEEEeqnarray*}{rCl}
\IEEEeqnarraymulticol{3}{l}{(B^{2}+B_{0}^{2})\left ( \frac{\pm \big(B^{2}+B_{0}^{2}-\rho^{\gamma}\big )+\sqrt{\big(B^{2}+B_{0}^{2}-\rho^{\gamma}\big )^{2}+4B^{2}\rho^{\gamma}}}{B^{2}}\right )^{\frac{2}{2-\gamma}}} \\
&& \quad\mp\int_{\frac{\pm \big(B^{2}+B_{0}^{2}-\rho^{\gamma}\big )+\sqrt{\big(B^{2}+B_{0}^{2}-\rho^{\gamma}\big )^{2}+4B^{2}\rho^{\gamma}}}{B^{2}}}^{4}\frac{2B_{0}^{2}}{2-\gamma}\frac{1\mp\frac{1}{4}s}{\left (1\mp\frac{1}{2}s\right )^{2}}s^{\frac{2}{2-\gamma}}\dd s.
\end{IEEEeqnarray*}
Now that we have found a diagonalizing variable $F(\rho, B)$, we can apply a maximum (minimum) principle for these variables to prove that $\rho, B$ and $\rho^{-1}$ are bounded. Details can be found in Section \ref{subsec2.2}.

\subsection{Acknowledgments}\label{subsec1.3}

The author would like to thank Vlad Vicol for various insights and thoughtful discussions. The author was partially supported by the NSF grant DMS-2307681.

\section{Local well-posedness and absence of vacuum}\label{sec2}

\subsection{A-priori estimate}

Throughout this paper, $C$ denotes generic positive constant, possibly with subindices to indicate its dependence, whose values may change from line to line.

\begin{lem}[A-priori weighted $L^{2}$ estimate]\label{20260111lem1}
Assume that $\gamma>1$ and $B_{0} \neq 0$. Let $\rho, B \in C^{\infty}([0, T]\times \TT)$ be smooth solutions to the system \eqref{20260110eq2}. Assume that $\rho \ge \eta$ on $[0, T] \times \TT$ for some positive constant $\eta>0$. Then one has \[
\partial_{t}\int_{\TT} \left (\frac{1}{\gamma(\gamma-1)}\rho^{\gamma}+\frac{1}{2}B^{2}\dd x\right ) \le 0.
\]
\end{lem}

\begin{proof}[Proof of Lemma \ref{20260111lem1}]
We have
\begin{IEEEeqnarray*}{rCl}
\IEEEeqnarraymulticol{3}{l}{\partial_{t}\int_{\TT} \left (\frac{1}{\gamma(\gamma-1)}\rho^{\gamma}+\frac{1}{2}B^{2}\dd x\right )} \\
&=& \int_{\TT} \frac{1}{\gamma-1}\rho^{\gamma-1}\partial_{x}^{2}\left (\frac{1}{\gamma}\rho^{\gamma}+\frac{1}{2}B^{2}\right ) + B \partial_{x}\left ( \frac{B^{2}+B_{0}^{2}}{\rho}\partial_{x}B+\rho^{\gamma-2}B\partial_{x}\rho\right )\\
&=& \int_{\TT} -\rho^{2\gamma-3}\big(\partial_{x}\rho\big)^{2} - 2\rho^{\gamma-2}B\partial_{x}\rho \partial_{x}B - \frac{B^{2}+B_{0}^{2}}{\rho}\big(\partial_{x}B\big)^{2} \\
&=& -\int_{\TT} \left ( \left (\rho^{\gamma-\frac{3}{2}}\partial_{x}\rho + \frac{B}{\rho^{\frac{1}{2}}}\partial_{x}B\right )^{2} + \left (\frac{B_{0}}{\rho^{\frac{1}{2}}}\partial_{x}B\right )^{2}\right ) \le 0.
\end{IEEEeqnarray*}
Hence the claim.
\end{proof}

\begin{lem}[A-priori weighted $H^{s}$ estimate] \label{20260110lem1}
Assume that $\gamma>1$ and $B_{0}\neq 0$. Let $\rho, B \in C^{\infty}([0, T]\times \TT)$ be smooth solutions to the system \eqref{20260110eq2}. Assume that $\rho \ge \eta$ on $[0, T]\times \TT$ for some positive constant $\eta>0$. Then for any integer $s \ge 3$, one has
\begin{IEEEeqnarray*}{rCl}
\IEEEeqnarraymulticol{3}{l}{\partial_{t} \int_{\TT} \left ( \left | \rho^{\frac{\gamma}{2}-1}\partial_{x}^{s}\rho\right |^{2} +\left |\partial_{x}^{s}B\right |^{2}\right ) } \\
& \le & C_{\gamma, B_{0}, s} \left ( \|\partial_{x}^{s}\rho\|_{L^{2}}+\|\partial_{x}^{s}B\|_{L^{2}}\right )^{2} \left ( \|\partial_{x}^{\lfloor\frac{s+3}{2}\rfloor}\rho\|_{L^{2}} + \|\partial_{x}^{\lfloor\frac{s+3}{2}\rfloor}B\|_{L^{2}}\right )\left ( \eta^{-1}+\|\rho\|_{H^{\lfloor \frac{s+3}{2}\rfloor}}+\|B\|_{H^{\lfloor \frac{s+3}{2}\rfloor}}\right )^{C_{\gamma, s}}. \yesnumber \label{20260111eq7}
\end{IEEEeqnarray*}
\end{lem}

\begin{proof}[Proof of Lemma \ref{20260110lem1}]
Observe that
\begin{IEEEeqnarray*}{rCl}
\IEEEeqnarraymulticol{3}{l}{\partial_{t}\frac{1}{2}\int_{\TT} \left ( \left | \rho^{\frac{\gamma}{2}-1}\partial_{x}^{s}\rho\right |^{2} + \left | \partial_{x}^{s}B \right |^{2} \right ) \dd x} \\
&=& \int_{\TT} \rho^{\gamma-2} \partial_{x}^{s}\rho \partial_{t}\partial_{x}^{s}\rho + \partial_{x}^{s} B \partial_{t} \partial_{x}^{s}B + \frac{\gamma-2}{2}\rho^{\gamma-3}\partial_{t}\rho\left (\partial_{x}^{s}\rho\right )^{2}\dd x \\
&=& \int_{\TT} \rho^{\gamma-2} \partial_{x}^{s}\rho \partial_{x}^{s+2}\left ( \frac{1}{\gamma}\rho^{\gamma}+\frac{1}{2}B^{2} \right ) + \partial_{x}^{s}B \partial_{x}^{s+1}\left ( \frac{B}{\rho} \partial_{x}\left (\frac{1}{\gamma}\rho^{\gamma}+\frac{1}{2}B^{2}\right ) + \frac{B_{0}^{2}}{\rho}\partial_{x}B\right ) \\
&& \quad + \frac{\gamma-2}{2}\rho^{\gamma-3}\partial_{x}^{2}\left ( \frac{1}{\gamma}\rho^{\gamma}+\frac{1}{2}B^{2}\right )\left ( \partial_{x}^{s}\rho\right )^{2} \dd x \\
&=& \int_{\TT} -\rho^{\gamma-2} \partial_{x}^{s+1}\rho \partial_{x}^{s+1}\left ( \frac{1}{\gamma}\rho^{\gamma}+\frac{1}{2}B^{2} \right ) -(\gamma-2)\rho^{\gamma-3}\partial_{x}\rho \partial_{x}^{s}\rho \partial_{x}^{s+1}\left ( \frac{1}{\gamma}\rho^{\gamma}+\frac{1}{2}B^{2}\right ) \\
&& \quad - \partial_{x}^{s+1}B \partial_{x}^{s}\left ( \frac{B}{\rho} \partial_{x}\left (\frac{1}{\gamma}\rho^{\gamma}+\frac{1}{2}B^{2}\right ) + \frac{B_{0}^{2}}{\rho}\partial_{x}B\right ) + \frac{\gamma-2}{2}\rho^{\gamma-3}\partial_{x}^{2}\left ( \frac{1}{\gamma}\rho^{\gamma}+\frac{1}{2}B^{2}\right )\left ( \partial_{x}^{s}\rho\right )^{2} \dd x. \yesnumber \label{20260110eq3}
\end{IEEEeqnarray*}

By Leibniz rule, we have
\begin{IEEEeqnarray*}{rCl}
\IEEEeqnarraymulticol{3}{l}{\partial_{x}^{s+1}\left ( \frac{1}{\gamma}\rho^{\gamma}+\frac{1}{2}B^{2}\right) - \rho^{\gamma-1}\partial_{x}^{s+1}\rho - B\partial_{x}^{s+1}B} \\
&=& \sum_{\substack{k, \ell \ge 0, \\
1 \le i_{1}, \cdots, i_{k} \le s, \\
0 \le j_{1}, \cdots, j_{\ell} \le s, \\
i_{1}+\cdots +i_{k}+j_{1}+\cdots +j_{\ell}=s+1, \\
\alpha+k+\frac{\gamma}{2}\ell = \gamma}}C_{\gamma, s, \alpha, i_{1}, \ldots, i_{k}, j_{1}, \ldots, j_{\ell}}\rho^{\alpha}\partial_{x}^{i_{1}}\rho\cdots \partial_{x}^{i_{k}}\rho\partial_{x}^{j_{1}}B\cdots \partial_{x}^{j_{\ell}}B, \yesnumber \label{20260110eq4} \\
\IEEEeqnarraymulticol{3}{l}{\partial_{x}^{s}\left ( \frac{B}{\rho} \partial_{x}\left (\frac{1}{\gamma}\rho^{\gamma}+\frac{1}{2}B^{2}\right ) + \frac{B_{0}^{2}}{\rho}\partial_{x}B\right ) - \rho^{\gamma-2} B \partial_{x}^{s+1}\rho - \rho^{-1}B^{2}\partial_{x}^{s+1}B - \rho^{-1}B_{0}^{2}\partial_{x}^{s+1}B} \\
&=& \sum_{\substack{k, \ell \ge 0, \\
1 \le i_{1}, \ldots, i_{k} \le s, \\
0 \le j_{1}, \ldots, j_{\ell} \le s, \\
i_{1}+\cdots +i_{k}+j_{1}+\cdots +j_{\ell}=s+1, \\
\alpha+k+\frac{\gamma}{2}\ell = \frac{3}{2}\gamma-1}}C_{\gamma, s, \alpha, i_{1}, \ldots, i_{k}, j_{1}, \ldots, j_{\ell}}\rho^{\alpha}\partial_{x}^{i_{1}}\rho\cdots \partial_{x}^{i_{k}}\rho\partial_{x}^{j_{1}}B\cdots \partial_{x}^{j_{\ell}}B \\
&& + B_{0}^{2}\sum_{\substack{k, \ell \ge 0, \\
1 \le i_{1}, \ldots, i_{k} \le s, \\
0 \le j_{1}, \ldots, j_{\ell} \le s, \\
i_{1}+\cdots + i_{k} + j_{1} + \cdots + j_{\ell}=s+1, \\
\alpha+k+\frac{\gamma}{2}\ell =\frac{\gamma}{2}-1}}C_{\gamma, s, \alpha, i_{1}, \ldots, i_{k}, j_{1}, \ldots, j_{\ell}}\rho^{\alpha}\partial_{x}^{i_{1}}\rho\cdots \partial_{x}^{i_{k}}\rho \partial_{x}^{j_{1}}B\cdots \partial_{x}^{j_{\ell}}B, \yesnumber \label{20260110eq5}
\end{IEEEeqnarray*}
where each summation contains finitely many nonzero terms. Plugging \eqref{20260110eq4} and \eqref{20260110eq5} into \eqref{20260110eq3} gives
\begin{IEEEeqnarray*}{rCl}
\IEEEeqnarraymulticol{3}{l}{\partial_{t}\frac{1}{2}\int_{\TT} \left ( \left | \rho^{\frac{\gamma}{2}-1}\partial_{x}^{s}\rho\right |^{2} + \left | \partial_{x}^{s}B \right |^{2} \right ) \dd x} \\
&=& \int_{\TT} -\rho^{2\gamma-3}\big(\partial_{x}^{s+1}\rho\big)^{2} -2\rho^{\gamma-2}B\partial_{x}^{s+1}\rho\partial_{x}^{s+1}B -\rho^{-1}B^{2}\big(\partial_{x}^{s+1}B\big)^{2} - \rho^{-1}B_{0}^{2}\big(\partial_{x}^{s+1}B\big)^{2} \\
&& + \partial_{x}^{s+1}\rho \sum_{\substack{k, \ell \ge 0, \\
1 \le i_{1}, \cdots, i_{k} \le s, \\
0 \le j_{1}, \cdots, j_{\ell} \le s, \\
i_{1}+\cdots +i_{k}+j_{1}+\cdots +j_{\ell}=s+1, \\
\alpha+k+\frac{\gamma}{2}\ell = 2\gamma-2}}C_{\gamma, s, \alpha, i_{1}, \ldots, i_{k}, j_{1}, \ldots, j_{\ell}}\rho^{\alpha}\partial_{x}^{i_{1}}\rho\cdots \partial_{x}^{i_{k}}\rho\partial_{x}^{j_{1}}B\cdots \partial_{x}^{j_{\ell}}B \\
&& + \partial_{x}^{s+1}B \sum_{\substack{k, \ell \ge 0, \\
1 \le i_{1}, \cdots, i_{k} \le s, \\
0 \le j_{1}, \cdots, j_{\ell} \le s, \\
i_{1}+\cdots +i_{k}+j_{1}+\cdots +j_{\ell}=s+1, \\
\alpha+k+\frac{\gamma}{2}\ell = \frac{3}{2}\gamma-1}}C_{\gamma, s, \alpha, i_{1}, \ldots, i_{k}, j_{1}, \ldots, j_{\ell}}\rho^{\alpha}\partial_{x}^{i_{1}}\rho\cdots \partial_{x}^{i_{k}}\rho\partial_{x}^{j_{1}}B\cdots \partial_{x}^{j_{\ell}}B \\
&& + B_{0}^{2}\partial_{x}^{s+1}B \sum_{\substack{k, \ell \ge 0, \\
1 \le i_{1}, \ldots, i_{k} \le s, \\
0 \le j_{1}, \ldots, j_{\ell} \le s, \\
i_{1}+\cdots + i_{k} + j_{1} + \cdots + j_{\ell}=s+1, \\
\alpha+k+\frac{\gamma}{2}\ell =\frac{\gamma}{2}-1}}C_{\gamma, s, \alpha, i_{1}, \ldots, i_{k}, j_{1}, \ldots, j_{\ell}}\rho^{\alpha}\partial_{x}^{i_{1}}\rho\cdots \partial_{x}^{i_{k}}\rho \partial_{x}^{j_{1}}B\cdots \partial_{x}^{j_{\ell}}B\\
&& + \big (\partial_{x}^{s}\rho\big )^{2}\sum_{\substack{k, \ell \ge 0, \\
1 \le i_{1}, \cdots, i_{k} \le s, \\
0 \le j_{1}, \cdots, j_{\ell} \le s, \\
i_{1}+\cdots +i_{k}+j_{1}+\cdots +j_{\ell}=2, \\
\alpha+k+\frac{\gamma}{2}\ell = 2\gamma-3}}C_{\gamma, s, \alpha, i_{1}, \ldots, i_{k}, j_{1}, \ldots, j_{\ell}}\rho^{\alpha}\partial_{x}^{i_{1}}\rho\cdots \partial_{x}^{i_{k}}\rho\partial_{x}^{j_{1}}B\cdots \partial_{x}^{j_{\ell}}B \\
&=:& \int_{\TT} -\left (\rho^{\gamma-\frac{3}{2}}\partial_{x}^{s+1}\rho + \rho^{-\frac{1}{2}}B\partial_{x}^{s+1}B\right )^{2} - \rho^{-1}B_{0}^{2}\big(\partial_{x}^{s+1}B\big)^{2} \\
&& + \partial_{x}^{s+1}\rho [P_{\rho}] + \partial_{x}^{s+1}B[P_{B}] + B_{0}^{2}\partial_{x}^{s+1}B[P_{B_{0}}] + [P],
\end{IEEEeqnarray*}
where $P_{\rho}, P_{B}, P_{B_{0}}, P$ denote the four summations. Then by Cauchy-Schwarz inequality,
\begin{IEEEeqnarray*}{rCl}
\IEEEeqnarraymulticol{3}{l}{\partial_{t}\frac{1}{2}\int_{\TT} \left ( \left | \rho^{\frac{\gamma}{2}-1}\partial_{x}^{s}\rho\right |^{2} + \left | \partial_{x}^{s}B \right |^{2} \right ) \dd x} \\ 
&=& \int_{\TT} -\left (\rho^{\gamma-\frac{3}{2}}\partial_{x}^{s+1}\rho + \rho^{-\frac{1}{2}}B\partial_{x}^{s+1}B\right )^{2} - \rho^{-1}B_{0}^{2}\big(\partial_{x}^{s+1}B\big)^{2} \\
&& + \left (\rho^{\gamma-\frac{3}{2}}\partial_{x}^{s+1}\rho + \rho^{-\frac{1}{2}}B\partial_{x}^{s+1}B\right ) \left ( \rho^{-\gamma+\frac{3}{2}}P_{\rho}\right ) \\
&& + \rho^{-\frac{1}{2}}B_{0}\partial_{x}^{s+1}B \left ( B_{0}^{-1} \rho^{\frac{1}{2}}P_{B} - B_{0}^{-1}\rho^{-\gamma+\frac{3}{2}}BP_{\rho} + \rho^{\frac{1}{2}}B_{0}P_{B_{0}} \right ) + P \\
& \le & \int_{\TT} \frac{1}{4}\rho^{-2\gamma+3}P_{\rho}^{2} + B_{0}^{-2}\rho P_{B}^{2} + B_{0}^{-2}\rho^{-2\gamma+3}B^{2}P_{\rho}^{2} + B_{0}^{2}\rho P_{B_{0}}^{2} + P \\
& \le & \int_{\TT} \sum_{\substack{k, \ell \ge 0, \\
1 \le i_{1}, \cdots, i_{k} \le s, \\
0 \le j_{1}, \cdots, j_{\ell} \le s, \\
i_{1}+\cdots +i_{k}+j_{1}+\cdots +j_{\ell}=s+1, \\
\frac{1}{2}\alpha+k+\frac{\gamma}{2}\ell = \gamma-\frac{1}{2}}}C_{\gamma, s, \alpha, i_{1}, \ldots, i_{k}, j_{1}, \ldots, j_{\ell}}\rho^{\alpha}\big (\partial_{x}^{i_{1}}\rho\big )^{2}\cdots \big (\partial_{x}^{i_{k}}\rho\big )^{2}\big (\partial_{x}^{j_{1}}B\big )^{2}\cdots \big (\partial_{x}^{j_{\ell}}B\big )^{2} \\
&& + B_{0}^{-2}\sum_{\substack{k, \ell \ge 0, \\
1 \le i_{1}, \cdots, i_{k} \le s, \\
0 \le j_{1}, \cdots, j_{\ell} \le s, \\
i_{1}+\cdots +i_{k}+j_{1}+\cdots +j_{\ell}=s+1, \\
\frac{1}{2}\alpha+k+\frac{\gamma}{2}\ell = \frac{3}{2}\gamma-\frac{1}{2}}}C_{\gamma, s, \alpha, i_{1}, \ldots, i_{k}, j_{1}, \ldots, j_{\ell}}\rho^{\alpha}\big (\partial_{x}^{i_{1}}\rho\big )^{2}\cdots \big (\partial_{x}^{i_{k}}\rho\big )^{2}\big (\partial_{x}^{j_{1}}B\big )^{2}\cdots \big (\partial_{x}^{j_{\ell}}B\big )^{2} \\
&& + B_{0}^{2}\sum_{\substack{k, \ell \ge 0, \\
1 \le i_{1}, \cdots, i_{k} \le s, \\
0 \le j_{1}, \cdots, j_{\ell} \le s, \\
i_{1}+\cdots +i_{k}+j_{1}+\cdots +j_{\ell}=s+1, \\
\frac{1}{2}\alpha+k+\frac{\gamma}{2}\ell = \frac{1}{2}\gamma-\frac{1}{2}}}C_{\gamma, s, \alpha, i_{1}, \ldots, i_{k}, j_{1}, \ldots, j_{\ell}}\rho^{\alpha}\big (\partial_{x}^{i_{1}}\rho\big )^{2}\cdots \big (\partial_{x}^{i_{k}}\rho\big )^{2}\big (\partial_{x}^{j_{1}}B\big )^{2}\cdots \big (\partial_{x}^{j_{\ell}}B\big )^{2} \\
&& + \big (\partial_{x}^{s}\rho\big )^{2}\sum_{\substack{k, \ell \ge 0, \\
1 \le i_{1}, \cdots, i_{k} \le s, \\
0 \le j_{1}, \cdots, j_{\ell} \le s, \\
i_{1}+\cdots +i_{k}+j_{1}+\cdots +j_{\ell}=2, \\
\alpha+k+\frac{\gamma}{2}\ell = 2\gamma-3}}C_{\gamma, s, \alpha, i_{1}, \ldots, i_{k}, j_{1}, \ldots, j_{\ell}}\rho^{\alpha}\partial_{x}^{i_{1}}\rho\cdots \partial_{x}^{i_{k}}\rho\partial_{x}^{j_{1}}B\cdots \partial_{x}^{j_{\ell}}B. \yesnumber \label{20260110eq7}
\end{IEEEeqnarray*}

Now consider any term of the form \begin{eqnarray}\label{20260110eq6}
\int_{\TT}\rho^{\alpha}\big (\partial_{x}^{i_{1}}\rho\big )^{2}\cdots \big (\partial_{x}^{i_{k}}\rho\big )^{2}\big ( \partial_{x}^{j_{1}}B\big )^{2} \cdots \big (\partial_{x}^{j_{\ell}}B\big )^{2}
\end{eqnarray} satisfying \[
k, \ell \ge 0, \quad 1 \le i_{1}, \ldots, i_{k} \le s, \quad 0 \le j_{1}, \ldots, j_{\ell} \le s, \quad i_{1} + \cdots + i_{k} + j_{1} + \cdots + j_{\ell} = s+1.
\] Without loss of generality, we may assume that $\max\{i_{1}, j_{1}\}$ is the largest among $i_{1}, \ldots, i_{k}, j_{1}, \ldots, j_{\ell}$. Assume for the moment that $i_{1}$ is the largest. Note that $k + \ell \ge 2$, so we may choose the second largest number $m$ among $i_{1}, \ldots, i_{k}, j_{1}, \ldots, j_{\ell}$ (If the largest number appears more than once, then $m=i_{1}$.) Then the integral \eqref{20260110eq6} can be bounded by
\begin{IEEEeqnarray*}{rCl}
\IEEEeqnarraymulticol{3}{l}{\left |\int_{\TT} \rho^{\alpha}\big (\partial_{x}^{i_{1}}\rho\big )^{2}\cdots \big (\partial_{x}^{i_{k}}\rho\big )^{2}\big ( \partial_{x}^{j_{1}}B\big )^{2} \cdots \big (\partial_{x}^{j_{\ell}}B\big )^{2}\right |} \\
&\le& \|\partial_{x}^{i_{1}}\rho\|_{L^{2}}^{2} \|\rho^{\alpha}\|_{L^{\infty}}\|\partial_{x}^{i_{2}}\rho\|_{L^{\infty}}^{2}\cdots \|\partial_{x}^{i_{k}}\rho\|_{L^{\infty}}^{2}\|\partial_{x}^{j_{1}}B\|_{L^{\infty}}^{2}\cdots \|\partial_{x}^{j_{\ell}}B\|_{L^{\infty}}^{2} \\
&\le& \|\partial_{x}^{i_{1}}\rho\|_{L^{2}}^{2} \left ( \eta^{-1} + \|\rho\|_{L^{\infty}}\right )^{|\alpha|} \left ( \|\partial_{x}^{m}\rho\|_{L^{\infty}}+\|\partial_{x}^{m}B\|_{L^{\infty}}\right )^{2}\left ( \|\rho\|_{W^{m, \infty}} + \|B\|_{W^{m, \infty}}\right )^{2k+2\ell-4} \\
& \le & \|\partial_{x}^{i_{1}}\rho\|_{L^{2}}^{2} \left (\|\partial_{x}^{m}\rho\|_{L^{\infty}} + \|\partial_{x}^{m}B\|_{L^{\infty}}\right )^{2} \left ( \eta^{-1}+\|\rho\|_{W^{m, \infty}} + \|B\|_{W^{m, \infty}}\right )^{C_{\gamma, s}} \\
& \le & C_{ \gamma, s} \|\partial_{x}^{i_{1}}\rho\|_{L^{2}}^{2}\left (\|\partial_{x}^{m+1}\rho\|_{L^{2}} + \|\partial_{x}^{m+1}B\|_{L^{2}}\right )^{2}\left (\eta^{-1}+\|\rho\|_{H^{m+1}}+\|B\|_{H^{m+1}}\right )^{C_{\gamma, s}}.
\end{IEEEeqnarray*}
Note that in the last line we have used Sobolev inequalities \[
\|f\|_{W^{m, \infty}(\TT)} \le C \|f\|_{H^{m+1}(\TT)}, \quad \|\partial_{x}^{m}f\|_{L^{\infty}} \le C\|\partial_{x}^{m+1}f\|_{L^{2}}. \quad (m \ge 1)
\] If $j_{1}$ is the largest instead of $i_{1}$, then we would obtain the same inequality with $\|\partial_{x}^{i_{1}}\rho\|_{L^{2}}$ replaced by $\|\partial_{x}^{j_{1}}B\|_{L^{2}}$. Hence, if $M, m$ are the largest and second largest number among $i_{1}, \ldots, i_{k}, j_{1}, \ldots, j_{\ell}$, then we have \begin{IEEEeqnarray*}{rCl}
\IEEEeqnarraymulticol{3}{l}{\left |\int_{\TT} \rho^{\alpha}\big (\partial_{x}^{i_{1}}\rho\big )^{2}\cdots \big (\partial_{x}^{i_{k}}\rho\big )^{2}\big ( \partial_{x}^{j_{1}}B\big )^{2} \cdots \big (\partial_{x}^{j_{\ell}}B\big )^{2}\right |} \\
& \le& C_{\gamma, s} \left ( \|\partial_{x}^{M}\rho\|_{L^{2}}+\|\partial_{x}^{M}B\|_{L^{2}}\right )^{2}\left ( \|\partial_{x}^{m+1}\rho\|_{L^{2}} + \|\partial_{x}^{m+1}B\|_{L^{2}} \right )^{2}\left (\eta^{-1}+ \|\rho\|_{H^{m+1}}+\|B\|_{H^{m+1}}\right )^{C_{\gamma, s}}. \yesnumber \label{20260105eq1}
\end{IEEEeqnarray*} By assumption $i_{1}, \ldots, i_{k}, j_{1}, \ldots, j_{\ell} \le s$, we have $M \le s$. Also, $i_{1}+\cdots +i_{k}+j_{1}+\cdots +j_{\ell}=s+1>0$ implies that $M \ge 1$. Hence by Poincar\'e's inequality, we have \[
\|\partial_{x}^{M}\rho\|_{L^{2}} \le C_{s} \|\partial_{x}^{s}\rho\|_{L^{2}}.
\] Moreover, since \[
2m \le m+M \le i_{1}+\cdots +i_{k}+j_{1}+\cdots +j_{\ell} = s+1,
\] we have $m \le \lfloor \frac{s+1}{2}\rfloor$. Thus, we have \[
\|\partial_{x}^{m+1}\rho\|_{L^{2}} \le C_{s} \|\partial_{x}^{\lfloor\frac{s+3}{2}\rfloor}\rho\|_{L^{2}}.
\] Using these inequalities for \eqref{20260105eq1} gives
\begin{IEEEeqnarray*}{rCl}
\IEEEeqnarraymulticol{3}{l}{\left |\int_{\TT} \rho^{\alpha}\big (\partial_{x}^{i_{1}}\rho\big )^{2}\cdots \big (\partial_{x}^{i_{k}}\rho\big )^{2}\big ( \partial_{x}^{j_{1}}B\big )^{2} \cdots \big (\partial_{x}^{j_{\ell}}B\big )^{2}\right |} \\
&\le& C_{\gamma, s} \left ( \|\partial_{x}^{s}\rho\|_{L^{2}}+\|\partial_{x}^{s}B\|_{L^{2}}\right )^{2}\left ( \|\partial_{x}^{\lfloor\frac{s+3}{2}\rfloor}\rho\|_{L^{2}} + \|\partial_{x}^{\lfloor\frac{s+3}{2}\rfloor}B\|_{L^{2}}\right )^{2} \left ( \eta^{-1}+\|\rho\|_{H^{\lfloor \frac{s+3}{2}\rfloor}}+\|B\|_{H^{\lfloor \frac{s+3}{2}\rfloor}}\right )^{C_{\gamma, s}}. \yesnumber \label{20260110eq8}
\end{IEEEeqnarray*} Similar argument gives \begin{IEEEeqnarray*}{rCl}
\IEEEeqnarraymulticol{3}{l}{\left | \int_{\TT} \big(\partial_{x}^{s}\rho\big)^{2} \rho^{\alpha}\partial_{x}^{i_{1}}\rho\cdots \partial_{x}^{i_{k}}\rho\partial_{x}^{j_{1}}B\cdots \partial_{x}^{j_{\ell}}B \right |} \\
&\le& C_{\gamma, s} \|\partial_{x}^{s}\rho\|_{L^{2}}^{2}\left ( \|\partial_{x}^{\lfloor\frac{s+3}{2}\rfloor}\rho\|_{L^{2}} + \|\partial_{x}^{\lfloor\frac{s+3}{2}\rfloor}B\|_{L^{2}}\right ) \left (\eta^{-1}+ \|\rho\|_{H^{\lfloor \frac{s+3}{2}\rfloor}}+\|B\|_{H^{\lfloor \frac{s+3}{2}\rfloor}}\right )^{C_{\gamma, s}}. \yesnumber \label{20260110eq9}
\end{IEEEeqnarray*} (The assumption $s \ge 3$ is used here to guarantee $\lfloor\frac{s+3}{2}\rfloor \ge 3$.) Now the desired result \eqref{20260111eq7} follows by applying \eqref{20260110eq8} and \eqref{20260110eq9} to \eqref{20260110eq7}.
\end{proof}

\begin{cor}[A-priori $H^{s}$ estimate]
\label{20260110cor1}
Assume that $\gamma>1$ and $B_{0}\neq 0$. Let $\rho, B \in C^{\infty}([0, T]\times \TT)$ be smooth solutions to the system \eqref{20260110eq2}. Assume that $\rho \ge \eta$ on $[0, T]\times \TT$ for some positive constant $\eta>0$. Then for any integer $s \ge 3$, one has
{\small\begin{IEEEeqnarray*}{rCl}
\IEEEeqnarraymulticol{3}{l}{\left (\|\partial_{x}^{s}\rho(t)\|_{L^{2}}^{2} + \|\partial_{x}^{s}B(t)\|_{L^{2}}^{2}\right )} \\
&\le& \left (\|\partial_{x}^{s}\rho_{t=0}\|_{L^{2}}^{2} + \|\partial_{x}^{s}B_{t=0}\|_{L^{2}}^{2}\right ) \left ( \eta^{-1} + \|\rho_{t=0}\|_{H^{\lfloor\frac{s+3}{2}\rfloor}} + \|\rho(t)\|_{H^{\lfloor\frac{s+3}{2}\rfloor}}\right )^{C_{\gamma}} \\
&& \exp \left ( C_{\gamma, B_{0}, s} \int_{0}^{t}  \left ( \|\partial_{x}^{\lfloor\frac{s+3}{2}\rfloor}\rho(\tau)\|_{L^{2}} + \|\partial_{x}^{\lfloor\frac{s+3}{2}\rfloor}B(\tau)\|_{L^{2}}\right ) \left ( \eta^{-1}+\|\rho(\tau)\|_{H^{\lfloor \frac{s+3}{2}\rfloor}}+\|B(\tau)\|_{H^{\lfloor \frac{s+3}{2}\rfloor}}\right )^{C_{\gamma, s}}\dd\tau \right ).
\end{IEEEeqnarray*}}
\end{cor}

\begin{proof}[Proof of Corollary \ref{20260110cor1}]
Note that \begin{eqnarray}\label{20260110eq10}
\|\partial_{x}^{s}\rho\|_{L^{2}} \le \|\rho^{\frac{\gamma}{2}-1}\partial_{x}^{s}\rho\|_{L^{2}}\left ( \|\rho\|_{L^{\infty}} + \|\rho^{-1}\|_{L^{\infty}}\right )^{\left |\frac{\gamma}{2}-1\right |} \le \|\rho^{\frac{\gamma}{2}-1}\partial_{x}^{s}\rho\|_{L^{2}}\left ( \eta^{-1}+\|\rho\|_{H^{\lfloor\frac{s+3}{2}\rfloor}}\right )^{C_{\gamma}}.
\end{eqnarray} Together with Lemma \ref{20260110lem1}, we obtain {\small \begin{IEEEeqnarray*}{rCl}
\IEEEeqnarraymulticol{3}{l}{ \partial_{t} \int_{\TT} \left ( \left | \rho^{\frac{\gamma}{2}-1}\partial_{x}^{s}\rho\right |^{2} +\left |\partial_{x}^{s}B\right |^{2}\right ) } \\
& \le & C_{\gamma, B_{0}, s} \left ( \|\partial_{x}^{\lfloor\frac{s+3}{2}\rfloor}\rho\|_{L^{2}} + \|\partial_{x}^{\lfloor\frac{s+3}{2}\rfloor}B\|_{L^{2}}\right )\left ( \eta^{-1}+\|\rho\|_{H^{\lfloor \frac{s+3}{2}\rfloor}}+\|B\|_{H^{\lfloor \frac{s+3}{2}\rfloor}}\right )^{C_{\gamma, s}} \int_{\TT} \left ( \left | \rho^{\frac{\gamma}{2}-1}\partial_{x}^{s}\rho\right |^{2} +\left |\partial_{x}^{s}B\right |^{2}\right ). \yesnumber \label{20260111eq1}
\end{IEEEeqnarray*}} Hence by Gr\"onwall's inequality, we conclude that
{\small \begin{IEEEeqnarray*}{rCl}
\IEEEeqnarraymulticol{3}{l}{\left (\|\rho(t)^{\frac{\gamma}{2}-1}\partial_{x}^{s}\rho(t)\|_{L^{2}}^{2} + \|\partial_{x}^{s}B(t)\|_{L^{2}}^{2}\right )} \\
&\le& \left (\|\rho_{t=0}^{\frac{\gamma}{2}-1}\partial_{x}^{s}\rho_{t=0}\|_{L^{2}}^{2} + \|\partial_{x}^{s}B_{t=0}\|_{L^{2}}^{2}\right ) \\
&& \exp \left ( C_{\gamma, B_{0}, s} \int_{0}^{t}  \left ( \|\partial_{x}^{\lfloor\frac{s+3}{2}\rfloor}\rho(\tau)\|_{L^{2}} + \|\partial_{x}^{\lfloor\frac{s+3}{2}\rfloor}B(\tau)\|_{L^{2}}\right ) \left ( \eta^{-1}+\|\rho(\tau)\|_{H^{\lfloor \frac{s+3}{2}\rfloor}}+\|B(\tau)\|_{H^{\lfloor \frac{s+3}{2}\rfloor}}\right )^{C_{\gamma, s}}\dd\tau \right ).
\end{IEEEeqnarray*}}
Using \eqref{20260110eq10} once again finishes the proof.
\end{proof}

\subsection{Boundedness of $\rho, \rho^{-1}$, and $B$}\label{subsec2.2}

In this section, assume $1< \gamma<2$, and $B_{0} \neq 0$. 

\subsubsection{Definition of new variables $W$ and $Z$} \label{subsubsec2.2.1}

Define \begin{IEEEeqnarray*}{rCl}
f(\rho, B)&:=& \begin{cases}
\frac{\left (B^{2}+B_{0}^{2}-\rho^{\gamma}\right )+\sqrt{(B^{2}+B_{0}^{2}-\rho^{\gamma})^{2}+4B^{2}\rho^{\gamma}}}{B^{2}} & B \neq 0, \\
\infty & B=0, \rho^{\gamma} \le B_{0}^{2}, \\
\frac{2\rho^{\gamma}}{\rho^{\gamma}-B_{0}^{2}} & B=0, \rho^{\gamma} > B_{0}^{2}
\end{cases} \\
&=& \begin{cases}
\infty & B=0, \rho^{\gamma}\le B_{0}^{2}, \\
\frac{4\rho^{\gamma}}{-\big(B^{2}+B_{0}^{2}-\rho^{\gamma}\big)+\sqrt{\big(B^{2}+B_{0}^{2}-\rho^{\gamma}\big)^{2}+4B^{2}\rho^{\gamma}}}& \text{otherwise.}
\end{cases}
\end{IEEEeqnarray*} and consider \[
w = w(\rho, B) := \begin{cases}
(B^{2}+B_{0}^{2}) f^{\frac{2}{2-\gamma}} - \int_{f}^{4}\frac{2B_{0}^{2}}{2-\gamma}\frac{1-\frac{1}{4}s}{\left(1-\frac{1}{2}s\right )^{2}}s^{\frac{2}{2-\gamma}}\dd s & f < \infty, \\
\infty & f = \infty.
\end{cases}
\] 

Similarly, we define \begin{IEEEeqnarray*}{rCl}
g(\rho, B)&:=& \begin{cases}
\frac{-\left (B^{2}+B_{0}^{2}-\rho^{\gamma}\right )+\sqrt{(B^{2}+B_{0}^{2}-\rho^{\gamma})^{2}+4B^{2}\rho^{\gamma}}}{B^{2}} & B \neq 0, \\
\frac{2\rho^{\gamma}}{B_{0}^{2}-\rho^{\gamma}} & B=0, \rho^{\gamma} < B_{0}^{2}, \\
\infty & B=0, \rho^{\gamma} \ge B_{0}^{2}
\end{cases} \\
&=& \begin{cases}
\infty & B=0, \rho^{\gamma}\ge B_{0}^{2}, \\
\frac{4\rho^{\gamma}}{\big(B^{2}+B_{0}^{2}-\rho^{\gamma}\big)+\sqrt{\big(B^{2}+B_{0}^{2}-\rho^{\gamma}\big)^{2}+4B^{2}\rho^{\gamma}}}& \text{otherwise.}
\end{cases}
\end{IEEEeqnarray*} and \[
z = z(\rho, B) = \begin{cases}
(B^{2}+B_{0}^{2}) g^{\frac{2}{2-\gamma}} - \int_{0}^{g}\frac{2B_{0}^{2}}{2-\gamma}\frac{1+\frac{1}{4}s}{\left(1+\frac{1}{2}s\right )^{2}}s^{\frac{2}{2-\gamma}}\dd s & g < \infty, \\
\infty & g = \infty.
\end{cases}
\] 

It is easy to see that $f>2$ and $g>0$. So, the integral appearing in the definition of $w$ and $z$, \[
\int_{f}^{4}\frac{2B_{0}^{2}}{2-\gamma}\frac{1-\frac{1}{4}s}{\left (1-\frac{1}{2}s\right )^{2}}s^{\frac{2}{2-\gamma}}\dd s \quad \text{and }\int_{0}^{g}\frac{2B_{0}^{2}}{2-\gamma}\frac{1+\frac{1}{4}s}{\left (1+\frac{1}{2}s\right )^{2}}s^{\frac{2}{2-\gamma}}\dd s,
\] are well-defined for $f<\infty$ and $g<\infty$, respectively. Define also
\begin{IEEEeqnarray*}{rCl}
\alpha(\rho, B)&:=& \frac{B_{0}^{2}}{\rho}\frac{f}{f-2} = \frac{B^{2}+B_{0}^{2}}{\rho}+\frac{2\rho^{\gamma-1}}{f} = \rho^{\gamma-1}+\frac{B^{2}f}{2\rho} \quad (\text{last equality holds if }f<\infty), \\
\beta(\rho, B)&:=& \frac{B_{0}^{2}}{\rho}\frac{g}{g+2} = \frac{B^{2}+B_{0}^{2}}{\rho}-\frac{2\rho^{\gamma-1}}{g} = \rho^{\gamma-1}-\frac{B^{2}g}{2\rho} \quad (\text{last equality holds if }g<\infty).
\end{IEEEeqnarray*}
Since $f>2$ and $g>0$, $\alpha$ and $\beta$ are positive.

Finally define \[
W(\rho, B) := \exp(-w(\rho, B)), \quad Z(\rho, B):=\exp(-z(\rho, B)).
\] One can easily check that $z >0$, so that $Z < 1$.

\subsubsection{Equations satisfied by $W$ and $Z$} \label{subsubsec2.2.2}

From the definition of $W$ and $Z$, one can prove that the map $(\rho, B) \mapsto (W(\rho, B), Z(\rho, B))$ is $C^{2}$. The verification of this fact is deferred to Appendix \ref{appendixA}, where we also provide the computation of their derivatives. (See equations \eqref{20260117eq11} - \eqref{20260117eq28}) Using these computed derivatives, we can show that for $(\rho, B)$ with $f(\rho, B)<\infty$, it holds that {\small
\begin{IEEEeqnarray*}{rCl}
\IEEEeqnarraymulticol{3}{l}{\partial_{t}W(\rho, B) - \alpha(\rho, B)\partial_{x}^{2}W(\rho, B)} \\
&=& \partial_{\rho}W\partial_{x}^{2}\left ( \frac{1}{\gamma}\rho^{\gamma}+\frac{1}{2}B^{2}\right ) + \partial_{B}W \partial_{x}\left ( \frac{B}{\rho}(\rho^{\gamma-1}\partial_{x}\rho+B\partial_{x}B)+\frac{B_{0}^{2}}{\rho}\partial_{x}B\right ) \\
&& - \alpha \left ( \partial_{\rho\rho}W\big (\partial_{x}\rho\big)^{2} + 2\partial_{\rho B}W \partial_{x}\rho\partial_{x}B + \partial_{BB}W\big(\partial_{x}B\big)^{2}+ \partial_{\rho}W\partial_{x}^{2}\rho+\partial_{B}W\partial_{x}^{2}B\right ) \\
&=& \partial_{\rho}W\left ( \rho^{\gamma-1}\partial_{x}^{2}\rho+(\gamma-1)\rho^{\gamma-2}\big(\partial_{x}\rho\big)^{2}+B\partial_{x}^{2}B+\big(\partial_{x}B\big)^{2}\right ) \\
&& + \partial_{B}W \bigg (  \rho^{\gamma-2}B\partial_{x}^{2}\rho + \frac{B^{2}+B_{0}^{2}}{\rho}\partial_{x}^{2}B + \rho^{\gamma-2}\partial_{x}\rho\partial_{x}B + (\gamma-2)\rho^{\gamma-3}B\big(\partial_{x}\rho\big)^{2}-\frac{B^{2}+B_{0}^{2}}{\rho^{2}}\partial_{x}\rho\partial_{x}B + \frac{2B}{\rho}\big(\partial_{x}B\big)^{2} \bigg ) \\
&& - \alpha \left ( \partial_{\rho\rho}W\big (\partial_{x}\rho\big)^{2} + 2\partial_{\rho B}W \partial_{x}\rho\partial_{x}B + \partial_{BB}W\big(\partial_{x}B\big)^{2}+ \partial_{\rho}W\partial_{x}^{2}\rho+\partial_{B}W\partial_{x}^{2}B\right ) \\
&=& \partial_{x}^{2}\rho\underbrace{\left ( \rho^{\gamma-1}\partial_{\rho}W + \rho^{\gamma-2}B\partial_{B}W - \alpha \partial_{\rho}W\right )}_{=0} \\
&& + \partial_{x}^{2}B\underbrace{\left ( B\partial_{\rho}W + \frac{B^{2}+B_{0}^{2}}{\rho}\partial_{B}W - \alpha \partial_{B}W\right )}_{=0} \\
&& + \big(\partial_{x}\rho\big)^{2}\left ( (\gamma-1)\rho^{\gamma-2}\partial_{\rho}W + (\gamma-2)\rho^{\gamma-3}B\partial_{B}W - \alpha \partial_{\rho\rho}W\right ) \\
&& + \partial_{x}\rho\partial_{x}B\left ( \rho^{\gamma-2}\partial_{B}W-\frac{B^{2}+B_{0}^{2}}{\rho^{2}}\partial_{B}W - 2\alpha \partial_{\rho B}W\right ) \\
&& + \big(\partial_{x}B\big)^{2}\left ( \partial_{\rho}W + \frac{2B}{\rho}\partial_{B}W - \alpha \partial_{BB}W\right ) \\
&=& W \big(\partial_{x}\rho\big)^{2}\bigg ( (\gamma-1)\frac{4\gamma}{2-\gamma}\rho^{2\gamma-3}f^{\frac{\gamma}{2-\gamma}} - 2\gamma \rho^{\gamma-3}B^{2}f^{\frac{2}{2-\gamma}} - \alpha\frac{4\gamma}{2-\gamma}(\gamma-1)\rho^{\gamma-2}f^{\frac{\gamma}{2-\gamma}}-\alpha \frac{4\gamma^{2}}{(2-\gamma)^{2}}\rho^{\gamma-1}f^{\frac{2\gamma-2}{2-\gamma}}\partial_{\rho}f\bigg ) \\
&& + W \big(\partial_{x}\rho\big)\big(\partial_{x}B\big) \left ( \frac{2\gamma}{2-\gamma}\rho^{\gamma-2}Bf^{\frac{2}{2-\gamma}} - \frac{2\gamma}{2-\gamma}\frac{B^{2}+B_{0}^{2}}{\rho^{2}}B f^{\frac{2}{2-\gamma}} - \alpha \frac{8\gamma^{2}}{(2-\gamma)^{2}} \rho^{\gamma-1}f^{\frac{2\gamma-2}{2-\gamma}}\partial_{B}f\right ) \\
&& + W \big(\partial_{x}B\big)^{2} \left ( \frac{4\gamma}{2-\gamma} \rho^{\gamma-1}f^{\frac{\gamma}{2-\gamma}} + \frac{4\gamma}{2-\gamma}\frac{B^{2}}{\rho}f^{\frac{2}{2-\gamma}} - \alpha \frac{2\gamma}{2-\gamma}f^{\frac{2}{2-\gamma}} - \alpha \frac{4\gamma}{(2-\gamma)^{2}}Bf^{\frac{\gamma}{2-\gamma}}\partial_{B}f\right ) \\
&& - \alpha W\left ( \frac{4\gamma}{2-\gamma}\rho^{\gamma-1}f^{\frac{\gamma}{2-\gamma}}\partial_{x}\rho + \frac{2\gamma}{2-\gamma}Bf^{\frac{2}{2-\gamma}}\partial_{x}B\right )^{2} \\
&=& W \big(\partial_{x}\rho\big)^{2}\left ( - \frac{2\gamma}{2-\gamma} \rho^{\gamma-3}B^{2}f^{\frac{2}{2-\gamma}} -\alpha \frac{4\gamma^{2}}{(2-\gamma)^{2}}\rho^{\gamma-1}f^{\frac{2\gamma-2}{2-\gamma}}\partial_{\rho}f\right ) \\
&& + W \big(\partial_{x}\rho\big)\big(\partial_{x}B\big) \left ( \frac{2\gamma}{2-\gamma}\rho^{\gamma-2}Bf^{\frac{2}{2-\gamma}} - \frac{2\gamma}{2-\gamma}\frac{B^{2}+B_{0}^{2}}{\rho^{2}}B f^{\frac{2}{2-\gamma}} - \alpha \frac{8\gamma}{(2-\gamma)^{2}} B f^{\frac{\gamma}{2-\gamma}}\partial_{\rho}f\right ) \\
&& + W \big(\partial_{x}B\big)^{2} \left (  \frac{2\gamma}{2-\gamma}\frac{B^{2}-B_{0}^{2}}{\rho}f^{\frac{2}{2-\gamma}} - \alpha \frac{4}{(2-\gamma)^{2}}\frac{B^{2}}{\rho^{\gamma-1}}f^{\frac{2}{2-\gamma}}\partial_{\rho}f\right ) \\
&& - \alpha W\left ( \frac{4\gamma}{2-\gamma}\rho^{\gamma-1}f^{\frac{\gamma}{2-\gamma}}\partial_{x}\rho + \frac{2\gamma}{2-\gamma}Bf^{\frac{2}{2-\gamma}}\partial_{x}B\right )^{2} \\
&=& W\underbrace{\left ( -\frac{4\gamma}{2-\gamma}\rho^{\gamma-1}f^{\frac{\gamma}{2-\gamma}}\partial_{x}\rho - \frac{2\gamma}{2-\gamma}Bf^{\frac{2}{2-\gamma}}\partial_{x}B\right )}_{=\partial_{x}w} \\
&& \bigg [ \left (\frac{1}{2}\frac{B^{2}}{\rho^{2}}f + \alpha \frac{\gamma}{2-\gamma} f^{-1}\partial_{\rho}f\right ) \partial_{x}\rho + \left ( -\frac{1}{2} \frac{B}{\rho}f + \frac{1}{2} \frac{B^{2}+B_{0}^{2}}{\rho^{\gamma+1}}Bf -\frac{1}{4}\frac{B^{3}}{\rho^{\gamma+1}}f^{2}- \alpha \frac{\gamma-4}{2(2-\gamma)} \frac{B}{\rho^{\gamma-1}}\partial_{\rho}f\right )\partial_{x}B\bigg ] \\
&& + W\big(\partial_{x}B\big)^{2}\bigg ( \frac{\gamma}{2(2-\gamma)}\frac{1}{\rho^{\gamma+1}}f^{\frac{2}{2-\gamma}}\underbrace{\left ( 4\rho^{\gamma}(B^{2}-B_{0}^{2})-2\rho^{\gamma}B^{2}f+2(B^{2}+B_{0}^{2})B^{2}f-B^{4}f^{2}\right )}_{=-4B_{0}^{2}\rho^{\gamma}} - \alpha \frac{B^{2}}{\rho^{\gamma-1}}f^{\frac{2}{2-\gamma}}\partial_{\rho}f\bigg ) \\
&& - \alpha W\left ( \frac{4\gamma}{2-\gamma}\rho^{\gamma-1}f^{\frac{\gamma}{2-\gamma}}\partial_{x}\rho + \frac{2\gamma}{2-\gamma}Bf^{\frac{2}{2-\gamma}}\partial_{x}B\right )^{2} \\
&=& -\partial_{x}W\bigg [ \left (\frac{1}{2}\frac{B^{2}}{\rho^{2}}f + \alpha \frac{\gamma}{2-\gamma} f^{-1}\partial_{\rho}f\right ) \partial_{x}\rho + \left ( -\frac{1}{2} \frac{B}{\rho}f + \frac{1}{2} \frac{B^{2}+B_{0}^{2}}{\rho^{\gamma+1}}Bf -\frac{1}{4}\frac{B^{3}}{\rho^{\gamma+1}}f^{2}- \alpha \frac{\gamma-4}{2(2-\gamma)} \frac{B}{\rho^{\gamma-1}}\partial_{\rho}f\right )\partial_{x}B\bigg ] \\
&& - W\big(\partial_{x}B\big)^{2}\left ( \frac{2\gamma}{2-\gamma}\frac{1}{\rho}B_{0}^{2}f^{\frac{2}{2-\gamma}}+\alpha \frac{B^{2}}{\rho^{\gamma-1}}f^{\frac{2}{2-\gamma}}\partial_{\rho}f\right )- \alpha W\left ( \frac{4\gamma}{2-\gamma}\rho^{\gamma-1}f^{\frac{\gamma}{2-\gamma}}\partial_{x}\rho + \frac{2\gamma}{2-\gamma}Bf^{\frac{2}{2-\gamma}}\partial_{x}B\right )^{2} \\
& \le & -\partial_{x}W\bigg [ \left (\frac{1}{2}\frac{B^{2}}{\rho^{2}}f + \alpha \frac{\gamma}{2-\gamma} f^{-1}\partial_{\rho}f\right ) \partial_{x}\rho + \left ( -\frac{1}{2} \frac{B}{\rho}f + \frac{1}{2} \frac{B^{2}+B_{0}^{2}}{\rho^{\gamma+1}}Bf -\frac{1}{4}\frac{B^{3}}{\rho^{\gamma+1}}f^{2}- \alpha \frac{\gamma-4}{2(2-\gamma)} \frac{B}{\rho^{\gamma-1}}\partial_{\rho}f\right )\partial_{x}B\bigg ],
\end{IEEEeqnarray*}}
where the last inequality follows from the fact that
{\small
\begin{IEEEeqnarray*}{rCl}
\IEEEeqnarraymulticol{3}{l}{\frac{2\gamma}{2-\gamma}\frac{1}{\rho}B_{0}^{2}+\alpha \frac{B^{2}}{\rho^{\gamma-1}}\partial_{\rho}f \ge 0} \\
&\Longleftrightarrow& \frac{2\gamma}{2-\gamma}\frac{1}{\rho}B_{0}^{2}\ge \frac{B_{0}^{2}}{\rho}\frac{f}{f-2} \frac{B^{2}}{\rho^{\gamma-1}}\frac{4 \gamma \rho^{\gamma-1}B_{0}^{2}}{\left (B^{2}+B_{0}^{2}-\rho^{\gamma}\right )^{2}+4B^{2}\rho^{\gamma}}\frac{1}{1+ \frac{B^{2}-B_{0}^{2}+\rho^{\gamma}}{\sqrt{\left (B^{2}+B_{0}^{2}-\rho^{\gamma}\right )^{2}+4B^{2}\rho^{\gamma}}}} \\
&\Longleftrightarrow& \frac{1}{2-\gamma}\ge \frac{f}{f-2} B^{2}\frac{2 B_{0}^{2}}{\left (B^{2}+B_{0}^{2}-\rho^{\gamma}\right )^{2}+4B^{2}\rho^{\gamma}}\frac{1}{1+ \frac{B^{2}-B_{0}^{2}+\rho^{\gamma}}{\sqrt{\left (B^{2}+B_{0}^{2}-\rho^{\gamma}\right )^{2}+4B^{2}\rho^{\gamma}}}} \\
&\Longleftarrow& \left (1-\frac{2}{f}\right )\left ( \big(B^{2}+B_{0}^{2}-\rho^{\gamma}\big)^{2}+4B^{2}\rho^{\gamma}+\big (B^{2}-B_{0}^{2}+\rho^{\gamma}\big )\sqrt{\big(B^{2}+B_{0}^{2}-\rho^{\gamma}\big)^{2}+4B^{2}\rho^{\gamma}}\right ) \ge 2B_{0}^{2}B^{2} \\
&\Longleftrightarrow& \big(B^{2}+B_{0}^{2}-\rho^{\gamma}\big)^{2}+4B^{2}\rho^{\gamma} - (B^{2}+B_{0}^{2}-\rho^{\gamma})\sqrt{\big(B^{2}+B_{0}^{2}-\rho^{\gamma}\big)^{2}+4B^{2}\rho^{\gamma}} \ge 2B^{2}\rho^{\gamma} \\
& \Longleftarrow & \big(B^{2}+B_{0}^{2}-\rho^{\gamma}\big)^{4} +4B^{2}\rho^{\gamma}\big(B^{2}+B_{0}^{2}-\rho^{\gamma}\big)^{2}+4B^{4}\rho^{2\gamma} \ge \big(B^{2}+B_{0}^{2}-\rho^{\gamma}\big)^{4}+4B^{2}\rho^{\gamma}\big(B^{2}+B_{0}^{2}-\rho^{\gamma}\big)^{2}.
\end{IEEEeqnarray*}}

Similar computation gives that for $(\rho, B)$ with $g(\rho, B)<\infty$, it holds that {\small
\begin{IEEEeqnarray*}{rCl}
\IEEEeqnarraymulticol{3}{l}{\partial_{t}Z(\rho, B) - \beta(\rho, B)\partial_{x}^{2}Z(\rho, B)} \\
&=& \partial_{\rho}Z\partial_{x}^{2}\left ( \frac{1}{\gamma}\rho^{\gamma}+\frac{1}{2}B^{2}\right ) + \partial_{B}Z \partial_{x}\left ( \frac{B}{\rho}(\rho^{\gamma-1}\partial_{x}\rho+B\partial_{x}B)+\frac{B_{0}^{2}}{\rho}\partial_{x}B\right ) \\
&& - \beta \left ( \partial_{\rho\rho}Z\big (\partial_{x}\rho\big)^{2} + 2\partial_{\rho B}Z \partial_{x}\rho\partial_{x}B + \partial_{BB}Z\big(\partial_{x}B\big)^{2}+ \partial_{\rho}Z\partial_{x}^{2}\rho+\partial_{B}Z\partial_{x}^{2}B\right ) \\
&=& \partial_{\rho}Z\left ( \rho^{\gamma-1}\partial_{x}^{2}\rho+(\gamma-1)\rho^{\gamma-2}\big(\partial_{x}\rho\big)^{2}+B\partial_{x}^{2}B+\big(\partial_{x}B\big)^{2}\right ) \\
&& + \partial_{B}Z \bigg (  \rho^{\gamma-2}B\partial_{x}^{2}\rho + \frac{B^{2}+B_{0}^{2}}{\rho}\partial_{x}^{2}B + \rho^{\gamma-2}\partial_{x}\rho\partial_{x}B + (\gamma-2)\rho^{\gamma-3}B\big(\partial_{x}\rho\big)^{2}-\frac{B^{2}+B_{0}^{2}}{\rho^{2}}\partial_{x}\rho\partial_{x}B + \frac{2B}{\rho}\big(\partial_{x}B\big)^{2} \bigg ) \\
&& - \beta \left ( \partial_{\rho\rho}Z\big (\partial_{x}\rho\big)^{2} + 2\partial_{\rho B}Z \partial_{x}\rho\partial_{x}B + \partial_{BB}Z\big(\partial_{x}B\big)^{2}+ \partial_{\rho}Z\partial_{x}^{2}\rho+\partial_{B}Z\partial_{x}^{2}B\right ) \\
&=& \partial_{x}^{2}\rho\underbrace{\left ( \rho^{\gamma-1}\partial_{\rho}Z + \rho^{\gamma-2}B\partial_{B}Z - \beta \partial_{\rho}Z\right )}_{=0} \\
&& + \partial_{x}^{2}B\underbrace{\left ( B\partial_{\rho}Z + \frac{B^{2}+B_{0}^{2}}{\rho}\partial_{B}Z - \beta \partial_{B}Z\right )}_{=0} \\
&& + \big(\partial_{x}\rho\big)^{2}\left ( (\gamma-1)\rho^{\gamma-2}\partial_{\rho}Z + (\gamma-2)\rho^{\gamma-3}B\partial_{B}Z - \beta \partial_{\rho\rho}Z\right ) \\
&& + \partial_{x}\rho\partial_{x}B\left ( \rho^{\gamma-2}\partial_{B}Z-\frac{B^{2}+B_{0}^{2}}{\rho^{2}}\partial_{B}Z - 2\beta \partial_{\rho B}Z\right ) \\
&& + \big(\partial_{x}B\big)^{2}\left ( \partial_{\rho}Z + \frac{2B}{\rho}\partial_{B}Z - \beta \partial_{BB}Z\right ) \\
&=& Z\big(\partial_{x}\rho\big)^{2}\bigg ( -(\gamma-1)\frac{4\gamma}{2-\gamma}\rho^{2\gamma-3}g^{\frac{\gamma}{2-\gamma}} - 2\gamma\rho^{\gamma-3}B^{2} g^{\frac{2}{2-\gamma}} + \beta \frac{4\gamma}{2-\gamma}(\gamma-1)\rho^{\gamma-2}g^{\frac{\gamma}{2-\gamma}}  + \beta \frac{4\gamma^{2}}{(2-\gamma)^{2}} \rho^{\gamma-1}g^{\frac{2\gamma-2}{2-\gamma}}\partial_{\rho}g \bigg ) \\
&& + Z\big(\partial_{x}\rho\big)\big(\partial_{x}B\big)\left (\frac{2\gamma}{2-\gamma} \rho^{\gamma-2}B g^{\frac{2}{2-\gamma}}-\frac{2\gamma}{2-\gamma}\frac{B^{2}+B_{0}^{2}}{\rho^{2}}B g^{\frac{2}{2-\gamma}} \right ) \\
&& + Z\big(\partial_{x}B\big)^{2}\left ( - \frac{4\gamma}{2-\gamma}\rho^{\gamma-1} g^{\frac{\gamma}{2-\gamma}} + \frac{4\gamma}{2-\gamma}\frac{B^{2}}{\rho} g^{\frac{2}{2-\gamma}} -\beta \frac{2\gamma}{2-\gamma}g^{\frac{2}{2-\gamma}} -\beta \frac{4\gamma}{(2-\gamma)^{2}} B g^{\frac{\gamma}{2-\gamma}}\partial_{B}g \right ) \\
&& - \beta Z \left ( \frac{4\gamma}{2-\gamma}\rho^{\gamma-1}g^{\frac{\gamma}{2-\gamma}}\partial_{x}\rho - \frac{2\gamma}{2-\gamma}B g^{\frac{2}{2-\gamma}}\partial_{x}B\right )^{2}\\
&=& Z\big(\partial_{x}\rho\big)^{2}\left ( - \frac{2\gamma}{2-\gamma}\rho^{\gamma-3}B^{2} g^{\frac{2}{2-\gamma}}  + \beta \frac{4\gamma^{2}}{(2-\gamma)^{2}} \rho^{\gamma-1}g^{\frac{2\gamma-2}{2-\gamma}}\partial_{\rho}g \right ) \\
&& + Z\big(\partial_{x}\rho\big)\big(\partial_{x}B\big)\left (\frac{2\gamma}{2-\gamma} \rho^{\gamma-2}B g^{\frac{2}{2-\gamma}}-\frac{2\gamma}{2-\gamma}\frac{B^{2}+B_{0}^{2}}{\rho^{2}}B g^{\frac{2}{2-\gamma}} -\beta \frac{8\gamma}{(2-\gamma)^{2}} B g^{\frac{\gamma}{2-\gamma}} \partial_{\rho}g\right ) \\
&& + Z\big(\partial_{x}B\big)^{2}\left ( \frac{2\gamma}{2-\gamma} \frac{B^{2}-B_{0}^{2}}{\rho}g^{\frac{2}{2-\gamma}} +\beta \frac{4}{(2-\gamma)^{2}} \frac{B^{2}}{\rho^{\gamma-1}} g^{\frac{2}{2-\gamma}}\partial_{\rho}g \right ) \\
&& - \beta Z \left ( \frac{4\gamma}{2-\gamma}\rho^{\gamma-1}g^{\frac{\gamma}{2-\gamma}}\partial_{x}\rho - \frac{2\gamma}{2-\gamma}B g^{\frac{2}{2-\gamma}}\partial_{x}B\right )^{2}\\
&=& Z\underbrace{\left (\frac{4\gamma}{2-\gamma}g^{\frac{\gamma}{2-\gamma}}\rho^{\gamma-1}\partial_{x}\rho-\frac{2\gamma}{2-\gamma}g^{\frac{2}{2-\gamma}}B\partial_{x}B\right )}_{=\partial_{x}z} \\
&& \left ( \left ( -\frac{1}{2}\frac{B^{2}}{\rho^{2}}g + \beta \frac{\gamma}{2-\gamma}g^{-1}\partial_{\rho}g\right )\partial_{x}\rho + \left ( \frac{1}{2} \frac{B}{\rho} g - \frac{1}{2} \frac{B^{2}+B_{0}^{2}}{\rho^{\gamma+1}}Bg - \frac{1}{4} \frac{B^{3}}{\rho^{\gamma+1}}g^{2} +\beta\frac{1}{2}\frac{\gamma-4}{2-\gamma}\frac{B}{\rho^{\gamma-1}}\partial_{\rho}g\right )\partial_{x}B\right )\\
&& + Z \big(\partial_{x}B\big)^{2} \left ( \frac{\gamma}{2(2-\gamma)}\frac{1}{\rho^{\gamma+1}}g^{\frac{2}{2-\gamma}}\underbrace{\left (4 \rho^{\gamma}(B^{2}-B_{0}^{2})+ 2 \rho^{\gamma}B^{2} g - 2(B^{2}+B_{0}^{2}) B^{2} g -B^{4} g^{2}\right )}_{=-4B_{0}^{2}\rho^{\gamma}} +\beta \frac{B^{2}}{\rho^{\gamma-1}} g^{\frac{2}{2-\gamma}}\partial_{\rho}g  \right ) \\
&& - \beta Z \left ( \frac{4\gamma}{2-\gamma}\rho^{\gamma-1}g^{\frac{\gamma}{2-\gamma}}\partial_{x}\rho - \frac{2\gamma}{2-\gamma}B g^{\frac{2}{2-\gamma}}\partial_{x}B\right )^{2}\\
&=& -\partial_{x}Z \left ( \left ( -\frac{1}{2}\frac{B^{2}}{\rho^{2}}g + \beta \frac{\gamma}{2-\gamma}g^{-1}\partial_{\rho}g\right )\partial_{x}\rho + \left ( \frac{1}{2} \frac{B}{\rho} g - \frac{1}{2} \frac{B^{2}+B_{0}^{2}}{\rho^{\gamma+1}}Bg - \frac{1}{4} \frac{B^{3}}{\rho^{\gamma+1}}g^{2} +\beta\frac{1}{2}\frac{\gamma-4}{2-\gamma}\frac{B}{\rho^{\gamma-1}}\partial_{\rho}g\right )\partial_{x}B\right ) \\
&& -Z\big(\partial_{x}B\big)^{2} \left ( \frac{2\gamma}{2-\gamma}\frac{1}{\rho}B_{0}^{2}g^{\frac{2}{2-\gamma}}-\beta \frac{B^{2}}{\rho^{\gamma-1}}g^{\frac{2}{2-\gamma}}\partial_{\rho}g\right )- \beta Z \left ( \frac{4\gamma}{2-\gamma}\rho^{\gamma-1}g^{\frac{\gamma}{2-\gamma}}\partial_{x}\rho - \frac{2\gamma}{2-\gamma}B g^{\frac{2}{2-\gamma}}\partial_{x}B\right )^{2} \\
& \le & -\partial_{x}Z \left ( \left ( -\frac{1}{2}\frac{B^{2}}{\rho^{2}}g + \beta \frac{\gamma}{2-\gamma}g^{-1}\partial_{\rho}g\right )\partial_{x}\rho + \left ( \frac{1}{2} \frac{B}{\rho} g - \frac{1}{2} \frac{B^{2}+B_{0}^{2}}{\rho^{\gamma+1}}Bg - \frac{1}{4} \frac{B^{3}}{\rho^{\gamma+1}}g^{2} +\beta\frac{1}{2}\frac{\gamma-4}{2-\gamma}\frac{B}{\rho^{\gamma-1}}\partial_{\rho}g\right )\partial_{x}B\right ),
\end{IEEEeqnarray*}}
where the last inequality follows from the fact that {\small
\begin{IEEEeqnarray*}{rCl}
\IEEEeqnarraymulticol{3}{l}{\frac{2\gamma}{2-\gamma}\frac{1}{\rho}B_{0}^{2}-\beta \frac{B^{2}}{\rho^{\gamma-1}}\partial_{\rho}g \ge 0} \\
&\Longleftrightarrow& \frac{2\gamma}{2-\gamma}\frac{1}{\rho}B_{0}^{2}\ge B_{0}^{2} \frac{g}{g+2} \frac{B^{2}}{\rho} \frac{4B_{0}^{2}\gamma}{\big(B^{2}+B_{0}^{2}-\rho^{\gamma}\big)^{2}+4B^{2}\rho^{\gamma}}\frac{1}{1-\frac{B^{2}-B_{0}^{2}+\rho^{\gamma}}{\sqrt{\big( B^{2}+B_{0}^{2}-\rho^{\gamma}\big)^{2}+4B^{2}\rho^{\gamma}}}} \\
&\Longleftrightarrow& \frac{1}{2-\gamma}\ge \frac{g}{g+2}B^{2} \frac{2B_{0}^{2}}{\big(B^{2}+B_{0}^{2}-\rho^{\gamma}\big)^{2}+4B^{2}\rho^{\gamma}}\frac{1}{1-\frac{B^{2}-B_{0}^{2}+\rho^{\gamma}}{\sqrt{\big( B^{2}+B_{0}^{2}-\rho^{\gamma}\big)^{2}+4B^{2}\rho^{\gamma}}}} \\
&\Longleftarrow& \left ( 1+\frac{2}{g}\right )\left ( \big(B^{2}+B_{0}^{2}-\rho^{\gamma}\big)^{2}+4B^{2}\rho^{\gamma}-\big (B^{2}-B_{0}^{2}+\rho^{\gamma}\big )\sqrt{\big(B^{2}+B_{0}^{2}-\rho^{\gamma}\big)^{2}+4B^{2}\rho^{\gamma}}\right ) \ge 2B_{0}^{2}B^{2} \\
& \Longleftrightarrow & \big (B^{2}+B_{0}^{2}-\rho^{\gamma}\big )^{2}+4B^{2}\rho^{\gamma} + (B^{2}+B_{0}^{2}-\rho^{\gamma}) \sqrt{\big(B^{2}+B_{0}^{2}-\rho^{\gamma}\big)^{2}+4B^{2}\rho^{\gamma}} \ge 2B^{2}\rho^{\gamma} \\
& \Longleftarrow & \big(B^{2}+B_{0}^{2}-\rho^{\gamma}\big)^{4} +4B^{2}\rho^{\gamma}\big(B^{2}+B_{0}^{2}-\rho^{\gamma}\big)^{2}+4B^{4}\rho^{2\gamma} \ge \big(B^{2}+B_{0}^{2}-\rho^{\gamma}\big)^{4}+4B^{2}\rho^{\gamma}\big(B^{2}+B_{0}^{2}-\rho^{\gamma}\big)^{2}.
\end{IEEEeqnarray*}}

\subsubsection{Proof that $\rho$, $\rho^{-1}$, and $|B|$ are bounded from above} \label{subsubsec2.2.3}

Recall that $(\rho, B) \mapsto (W(\rho, B), Z(\rho, B))$ is $C^{2}$ by Appendix \ref{appendixA}. Hence, if $\rho, B$ are both $C^{2}([0, T) \times \TT)$ functions, then $W(\rho, B)$ and $Z(\rho, B)$ are also both $C^{2}$ functions. Now assume further that $\rho$ is positive on $[0, T) \times \TT$. Under these assumptions, we claim that $\rho, \rho^{-1}$, and $B$ have uniform upper bound on $[0, T) \times \TT$, depending only on $\rho_{t=0}$ and $B_{t=0}$. Note that it is enough to prove boundedness on $[0, T-\epsilon] \times \TT$ with uniform-in-$\epsilon$ bound, so without loss of generality we may assume that $\rho, B$ are $C^{2}([0, T]\times \TT)$ and $\rho>0$ on $[0, T]\times \TT$.

Let $\delta>0$ be arbitrary, and consider the point $(t_{W}, x_{W})$ where \[
W(\rho(t, x), B(t, x))e^{-\delta t} \quad ((t, x) \in [0, T] \times \TT)
\]
is maximized. We claim that $t_{W}=0$, or at least that we may assume so. Assume $t_{W}>0$. Then \[
\partial_{x}W(t_{W}, x_{W}) =0, \quad \partial_{x}^{2}W(t_{W}, x_{W}) \le 0, \quad \partial_{t}W(t_{W}, x_{W}) -\delta W(t_{W}, x_{W}) \ge 0.
\] 

If $f(t_{W}, x_{W})=\infty$, this would imply $W(t_{W}, x_{W})=0$, and thus $W\equiv 0$ on $[0, T] \times \TT$. Hence we may assume $t_{W}=0$ in this case. 

If $f(t_{W}, x_{W})<\infty$, then from the equation we obtained in Section \ref{subsubsec2.2.2}, \[
\delta W \le \partial_{t}W-\alpha \partial_{x}^{2}W \le -\partial_{x}W \cdot \left ( \text{something} \right ) =0
\] at $(t_{W}, x_{W})$. Thus $W(t_{W}, x_{W})=0$, and this contradicts that $f(t_{W}, x_{W})<\infty$.

This proves that \[
\max_{(t, x) \in [0, T]\times \TT} W(t, x)e^{-\epsilon t} = \max_{x \in \TT} W(0, x),
\] and taking the limit $\delta \to 0+$ gives \[
\max_{(t, x) \in [0, T]\times \TT}W(t, x) = \max_{x \in \TT}W(0, x).
\] Similar argument with $(Z, \beta, g)$ instead of $(W, \alpha, f)$ shows that \[
\max_{(t, x) \in [0, T]\times \TT}Z(t, x) = \max_{x \in \TT} Z(0, x).
\] Thus, $W, Z$ has uniform upper bound on $[0, T]\times \TT$, which depend only on $\rho_{t=0}$ and $B_{t=0}$. We are left to prove that \[
\{(\rho, B) \in \RR^{+} \times \RR: W(\rho, B) \le W_{0}, Z(\rho, B) \le Z_{0}\}
\] is a bounded set away from the line $\rho=0$. (It is worthwhile to note that $Z_{0}<1$.) The level curves depicted in Figure \ref{20260111fig1} suggests that this must be true.

\begin{figure}[htbp]
\centering
\begin{subfigure}[b]{0.4\textwidth}
\centering
\includegraphics[width=\textwidth]{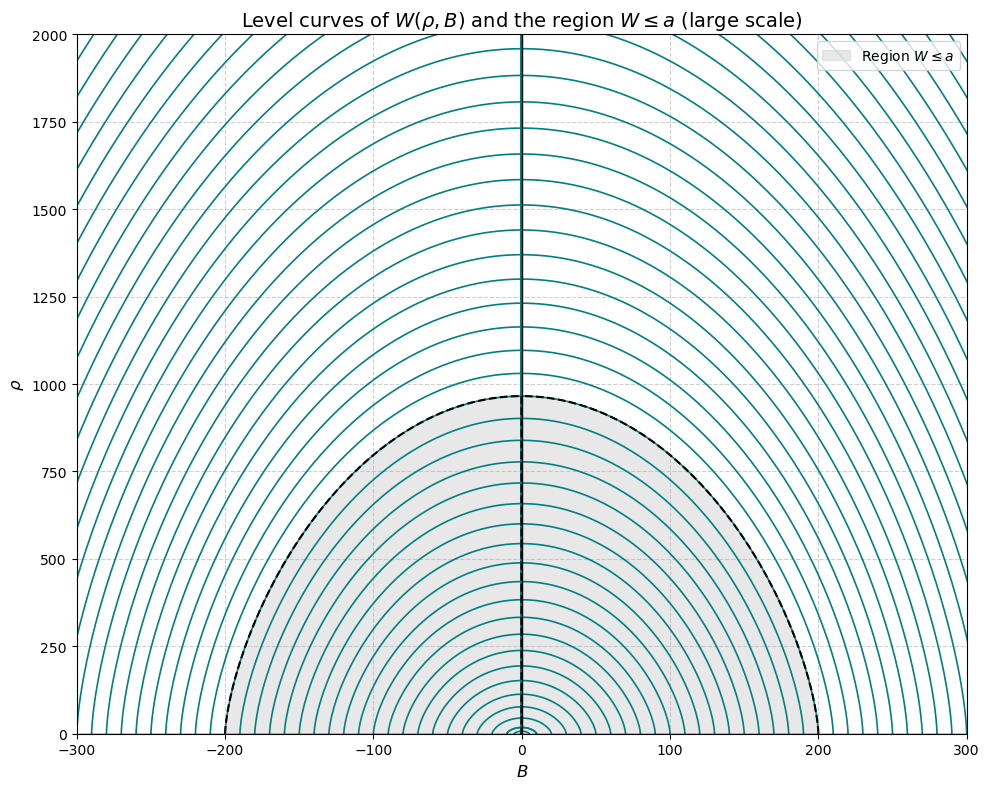}
\caption{Level curves of $W$ (large scale)}
\end{subfigure}
\hfill
\begin{subfigure}[b]{0.4\textwidth}
\centering
\includegraphics[width=\textwidth]{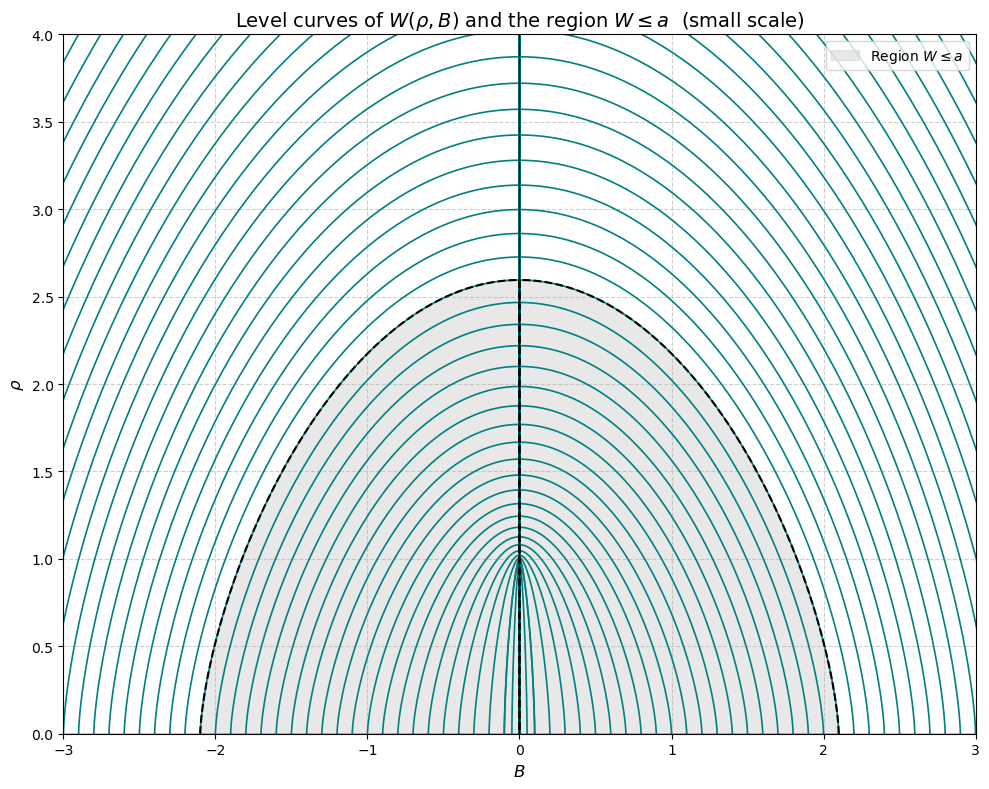}
\caption{Level curves of $W$ (small scale)}
\end{subfigure}
\vspace{1em}
\begin{subfigure}[b]{0.4\textwidth}
\centering
\includegraphics[width=\textwidth]{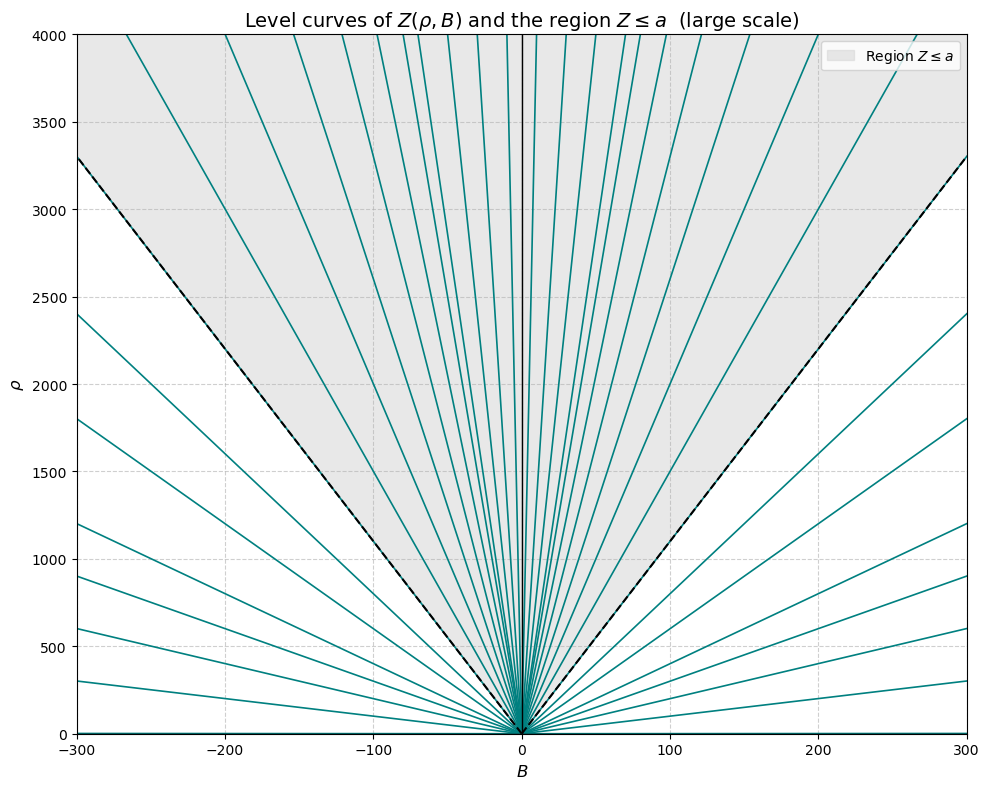}
\caption{Level curves of $Z$ (large scale)}
\end{subfigure}
\hfill
\begin{subfigure}[b]{0.4\textwidth}
\centering
\includegraphics[width=\textwidth]{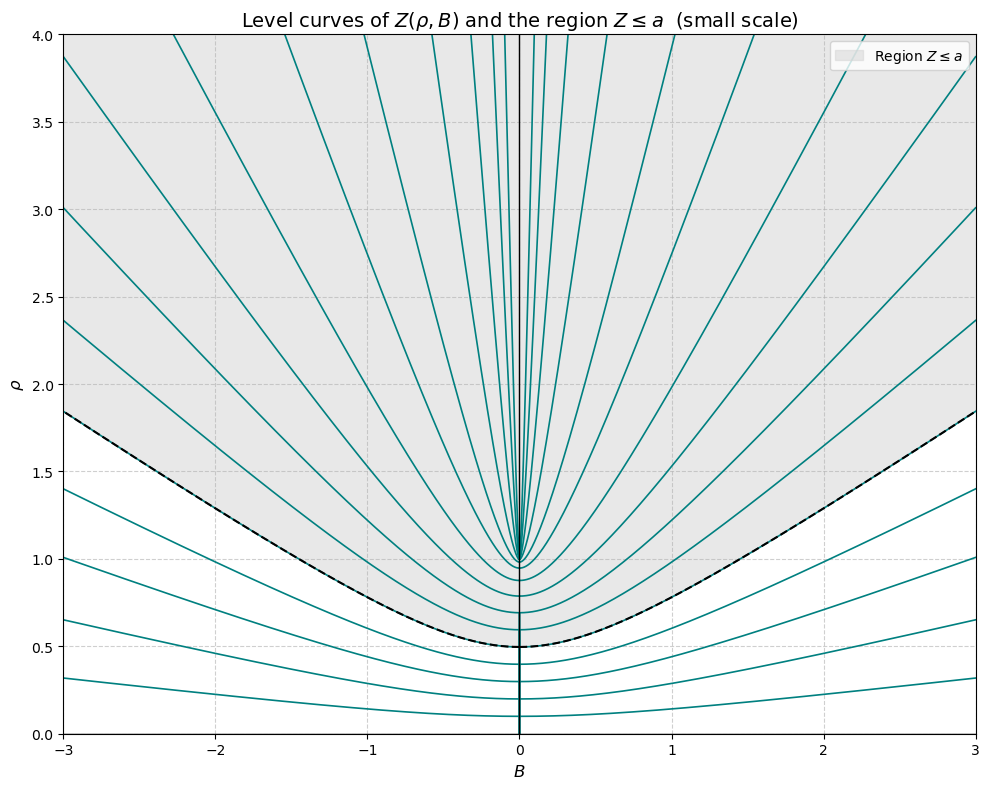}
\caption{Level curves of $Z$ (small scale)}
\end{subfigure}
\caption{Level curves of $W$ and $Z$ drawn for $\gamma=1.5, B_{0}=1$. As one can see in the small scale plots, there is a singularity near $(\rho, B)=(B_{0}^{\frac{2}{\gamma}}, 0)$. The shaded regions are the set $\{W \le W_{0}\}$, $\{Z \le Z_{0}\}$, for some constants $W_{0} \in \RR$ and $Z_{0}<1$.}
\label{20260111fig1}
\end{figure}

To prove this rigorously, we will first prove that if $W \le W_{0}$, then $\rho$ and $|B|$ are bounded from above. From our calculation of $\partial_{\rho}W$ and $\partial_{B}W$ in Appendix \ref{appendixA} (equation \eqref{20260117eq15} and \eqref{20260117eq16}), we observe that \[
\partial_{\rho}W(\rho, B) \ge 0, \quad \partial_{B}W(\rho, B) \ge 0 \quad \text{if }\rho>0 \text{ and }B \ge 0.
\] Hence, if $W(\rho_{M}, 0) \ge M$ and $W(\rho, B_{M})$ for all $0<\rho\le 1$, then $W(\rho, B) \ge M$ whenever $\rho \ge \rho_{M}$ or $B \ge B_{M}$. Since $W(\rho, B)$ is even in $B$, it is enough to prove that for any $M>0$, there exists some $\rho_{M}$ and $B_{M}$ such that 
\begin{itemize}
\item[(1)] $W(\rho_{M}, 0) \ge M$, and
\item[(2)] $W(\rho, B_{M}) \ge M$ for all $0<\rho\le 1$.
\end{itemize}
(1) follows from the following observation: \begin{IEEEeqnarray*}{rCl}
\lim_{\rho \to \infty}w(\rho, 0) &=& \lim_{\rho \to \infty}B_{0}^{2}\left ( \frac{2\rho^{\gamma}}{\rho^{\gamma}-B_{0}^{2}}\right )^{\frac{2}{2-\gamma}} - \int_{\frac{2\rho^{\gamma}}{\rho^{\gamma}-B_{0}^{2}}}^{4}\frac{2B_{0}^{2}}{2-\gamma}\frac{1-\frac{1}{4}s}{\left (1-\frac{1}{2}s\right )^{2}}s^{\frac{2}{2-\gamma}}\dd s \\
&=& \lim_{\rho \to \infty} 4^{\frac{2}{2-\gamma}}B_{0}^{2} -\int_{\frac{2\rho^{\gamma}}{\rho^{\gamma}-B_{0}^{2}}}^{4}\frac{2B_{0}^{2}}{2-\gamma}\frac{1}{\left (1-\frac{1}{2}s\right )^{2}}s^{\frac{\gamma}{2-\gamma}}\dd s \\
&=& C_{1} - C_{2}\lim_{f \to 2+} \int_{f}^{4} \frac{1}{\left(1-\frac{1}{2}s\right )^{2}}s^{\frac{\gamma}{2-\gamma}}\dd s = -\infty.
\end{IEEEeqnarray*}
(2) follows from 
\begin{IEEEeqnarray*}{rCl}
\IEEEeqnarraymulticol{3}{l}{\lim_{B \to \infty}\sup_{0 < \rho \le 1}w(\rho, B)} \\
&=& \lim_{B \to \infty} \sup_{0<\rho\le 1}\left [4^{\frac{2}{2-\gamma}}B_{0}^{2}+B^{2}f^{\frac{2}{2-\gamma}} - \int_{f}^{4} \frac{2B_{0}^{2}}{2-\gamma} \frac{1}{\left (1-\frac{1}{2}s\right )^{2}}s^{\frac{\gamma}{2-\gamma}}\dd s\right ] \\
& \le & \limsup_{B \to \infty}\sup_{0<\rho \le 1} \left [ 4^{\frac{2}{2-\gamma}}B_{0}^{2} + B^{2}f^{\frac{2}{2-\gamma}} - \frac{2B_{0}^{2}}{2-\gamma}\int_{f}^{4}\frac{1}{\left (1-\frac{1}{2}s\right )^{2}}f^{\frac{\gamma}{2-\gamma}}\dd s \right ] \\
& =& \limsup_{B \to \infty} \sup_{0<\rho \le 1} \left [C_{1} + B^{2}f^{\frac{2}{2-\gamma}} -\frac{8}{2-\gamma}B_{0}^{2}\frac{1}{f-2}f^{\frac{\gamma}{2-\gamma}}\right ] \\
& =& \limsup_{B \to \infty} \sup_{0<\rho \le 1} \left [C_{1} + B^{2}f^{\frac{2}{2-\gamma}} -\frac{4}{2-\gamma}B_{0}^{2}\frac{\sqrt{\big(B^{2}+B_{0}^{2}-\rho^{\gamma}\big)^{2}+4B^{2}\rho^{\gamma}}-\big(B^{2}+B_{0}^{2}-\rho^{\gamma}\big)}{B^{2}+B_{0}^{2}+\rho^{\gamma}-\sqrt{\big(B^{2}+B_{0}^{2}-\rho^{\gamma}\big)^{2}+4B^{2}\rho^{\gamma}}}f^{\frac{\gamma}{2-\gamma}}\right ] \\
& =& \limsup_{B \to \infty} \sup_{0<\rho \le 1} \left [C_{1} + B^{2}f^{\frac{2}{2-\gamma}} -\frac{2}{2-\gamma}\left (\big(B^{2}-B_{0}^{2}+\rho^{\gamma}\big)+\sqrt{\big(B^{2}+B_{0}^{2}-\rho^{\gamma}\big)^{2}+4B^{2}\rho^{\gamma}}\right )f^{\frac{\gamma}{2-\gamma}}\right ] \\
&\le & \limsup_{B \to \infty}\sup_{0 < \rho \le 1} \left [ C_{1} + f^{\frac{\gamma}{2-\gamma}} \left ( B^{2}+B_{0}^{2}-\rho^{\gamma} - \sqrt{\big(B^{2}+B_{0}^{2}-\rho^{\gamma}\big)^{2}+4B^{2}\rho^{\gamma}} - 2(B^{2}-B_{0}^{2}+\rho^{\gamma}) \right ) \right ] \\
&=& -\infty.
\end{IEEEeqnarray*}
Hence, $W \le W_{0}$ implies that $\rho$ and $|B|$ are bounded from above. Next, we will prove that if $Z \le Z_{0}<1$, then $\rho$ is bounded away from zero. We have
\begin{IEEEeqnarray*}{rCl}
\lim_{\rho \to 0}z(\rho, 0) &=& \lim_{\rho \to 0} \int_{0}^{\frac{2\rho^{\gamma}}{B_{0}^{2}-\rho^{\gamma}}}\frac{2B_{0}^{2}}{2-\gamma}\frac{1}{\left (1+\frac{1}{2}s\right )^{2}}s^{\frac{\gamma}{2-\gamma}}\dd s \\
&=& \int_{g \to 0+} \int_{0}^{g}\frac{2B_{0}^{2}}{2-\gamma}\frac{1}{\left (1+\frac{1}{2}s \right )^{2}}s^{\frac{\gamma}{2-\gamma}}\dd s =0.
\end{IEEEeqnarray*} 
Hence $Z(\rho, 0) \to 1$ as $\rho \to 0+$, and thus for any $Z_{0}<1$, there exists $\rho_{Z_{0}}>0$ such that $Z(\rho, 0) >Z_{0}$ for all $0<\rho \le \rho_{Z_0}$. Now, from \eqref{20260117eq25} we have $\partial_{B}Z(\rho, B)\mathrm{sgn}(B) \ge 0$, and this shows that $Z(\rho, B) \ge Z(\rho, 0) > Z_{0}$ for any $0 < \rho \le \rho_{Z_0}$ and $B \in \RR$. Thus, $Z \le Z_{0}$ implies that $\rho > \rho_{Z_0}>0$, as desired.

Hence, $\rho, \rho^{-1}$ and $|B|$ are bounded from above by a constant depending on $\rho_{t=0}$ and $B_{t=0}$. We finally remark that this upper bound can be completely determined by the values of $W_{0}$ and $Z_{0}$ (along with constants $\gamma$ and $B_{0}$). By definition of $W_{0}$ and $Z_{0}$, they can be chosen to depend only on $\|\rho_{t=0}\|_{L^{\infty}}, \|\rho_{t=0}^{-1}\|_{L^{\infty}}$, and $\|B_{t=0}\|_{L^{\infty}}$. Thus, the uniform bounds on $\rho, \rho^{-1}$, and $|B|$ can be chosen to depend only on $\gamma$, $B_{0}$, $\|\rho_{t=0}\|_{L^{\infty}}$, $\|\rho_{t=0}^{-1}\|_{L^{\infty}}$, and $\|B_{t=0}\|_{L^{\infty}}$.

\subsection{Proof of Theorem \ref{20260110thm1}}

From the a priori estimate in Lemma \ref{20260110lem1} and no-vacuum result in Section \ref{subsec2.2}, standard approximation argument gives the local well-posedness result. Indeed, consider the parabolic regularization \begin{eqnarray} \label{20260111eq2}
\begin{cases}
\partial_{t}\rho^{(\epsilon)} = \partial_{x}^{2}\left ( \frac{1}{\gamma}\left (\rho^{(\epsilon)}\right )^{\gamma}+\frac{1}{2}\left (B^{(\epsilon)}\right )^{2}\right ) - \epsilon \partial_{x}^{4}\rho^{(\epsilon)}, \\
\partial_{t}B^{(\epsilon)} = \partial_{x}\left ( \frac{B^{(\epsilon)}}{\rho^{(\epsilon)}}\partial_{x}\left ( \frac{1}{\gamma}\left (\rho^{(\epsilon)}\right )^{\gamma}+\frac{1}{2}\left (B^{(\epsilon)}\right )^{2} \right )+\frac{B_{0}^{2}}{\rho^{(\epsilon)}}\partial_{x}B^{(\epsilon)}\right )-\epsilon \partial_{x}^{4}B^{(\epsilon)}, \\
(\rho^{(\epsilon)}, B^{(\epsilon)})(t=0) = (\rho_{t=0}, B_{t=0}).
\end{cases}
\end{eqnarray} Since the initial data is smooth, we have local-in-time smooth solution for this equation. To check that we have a uniform-in-$\epsilon$ bound, we can use Lemma \ref{20260110lem1}. The additional term needed to estimate \[
\partial_{t}\int_{\TT} \left (\left |\left (\rho^{(\epsilon)}\right )^{\frac{\gamma}{2}-1}\partial_{x}^{s}\rho^{(\epsilon)} \right |^{2} + \left |\partial_{x}^{s}B^{(\epsilon)}\right |^{2}\right )
\] is \[
 - \epsilon \int_{\TT}\left |\partial_{x}^{s+2}B^{(\epsilon)}\right |^{2}\dd x - \epsilon \int_{\TT} \left (\rho^{(\epsilon)}\right )^{\frac{\gamma}{2}-1}\partial_{x}^{s}\rho^{(\epsilon)} \left ( \left (\rho^{(\epsilon)}\right )^{\frac{\gamma}{2}-1}\partial_{x}^{s+4}\rho^{(\epsilon)} + (\frac{\gamma}{2}-1)\left (\rho^{(\epsilon)}\right )^{\frac{\gamma}{2}-2} \partial_{x}^{s}\rho^{(\epsilon)}\partial_{x}^{4}\rho^{(\epsilon)}\right )\dd x,
\] and performing integration by parts shows that this can be bounded by
\begin{IEEEeqnarray*}{rCl}
\IEEEeqnarraymulticol{3}{l}{- \epsilon \int_{\TT}\left |\partial_{x}^{s+2}B^{(\epsilon)}\right |^{2}\dd x - \epsilon \int_{\TT} \left (\rho^{(\epsilon)}\right )^{\frac{\gamma}{2}-1}\partial_{x}^{s}\rho^{(\epsilon)} \left ( \left (\rho^{(\epsilon)}\right )^{\frac{\gamma}{2}-1}\partial_{x}^{s+4}\rho^{(\epsilon)} + (\frac{\gamma}{2}-1)\left (\rho^{(\epsilon)}\right )^{\frac{\gamma}{2}-2} \partial_{x}^{s}\rho^{(\epsilon)}\partial_{x}^{4}\rho^{(\epsilon)}\right )\dd x} \\
&=& -\epsilon \int_{\TT} \left | \partial_{x}^{s+2}B^{(\epsilon)}\right |^{2}\dd x - \epsilon \int_{\TT} \left ( \frac{\gamma}{2}-1\right ) \left (\rho^{(\epsilon)}\right )^{\gamma-3}\partial_{x}^{4}\rho^{(\epsilon)} \big(\partial_{x}^{s}\rho^{(\epsilon)}\big)^{2} - \epsilon \int_{\TT} \partial_{x}^{2}\left ( \left (\rho^{(\epsilon)}\right )^{\gamma-2}\partial_{x}^{s}\rho^{(\epsilon)}\right )\partial_{x}^{s+2}\rho^{(\epsilon)} \\
&=& -\epsilon \int_{\TT} \left | \partial_{x}^{s+2}B^{(\epsilon)}\right |^{2}\dd x  -\epsilon \int_{\TT} \left (\rho^{(\epsilon)}\right )^{\gamma-2}\big (\partial_{x}^{s+2}\rho^{(\epsilon)}\big )^{2}  \\
&& - 2\epsilon (\gamma-2) \int_{\TT} \left ( \left (\rho^{(\epsilon)}\right )^{\gamma-3}\partial_{x}^{2}\rho^{(\epsilon)} + (\gamma-3)\left (\rho^{(\epsilon)}\right )^{\gamma-4}\big (\partial_{x}\rho^{(\epsilon)}\big )^{2}\right )\partial_{x}^{s}\rho^{(\epsilon)}\partial_{x}^{s+2}\rho^{(\epsilon)} \\
&& +\frac{1}{2} \epsilon (\gamma-2) \int_{\TT} \left [\partial_{x}^{3}\left (\left (\rho^{(\epsilon)}\right )^{\gamma-3}\partial_{x}\rho^{(\epsilon)}\right ) - \left (\rho^{(\epsilon)}\right )^{\gamma-3}\partial_{x}^{4}\rho^{(\epsilon)} \right ]\big (\partial_{x}^{s}\rho^{(\epsilon)}\big )^{2} \\
& \le & 2\epsilon (\gamma-2)^{2} \int_{\TT} \left (\rho^{(\epsilon)}\right )^{\gamma-4}\big(\partial_{x}^{2}\rho^{(\epsilon)}\big)^{2}\big(\partial_{x}^{s}\rho^{(\epsilon)}\big)^{2} + (\gamma-3)^{2}\left (\rho^{(\epsilon)}\right )^{\gamma-6}\big(\partial_{x}\rho^{(\epsilon)}\big)^{4}\big(\partial_{x}^{s}\rho^{(\epsilon)}\big)^{2} \\
&& +\frac{1}{2} \epsilon (\gamma-2) \int_{\TT} \left [\partial_{x}^{3}\left (\left (\rho^{(\epsilon)}\right )^{\gamma-3}\partial_{x}\rho^{(\epsilon)}\right ) - \left (\rho^{(\epsilon)}\right )^{\gamma-3}\partial_{x}^{4}\rho^{(\epsilon)} \right ]\big (\partial_{x}^{s}\rho^{(\epsilon)}\big )^{2} \\ 
& \le & C_{\gamma} \|\partial_{x}^{s}\rho^{(\epsilon)}\|_{L^{2}}^{2} \|\partial_{x}^{3}\rho^{(\epsilon)}\|_{L^{\infty}} \left ( \eta^{-1} + \|\rho^{(\epsilon)}\|_{W^{3, \infty}}\right )^{C_{\gamma}}.
\end{IEEEeqnarray*}
Thus if $s \ge 5$, this additional term can be absorbed into the right-handed side of \eqref{20260111eq7}. This gives uniform-in-$\epsilon$ bound for $\dot{H}^{s} (s \ge 5)$, assuming $\rho^{(\epsilon)} \ge \eta$. Together with Lemma \ref{20260111lem1}, we obtain a uniform-in-$\epsilon$ bound for $H^{s} (s \ge 5)$, assuming $\rho^{(\epsilon)} \ge \eta$. Noting that \[
\|\partial_{t}\rho^{(\epsilon)}\|_{L^{\infty}} \le C_{\gamma} \left ( \|\rho^{(\epsilon)}\|_{H^{5}}+\|B^{(\epsilon)}\|_{H^{5}}+\eta^{-1}\right )^{C_{\gamma}},
\] a bootstrapping argument shows that there exist a uniform time interval $[0, T]$ such that $\rho^{(\epsilon)}$ is guaranteed to be bigger than some $\eta>0$ (not depending on $\epsilon$), and that there is a uniform $H^{5}$ bound given by Corollary \ref{20260110cor1}. Using $H^{s}$ a-priori estimate inductively gives uniform-in-$\epsilon$ $H^{s}$ estimate for all positive integer $s$. Thus some subsequence of the family $\{(\rho^{(\epsilon)}, B^{(\epsilon)})\}_{\epsilon>0}$ converges to a smooth solution $(\rho, B)$ of \eqref{20260110eq2}.

Based on the result in Section \ref{subsec2.2}, this solution satisfies $\rho \ge \eta$ for some $\eta>0$. In view of Lemma \ref{20260110lem1}, this smooth solution $(\rho, B)$ can be extended up to a maximal time $T$ such that $\lim_{t \to T-}\left (\|\rho(t)\|_{H^{3}} + \|B(t)\|_{H^{3}}\right ) = \infty$. Applying Lemma \ref{20260110lem1} again with $s=3$, we see that this existence time $T$ is bounded from below by a constant depending only on $\gamma, B_{0}, \|\rho_{t=0}\|_{H^{3}}, \|B_{t=0}\|_{H^{3}}$, and $\eta$. As $\eta$ has a lower bound depending on $\gamma, B_{0}, \|\rho_{t=0}\|_{L^{\infty}}, \|B_{t=0}\|_{L^{\infty}}$, and $\|\rho_{t=0}^{-1}\|_{L^{\infty}}$ (see the remark at the end of section \ref{subsubsec2.2.3}), we conclude the lower bound for $T$ of the form $T \ge T(\gamma, B_{0}, \|\rho_{t=0}\|_{H^{3}}, \|B_{t=0}\|_{H^{3}}, \|\rho_{t=0}^{-1}\|_{L^{\infty}})$. This completes the existence part of Theorem \ref{20260110thm1}.

To prove uniqueness, let $(\rho, B)$ and $(\widetilde{\rho}, \widetilde{B})$ be two solutions with the same initial data $(\rho_{t=0}, B_{t=0})$. Let $\eta = \eta(\gamma, B_{0}, \rho_{t=0}, B_{t=0})>0$ be a constant such that $\rho, \rho^{-1}, |B|, \widetilde{\rho}, \widetilde{\rho}^{-1}, \left |\widetilde{B}\right | < \eta^{-1}$. Observe that
\begin{IEEEeqnarray*}{rCl}
\IEEEeqnarraymulticol{3}{l}{\partial_{t}\frac{1}{2}\int_{\TT} \left ( \rho^{\gamma-2}\left |\rho-\widetilde{\rho}\right |^{2}+\left |B-\widetilde{B}\right |^{2}\right )} \\
& = & \int_{\TT} \frac{\gamma-2}{2}\rho^{\gamma-3}\partial_{x}^{2}\left (\frac{1}{\gamma}\rho^{\gamma}+\frac{1}{2}B^{2}\right ) \left |\rho-\widetilde{\rho}\right |^{2} + \int_{\TT} \rho^{\gamma-2}\left (\rho-\widetilde{\rho}\right )\partial_{x}^{2}\left ( \frac{1}{\gamma}\left (\rho^{\gamma}-{\widetilde{\rho}}^{\gamma}\right )+\frac{1}{2}\left (B^{2}-{\widetilde{B}}^{2}\right )\right ) \\
&& + \int_{\TT} \left ( B - \widetilde{B}\right )\partial_{x}\left ( \rho^{\gamma-2}B\partial_{x}\rho - {\widetilde{\rho}}^{\gamma-2}\widetilde{B}\partial_{x}\widetilde{\rho} + \frac{B^{2}+B_{0}^{2}}{\rho}\partial_{x}B - \frac{{\widetilde{B}}^{2}+B_{0}^{2}}{\widetilde{\rho}}\partial_{x}\widetilde{B}\right ) \\
& = & \int_{\TT} \frac{\gamma-2}{2}\rho^{\gamma-3}\partial_{x}^{2}\left (\frac{1}{\gamma}\rho^{\gamma}+\frac{1}{2}B^{2}\right ) \left |\rho-\widetilde{\rho}\right |^{2} \\
&& - \int_{\TT} (\gamma-2)\rho^{\gamma-3}\partial_{x}\rho\left (\rho-\widetilde{\rho}\right )\left ( \rho^{\gamma-1}\partial_{x}\rho - {\widetilde{\rho}}^{\gamma-1}\partial_{x}\widetilde{\rho}\right ) \\
&& - \int_{\TT} (\gamma-2)\rho^{\gamma-3}\partial_{x}\rho\left (\rho-\widetilde{\rho}\right )\left ( B\partial_{x}B-\widetilde{B}\partial_{x}\widetilde{B}\right ) \\
&& - \int_{\TT} \rho^{\gamma-2}\left (\partial_{x}\rho-\partial_{x}\widetilde{\rho}\right )\left ( \rho^{\gamma-1}\partial_{x}\rho - {\widetilde{\rho}}^{\gamma-1}\partial_{x}\widetilde{\rho}\right ) \\
&& - \int_{\TT} \rho^{\gamma-2}\left (\partial_{x}\rho-\partial_{x}\widetilde{\rho}\right )\left ( B\partial_{x}B-\widetilde{B}\partial_{x}\widetilde{B}\right ) \\
&& - \int_{\TT} \left ( \partial_{x}B - \partial_{x}\widetilde{B}\right ) \rho^{\gamma-2}B\left (\partial_{x}\rho-\partial_{x}\widetilde{\rho}\right ) \\
&& - \int_{\TT} \left ( \partial_{x}B - \partial_{x}\widetilde{B}\right ) \rho^{\gamma-2}\left (B-\widetilde{B}\right )\partial_{x}\widetilde{\rho} \\
&& - \int_{\TT} \left ( \partial_{x}B - \partial_{x}\widetilde{B}\right ) \left (\rho^{\gamma-2}-{\widetilde{\rho}}^{\gamma-2}\right )\widetilde{B}\partial_{x}\widetilde{\rho} \\
&& - \int_{\TT} \left ( \partial_{x}B - \partial_{x}\widetilde{B}\right ) \frac{B^{2}+B_{0}^{2}}{\rho}\left (\partial_{x}B-\partial_{x}\widetilde{B}\right ) \\
&& - \int_{\TT} \left ( \partial_{x}B - \partial_{x}\widetilde{B}\right ) \frac{\left (B-\widetilde{B}\right )\left (B+\widetilde{B}\right )}{\rho}\partial_{x}\widetilde{B} \\
&& - \int_{\TT} \left ( \partial_{x}B - \partial_{x}\widetilde{B}\right ) \frac{\widetilde{B}^{2}+B_{0}^{2}}{\rho\widetilde{\rho}}\partial_{x}\widetilde{B}\left (\widetilde{\rho}-\rho\right ) \\
&=:& I_{1}+I_{2}+I_{3}+I_{4}+I_{5}+I_{6}+I_{7}+I_{8}+I_{9}+I_{10}+I_{11}.
\end{IEEEeqnarray*}
We can bound each term by
\begin{IEEEeqnarray*}{rCl}
I_{1} &\le & C_{\gamma}\left ( \eta^{-1}+\|\rho\|_{W^{2, \infty}}+\|B\|_{W^{2, \infty}}\right )^{C_{\gamma}} \|\rho-\widetilde{\rho}\|_{L^{2}}^{2}, \\
I_{2} & =& -\int_{\TT}(\gamma-2)\rho^{\gamma-3}\partial_{x}\rho \left (\rho-\widetilde{\rho}\right )\rho^{\gamma-1}\left (\partial_{x}\rho-\partial_{x}\widetilde{\rho}\right ) - \int_{\TT} (\gamma-2)\rho^{\gamma-3}\partial_{x}\rho \left (\rho-\widetilde{\rho}\right ) \left (\rho^{\gamma-1}-{\widetilde{\rho}}^{\gamma-1}\right )\partial_{x}\widetilde{\rho} \\
&=& \int_{\TT} \frac{\gamma-2}{2} \partial_{x}\left ( \rho^{2\gamma-4}\partial_{x}\rho\right ) \left (\rho-\widetilde{\rho}\right )^{2} - \int_{\TT}(\gamma-2)\rho^{\gamma-3}\partial_{x}\rho \partial_{x}\widetilde{\rho} \frac{\rho^{\gamma-1}-{\widetilde{\rho}}^{\gamma-1}}{\rho-\widetilde{\rho}} \left (\rho-\widetilde{\rho}\right )^{2} \\
& \le & C_{\gamma} \left (\eta^{-1}+\|\rho\|_{W^{2, \infty}} + \|\widetilde{\rho}\|_{W^{1, \infty}}\right )^{C_{\gamma}}\|\rho-\widetilde{\rho}\|_{L^{2}}^{2}, \\
I_{3} &=& -\int_{\TT}(\gamma-2)\rho^{\gamma-3}\partial_{x}\rho \left (\rho-\widetilde{\rho}\right )B\left (\partial_{x}B-\partial_{x}\widetilde{B}\right ) - \int_{\TT} (\gamma-2)\rho^{\gamma-3}\partial_{x}\rho \left (\rho-\widetilde{\rho}\right ) \left (B-\widetilde{B}\right )\partial_{x}\widetilde{B} \\
& \le & C_{\gamma}\left (\eta^{-1}+\|\rho\|_{W^{1, \infty}}+\|B\|_{L^{\infty}}+\|\widetilde{B}\|_{W^{1, \infty}}\right )^{C_{\gamma}}\|\rho-\widetilde{\rho}\|_{L^{2}}\left (\|B-\widetilde{B}\|_{L^{2}}+\|\partial_{x}B-\partial_{x}\widetilde{B}\|_{L^{2}} \right), \\
I_{4} &=& -\int_{\TT} \rho^{\gamma-2}\left (\partial_{x}\rho-\partial_{x}\widetilde{\rho}\right )^{2}\rho^{\gamma-1} - \int_{\TT} \left (\rho^{\gamma-2}\partial_{x}\rho-{\widetilde{\rho}}^{\gamma-2}\partial_{x}\widetilde{\rho}\right )\left (\rho^{\gamma-1}-\widetilde{\rho}^{\gamma-1}\right )\partial_{x}\widetilde{\rho} \\
&& +\int_{\TT} \left (\rho^{\gamma-2}-\widetilde{\rho}^{\gamma-2}\right )\partial_{x}\widetilde{\rho}\left (\rho^{\gamma-1}-\widetilde{\rho}^{\gamma-1}\right )\partial_{x}\widetilde{\rho} \\
& =& -\int_{\TT} \rho^{2\gamma-3}\left (\partial_{x}\rho-\partial_{x}\widetilde{\rho}\right )^{2} + \int_{\TT} \frac{1}{2(\gamma-1)} \left (\rho^{\gamma-1}-{\widetilde{\rho}}^{\gamma-1}\right )^{2} \partial_{x}^{2}\widetilde{\rho}\\
&& +\int_{\TT} \frac{\rho^{\gamma-1}-\widetilde{\rho}^{\gamma-1}}{\rho-\widetilde{\rho}} \frac{\rho^{\gamma-2}-\widetilde{\rho}^{\gamma-2}}{\rho-\widetilde{\rho}}\big (\partial_{x}\widetilde{\rho}\big )^{2}\left (\rho-\widetilde{\rho}\right )^{2} \\
& \le & -\int_{\TT} \rho^{2\gamma-3}\left (\partial_{x}\rho-\partial_{x}\widetilde{\rho}\right )^{2} + C_{\gamma} \left ( \eta^{-1}+\|\widetilde{\rho}\|_{W^{2, \infty}}\right )^{C_{\gamma}}\|\rho-\widetilde{\rho}\|_{L^{2}}^{2}, \\
I_{5} & =& -\int_{\TT}\rho^{\gamma-2}\left (\partial_{x}\rho-\partial_{x}\widetilde{\rho}\right )B \left (\partial_{x}B-\partial_{x}\widetilde{B}\right ) - \int_{\TT} \rho^{\gamma-2}\left (\partial_{x}\rho-\partial_{x}\widetilde{\rho}\right )\left (B-\widetilde{B}\right )\partial_{x}\widetilde{B} \\
&=& -\int_{\TT}\rho^{\gamma-2}\left (\partial_{x}\rho-\partial_{x}\widetilde{\rho}\right )B \left (\partial_{x}B-\partial_{x}\widetilde{B}\right ) + \int_{\TT} \rho^{\gamma-2}\left (\rho-\widetilde{\rho}\right )\left (\partial_{x}B-\partial_{x}\widetilde{B}\right )\partial_{x}\widetilde{B}\\
&& + \int_{\TT} \left (\rho-\widetilde{\rho}\right )\left (B-\widetilde{B}\right )\partial_{x}\left (\rho^{\gamma-2}\partial_{x}\widetilde{B}\right ) \\
& \le & -\int_{\TT}\rho^{\gamma-2}B\left (\partial_{x}\rho-\partial_{x}\widetilde{\rho}\right ) \left (\partial_{x}B-\partial_{x}\widetilde{B}\right ) \\
&& +C_{\gamma} \left ( \eta^{-1} + \|\rho\|_{W^{1, \infty}} + \|\widetilde{B}\|_{W^{2, \infty}} \right )^{C_{\gamma}} \|\rho-\widetilde{\rho}\|_{L^{2}}\left (\|B-\widetilde{B}\|_{L^{2}}+\|\partial_{x}B-\partial_{x}\widetilde{B}\|_{L^{2}}\right ), \\
I_{6} &=& -\int_{\TT} \rho^{\gamma-2}B\left (\partial_{x}\rho-\partial_{x}\widetilde{\rho}\right )\left (\partial_{x}B-\partial_{x}\widetilde{B}\right ), \\
I_{7} &=& \frac{1}{2} \int_{\TT} \partial_{x}\left (\rho^{\gamma-2}\partial_{x}\widetilde{\rho}\right )\left (B-\widetilde{B}\right )^{2} \le C_{\gamma}\left (\eta^{-1} + \|\widetilde{\rho}\|_{W^{1, \infty}}\right )^{C_{\gamma}}\|B-\widetilde{B}\|_{L^{2}}^{2}, \\
I_{8} &\le & \|\partial_{x}B-\partial_{x}\widetilde{B}\|_{L^{2}} \|\rho-\widetilde{\rho}\|_{L^{2}} C_{\gamma}\left (\eta^{-1}+\|\widetilde{\rho}\|_{W^{1, \infty}} + \|\widetilde{B}\|_{L^{\infty}}\right )^{C_{\gamma}}, \\
I_{9} & =& -\int_{\TT} \frac{B^{2}+B_{0}^{2}}{\rho} \left (\partial_{x}B-\partial_{x}\widetilde{B}\right )^{2}, \\
I_{10} &=& \frac{1}{2}\int_{\TT} \left (B-\widetilde{B}\right )^{2}\partial_{x}\left ( \frac{B+\widetilde{B}}{\rho}\partial_{x}\widetilde{B} \right ) \le C_{\gamma}\left (\eta^{-1}+\|\rho\|_{W^{1, \infty}} + \|B\|_{W^{1, \infty}} + \|\widetilde{B}\|_{W^{2, \infty}}\right )^{C_{\gamma}}\|B-\widetilde{B}\|_{L^{2}}^{2}, \\
I_{11} &\le &\|\partial_{x}B-\partial_{x}\widetilde{B}\|_{L^{2}} \|\rho-\widetilde{\rho}\|_{L^{2}} C_{B_{0}}\left ( \eta^{-1} + \|\widetilde{B}\|_{W^{1, \infty}} \right )^{C}.
\end{IEEEeqnarray*}
Hence,
\begin{IEEEeqnarray*}{rCl}
\IEEEeqnarraymulticol{3}{l}{\partial_{t}\frac{1}{2}\int_{\TT}\left (\rho^{\gamma-2}\left |\rho-\widetilde{\rho}\right |^{2}+\left |B-\widetilde{B}\right |^{2}\right )} \\
&\le & -\int_{\TT} \rho^{2\gamma-3}\left (\partial_{x}\rho-\partial_{x}\widetilde{\rho}\right )^{2} - \int_{\TT} 2\rho^{\gamma-2}B\left (\partial_{x}\rho-\partial_{x}\widetilde{\rho}\right )\left (\partial_{x}B-\partial_{x}\widetilde{B}\right ) - \int_{\TT} \frac{B^{2}+B_{0}^{2}}{\rho} \left (\partial_{x}B-\partial_{x}\widetilde{B}\right )^{2} \\
&& + C_{\gamma, B_{0}}\left (\eta^{-1}+\|\rho\|_{W^{2, \infty}}+\|\widetilde{\rho}\|_{W^{2, \infty}}+\|B\|_{W^{2, \infty}}+\|\widetilde{B}\|_{W^{2, \infty}}\right )^{C_{\gamma}}  \left ( \|\rho-\widetilde{\rho}\|_{L^{2}}^{2}+\|B-\widetilde{B}\|_{L^{2}}^{2}\right ) \\
&& + \frac{1}{2}\eta B_{0}^{2}\|\partial_{x}B-\partial_{x}\widetilde{B}\|_{L^{2}}^{2} \\
& \le & -\int_{\TT}\left [ \rho^{\gamma-\frac{3}{2}}\left (\partial_{x}\rho-\partial_{x}\widetilde{\rho}\right )+\rho^{-\frac{1}{2}}B\left (\partial_{x}B-\partial_{x}\widetilde{B}\right )\right ]^{2} -\int_{\TT} \frac{B_{0}^{2}}{2\rho}\left (\partial_{x}B-\partial_{x}\widetilde{B}\right )^{2} \\
&& + C_{\gamma, B_{0}} \left (\eta^{-1}+\|\rho\|_{W^{2, \infty}}+\|\widetilde{\rho}\|_{W^{2, \infty}}+\|B\|_{W^{2, \infty}} + \|\widetilde{B}\|_{W^{2, \infty}}\right )^{C_{\gamma}} \eta^{-C_{\gamma}}\left ( \|\rho^{\frac{\gamma}{2}-1}\left (\rho-\widetilde{\rho}\right )\|_{L^{2}}^{2} + \|B-\widetilde{B}\|_{L^{2}}^{2}\right ).
\end{IEEEeqnarray*}
Since $W^{2, \infty}$ norms of $(\rho, B)$ and $(\widetilde{\rho}, \widetilde{B})$ are bounded (by $H^{3}$ norm of them), we conclude by Gr\"onwall's inequality that \[
\int_{\TT}\left (\rho(t)^{\gamma-2}\left |\rho(t)-\widetilde{\rho}(t)\right |^{2}+\left |B(t)-\widetilde{B}(t)\right |^{2}\right ) =0.
\] Hence $(\rho, B)=(\widetilde{\rho}, \widetilde{B})$, and this proves the uniqueness part of Theorem \ref{20260110thm1}.

\section{Relaxation to constant steady state}\label{sec3}

\subsection{Decay of $H^{s}$ norm}\label{sec3.1}
Let $(\rho, B)$ be a smooth solution to the system \eqref{20260110eq2}. By Theorem \ref{20260110thm1}, there exists $\eta=\eta(\gamma, B_{0}, \rho_{t=0}, B_{t=0})>0$ such that $\eta \le \rho \le \eta^{-1}$. Let $s \ge 1$ be any integer. Then {\small \begin{IEEEeqnarray*}{rCl}
\IEEEeqnarraymulticol{3}{l}{\partial_{t}\frac{1}{2}\int_{\TT}\left |\partial_{x}^{s}\rho \right |^{2}} \\
&=& \int_{\TT}-\partial_{x}^{s+1}\rho\partial_{x}^{s+1}\left ( \frac{1}{\gamma}\rho^{\gamma}+\frac{1}{2}B^{2}\right ) \\
&=& \int_{\TT} -\rho^{\gamma-1}\big(\partial_{x}^{s+1}\rho\big)^{2} + \partial_{x}^{s+1}\rho \left (\sum_{\substack{k \ge 1, \\ 1 \le i_{1}, \ldots, i_{k} \le s, \\ i_{1}+\cdots +i_{k}=s+1, \\ \alpha+k=\gamma}}C_{\gamma, s, \alpha, i_{1}, \ldots, i_{k}}\rho^{\alpha}\partial_{x}^{i_{1}}\rho\cdots \partial_{x}^{i_{k}}\rho + \sum_{\substack{0 \le j_{1}, j_{2}\le s+1, \\j_{1}+j_{2}=s+1}} C_{s, j_{1}, j_{2}}\partial_{x}^{j_{1}}B\partial_{x}^{j_{2}}B \right ) \\
& \le & \int_{\TT}-\frac{1}{2}\rho^{\gamma-1}\big(\partial_{x}^{s+1}\rho\big)^{2} + \frac{1}{2}\rho^{1-\gamma}\left (\sum_{\substack{k \ge 1, \\ 1 \le i_{1}, \ldots, i_{k} \le s, \\ i_{1}+\cdots +i_{k}=s+1, \\ \alpha+k=\gamma}}C_{\gamma, s, \alpha, i_{1}, \ldots, i_{k}}\rho^{\alpha}\partial_{x}^{i_{1}}\rho\cdots \partial_{x}^{i_{k}}\rho + \sum_{\substack{0 \le j_{1}, j_{2}\le s+1, \\j_{1}+j_{2}=s+1}} C_{s, j_{1}, j_{2}}\partial_{x}^{j_{1}}B\partial_{x}^{j_{2}}B \right )^{2} \\
& \le &  \int_{\TT}-\frac{1}{2}\rho^{\gamma-1}\big(\partial_{x}^{s+1}\rho\big)^{2}+\sum_{\substack{k \ge 1, \\ 1 \le i_{1}, \ldots, i_{k} \le s, \\ i_{1}+\cdots +i_{k}=s+1, \\ \alpha+k=\gamma}}C_{\gamma, s, \alpha, i_{1}, \ldots, i_{k}}\rho^{2\alpha-\gamma+1}\big(\partial_{x}^{i_{1}}\rho\big)^{2}\cdots \big(\partial_{x}^{i_{k}}\rho\big)^{2} \\
&& + \sum_{\substack{0 \le j_{1}, j_{2}\le s+1, \\j_{1}+j_{2}=s+1}} C_{\gamma, s, j_{1}, j_{2}}\rho^{1-\gamma}\big(\partial_{x}^{j_{1}}B\big)^{2}\big(\partial_{x}^{j_{2}}B\big)^{2}.
\end{IEEEeqnarray*}}
Similarly as in the proof of Lemma \ref{20260110lem1}, for any $(\alpha, i_{1}, \ldots, i_{k})$ satisfying $k \ge 1$, $1 \le i_{1}, \ldots, i_{k}\le s$, $i_{1}+\cdots +i_{k}=s+1$, and $\alpha+k=\gamma$, we have \begin{IEEEeqnarray*}{rCl}
\int_{\TT} \rho^{2\alpha-\gamma+1}\big(\partial_{x}^{i_{1}}\rho\big)^{2}\cdots \big(\partial_{x}^{i_{k}}\rho\big)^{2} &\le& \left ( \eta^{-1}+\|\rho\|_{L^{\infty}}\right )^{|2\alpha-\gamma+1|}\|\partial_{x}^{s}\rho\|_{L^{2}}^{2}\|\partial_{x}^{\lfloor\frac{s+1}{2}\rfloor}\rho\|_{L^{\infty}}^{2}\|\rho\|_{W^{\lfloor\frac{s+1}{2}\rfloor, \infty}}^{2k-4} \\
&\le& \|\partial_{x}^{s+1}\rho\|_{L^{2}}^{2} \|\partial_{x}^{\lfloor\frac{s+3}{2}\rfloor}\rho\|_{L^{2}}^{2}\left ( \eta^{-1}+\|\rho\|_{H^{\lfloor\frac{s+3}{2}\rfloor}}\right )^{C_{\gamma, s}},
\end{IEEEeqnarray*}
and for any $(j_{1}, j_{2})$ satisfying $0 \le j_{1}, j_{2} \le s+1$ and $j_{1}+j_{2}=s+1$, we have that \begin{IEEEeqnarray*}{rCl}
\int_{\TT} \rho^{1-\gamma}\big(\partial_{x}^{j_{1}}B\big)^{2}\big(\partial_{x}^{j_{2}}B\big)^{2} &\le& \left ( \eta^{-1}+\|\rho\|_{L^{\infty}}\right )^{|\gamma-1|} \left (\|\partial_{x}^{s}B\|_{L^{2}}^{2}\|\partial_{x}^{\lfloor \frac{s+1}{2}\rfloor}B\|_{L^{\infty}}^{2} + \|\partial_{x}^{s+1}B\|_{L^{2}}^{2}\|B\|_{L^{\infty}}^{2}\right ) \\
& \le & \|\partial_{x}^{s+1}B\|_{L^{2}}^{2}\left (\|\partial_{x}^{\lfloor \frac{s+3}{2}\rfloor}B\|_{L^{2}}^{2}+\|B\|_{L^{\infty}}^{2}\right ) \left ( \eta^{-1}+\|\rho\|_{H^{\lfloor\frac{s+3}{2}\rfloor}}\right )^{C_{\gamma}}.
\end{IEEEeqnarray*}
Hence, \begin{IEEEeqnarray*}{rCl}
\IEEEeqnarraymulticol{3}{l}{\partial_{t}\left (\|\partial_{x}^{s}\rho\|_{L^{2}}^{2}\right ) + \|\rho^{\frac{\gamma-1}{2}}\partial_{x}^{s+1}\rho\|_{L^{2}}^{2}} \\
&\le& C_{\gamma, s}\left ( \|\partial_{x}^{s+1}\rho\|_{L^{2}}^{2}+\|\partial_{x}^{s+1}B\|_{L^{2}}^{2}\right ) \left ( \|\partial_{x}^{\lfloor\frac{s+3}{2}\rfloor}\rho\|_{L^{2}}+\|B\|_{H^{\lfloor\frac{s+3}{2}\rfloor}}\right )^{2} \left ( \eta^{-1} + \|\rho\|_{H^{\lfloor\frac{s+3}{2}\rfloor}}+\|B\|_{H^{\lfloor\frac{s+3}{2}\rfloor}}\right )^{C_{\gamma, s}}.
\end{IEEEeqnarray*} By Poincar\'e's inequality, \[
\|\partial_{x}^{s}\rho\|_{L^{2}} \le C\|\partial_{x}^{s+1}\rho\|_{L^{2}} \le C\|\rho^{\frac{\gamma-1}{2}}\partial_{x}^{s+1}\rho\|_{L^{2}}\|\rho^{-1}\|_{L^{\infty}}^{\frac{\gamma-1}{2}} \le C\|\rho^{\frac{\gamma-1}{2}}\partial_{x}^{s+1}\rho\|_{L^{2}}\eta^{-C_{\gamma}}.
\]Hence, for some constant $\sigma_{1}=\sigma_{1}(\gamma, B_{0}, \rho_{t=0}, B_{t=0})>0$, we have {\small\begin{IEEEeqnarray*}{rCl}
\IEEEeqnarraymulticol{3}{l}{\partial_{t}\left (\|\partial_{x}^{s}\rho\|_{L^{2}}^{2}\right )+\sigma_{1}\|\partial_{x}^{s}\rho\|_{L^{2}}^{2}} \\
&\le& C_{\gamma, s}\left ( \|\partial_{x}^{s+1}\rho\|_{L^{2}}^{2}+\|\partial_{x}^{s+1}B\|_{L^{2}}^{2}\right ) \left ( \|\partial_{x}^{\lfloor\frac{s+3}{2}\rfloor}\rho\|_{L^{2}}+\|B\|_{H^{\lfloor\frac{s+3}{2}\rfloor}}\right )^{2} \left ( \eta^{-1} + \|\rho\|_{H^{\lfloor\frac{s+3}{2}\rfloor}}+\|B\|_{H^{\lfloor\frac{s+3}{2}\rfloor}}\right )^{C_{\gamma, s}}. \yesnumber \label{20260111eq3}
\end{IEEEeqnarray*}}
Similarly,
\begin{IEEEeqnarray*}{rCl}
\IEEEeqnarraymulticol{3}{l}{\partial_{t}\frac{1}{2}\int_{\TT} |\partial_{x}^{s}B|^{2}} \\
&=& \int_{\TT} -\partial_{x}^{s+1}B \partial_{x}^{s}\left ( \frac{B}{\rho}(\rho^{\gamma-1}\partial_{x}\rho+B\partial_{x}B) + \frac{B_{0}^{2}}{\rho}\partial_{x}B\right ) \\
&=& \int_{\TT} -\frac{B^{2}+B_{0}^{2}}{\rho}\big(\partial_{x}^{s+1}B\big)^{2} -\rho^{\gamma-2}B \partial_{x}^{s+1}\rho\partial_{x}^{s+1}B\\
&& + \partial_{x}^{s+1}B \sum_{\substack{k, \ell \ge 0, \\ 1 \le i_{1}, \ldots, i_{k} \le s, \\ 0 \le j_{1}, \ldots, j_{\ell} \le s, \\ i_{1}+\cdots +i_{k}+j_{1}+\cdots +j_{\ell}=s+1, \\ \alpha+k+\frac{\gamma}{2}\ell=\frac{3}{2}\gamma-1}} C_{\gamma, s, \alpha, i_{1}, \ldots, i_{k}, j_{1}, \ldots, j_{\ell}} \rho^{\alpha}\partial_{x}^{i_{1}}\rho \cdots \partial_{x}^{i_{k}}\rho\partial_{x}^{j_{1}}B\cdots \partial_{x}^{j_{\ell}}B \\
&& + B_{0}^{2}\partial_{x}^{s+1}B \sum_{\substack{k, \ell \ge 0, \\ 1 \le i_{1}, \ldots,  i_{k} \le s, \\ 0 \le j_{1}, \ldots, j_{\ell} \le s, \\ i_{1}+\cdots +i_{k}+j_{1}+\cdots +j_{\ell}=s+1, \\ \alpha+k=-1, \ell=1}}C_{\gamma, s, \alpha, i_{1}, \ldots, i_{k}, j_{1}, \ldots, j_{\ell}}\rho^{\alpha}\partial_{x}^{i_{1}}\rho\cdots \partial_{x}^{i_{k}}\rho\partial_{x}^{j_{1}}B\cdots \partial_{x}^{j_{\ell}}B \\
& \le & -\int_{\TT} \frac{B_{0}^{2}}{2\rho}\big(\partial_{x}^{s+1}B\big)^{2}\\
&& +C_{\gamma, B_{0}, s} \left ( \|\partial_{x}^{s+1}\rho\|_{L^{2}}^{2}+\|\partial_{x}^{s+1}B\|_{L^{2}}^{2}\right ) \left ( \|\partial_{x}^{\lfloor\frac{s+3}{2}\rfloor}\rho\|_{L^{2}}+\|B\|_{H^{\lfloor\frac{s+3}{2}\rfloor}}\right )^{2}\left ( \eta^{-1} + \|\rho\|_{H^{\lfloor\frac{s+3}{2}\rfloor}}+\|B\|_{H^{\lfloor \frac{s+3}{2}\rfloor}}\right )^{C_{\gamma, s}},
\end{IEEEeqnarray*} and again applying Poincar\'e's inequality, we conclude that {\small\begin{IEEEeqnarray*}{rCl}
\IEEEeqnarraymulticol{3}{l}{\partial_{t}\left ( \|\partial_{x}^{s}B\|_{L^{2}}^{2}\right ) + \sigma_{2} \|\partial_{x}^{s}B\|_{L^{2}}^{2}} \\
&\le& C_{\gamma, B_{0}, s} \left ( \|\partial_{x}^{s+1}\rho\|_{L^{2}}^{2}+\|\partial_{x}^{s+1}B\|_{L^{2}}^{2}\right ) \left ( \|\partial_{x}^{\lfloor\frac{s+3}{2}\rfloor}\rho\|_{L^{2}}+\|B\|_{H^{\lfloor\frac{s+3}{2}\rfloor}}\right )^{2}\left ( \eta^{-1} + \|\rho\|_{H^{\lfloor\frac{s+3}{2}\rfloor}}+\|B\|_{H^{\lfloor \frac{s+3}{2}\rfloor}}\right )^{C_{\gamma, s}} \yesnumber \label{20260111eq4}
\end{IEEEeqnarray*}}
for some constant $\sigma_{2}=\sigma_{2}(\gamma, B_{0}, \rho_{t=0}, B_{t=0})>0$.

\subsection{Proof of Theorem \ref{20260110thm2} for the case $\overline{B}=0$} \label{subsec3.2}

In this proof, $A_{1}(\cdot), A_{2}(\cdot), \dots$ denote positive constants whose values remain fixed throughout. The variables within the parentheses indicate the parameters on which each constant depends. 

Assume that initial data $(\rho_{t=0}, B_{t=0})$ satisfy \[
\|\rho_{t=0}-\overline{\rho}\|_{H^{s}} + \|B_{t=0}\|_{H^{s}} \le \epsilon, \quad \fint_{\TT}\rho_{t=0}=\overline{\rho}, \quad \fint_{\TT}B_{t=0}=0,
\] and let $(\rho, B)$ be a smooth solution to the system \eqref{20260110eq2} defined on a maximal time interval $[0, T)$. ($T$ may be equal to $\infty$.) Here, $\epsilon>0$ will be chosen later. Assume that $\epsilon$ is small enough so that \[
\frac{1}{2}\overline{\rho} \le \rho_{t=0} \le 2\overline{\rho}, \quad |B_{t=0}| \le 1.
\] Then $W(\rho_{t=0}, B_{t=0})$ and $Z(\rho_{t=0}, B_{t=0})$ has a uniform upper bound depending only on $\gamma, B_{0}$, and $\overline{\rho}$ (but not on $\rho_{t=0}$ and $B_{t=0}$), and so by repeating the argument in Section \ref{subsubsec2.2.3} we see that there exists $\eta=\eta(\gamma, B_{0}, \overline{\rho}, \overline{B})>0$, not depending on $\rho_{t=0}$ and $B_{t=0}$, satisfying $\rho, \rho^{-1}, |B| \le \eta^{-1}$. Define $\sigma = \frac{1}{2}\min\{\sigma_{1}, \sigma_{2}\}$, where $\sigma_{1}$ and $\sigma_{2}$ are defined in \eqref{20260111eq3} and \eqref{20260111eq4}. Recall that the dependence of $\sigma_{1}$ and $\sigma_{2}$ on $(\rho_{t=0}, B_{t=0})$ was inherited from $\eta$. Since we have chosen $\eta$ to depend on $\overline{\rho}$ rather than $(\rho_{t=0}, B_{t=0})$, we can ensure that $\sigma$ depends only on $\gamma, B_{0}$, and $\overline{\rho}$.

Let $A_{1}>1$ be a constant, to be chosen later. Let $s \ge 6$ be a given integer. We claim that the following four inequalities hold for all time $t \in [0, T)$:
\begin{itemize}
\item[(C1)] $\|\rho(t)-\overline{\rho}\|_{H^{s}} \le A_{1}\epsilon$,
\item[(C2)] $\|B(t)\|_{H^{s}} \le A_{1}\epsilon$,
\item[(C3)] $\|\rho(t)-\overline{\rho}\|_{H^{s-2}} \le A_{1}\epsilon e^{-\sigma t}$,
\item[(C4)] $\|B(t)\|_{H^{s-2}} \le A_{1}\epsilon e^{-\sigma t}$,
\end{itemize} assuming that the following four inequalities hold for all time $t \in [0, T)$:
\begin{itemize}
\item[(A1)] $\|\rho(t)-\overline{\rho}\|_{H^{s}} \le 2A_{1}\epsilon$,
\item[(A2)] $\|B(t)\|_{H^{s}} \le 2A_{1}\epsilon$,
\item[(A3)] $\|\rho(t)-\overline{\rho}\|_{H^{s-2}} \le 2A_{1}\epsilon e^{-\sigma t}$,
\item[(A4)] $\|B(t)\|_{H^{s-2}} \le 2A_{1}\epsilon e^{-\sigma t}$.
\end{itemize} 

To prove (C1) and (C2), we start with Poincar\'e's inequality and Corollary \ref{20260110cor1}:
{\small \begin{IEEEeqnarray*}{rCl}
\IEEEeqnarraymulticol{3}{l}{\|\rho(t)-\overline{\rho}\|^{2}_{H^{s}} + \|B(t)\|_{H^{s}}^{2}} \\
&\le& C_{s} \|\partial_{x}^{s}\rho(t)\|_{L^{2}}^{2} + C_{s}\|\partial_{x}^{s}B(t)\|_{L^{2}}^{2} \\
& \le & C_{s} \left (\|\partial_{x}^{s}\rho_{t=0}\|_{L^{2}}^{2} + \|\partial_{x}^{s}B_{t=0}\|_{L^{2}}^{2}\right ) \left (\eta^{-1} +\|\rho_{t=0}\|_{H^{\lfloor\frac{s+3}{2}\rfloor}}+ \|\rho(t)\|_{H^{\lfloor\frac{s+3}{2}\rfloor}}\right )^{C_{\gamma}} \\
&& \exp\left ( C_{\gamma, B_{0}, s}\int_{0}^{t} \left ( \|\partial_{x}^{\lfloor\frac{s+3}{2}\rfloor}\rho(\tau)\|_{L^{2}}+\|\partial_{x}^{\lfloor\frac{s+3}{2}\rfloor}B(\tau)\|_{L^{2}}\right )\left ( \eta^{-1} +\|\rho(\tau)\|_{H^{\lfloor\frac{s+3}{2}\rfloor}}+\|B(\tau)\|_{H^{\lfloor\frac{s+3}{2}\rfloor}} \right )^{C_{\gamma, s}}\dd\tau\right ). \yesnumber \label{20260118eq1}
\end{IEEEeqnarray*}}
Observe that by Poincar\'e's inequality and assumptions (A1)-(A4),
{\small \begin{IEEEeqnarray*}{rCl}
\|\rho_{t=0}\|_{H^{\lfloor\frac{s+3}{2}\rfloor}} &\le & \|\rho_{t=0}\|_{H^{s}} \le C_{s}\|\rho_{t=0}-\overline{\rho}\|_{H^{s}}+C_{s}\|\rho_{t=0}\|_{L^{2}} \le C_{s}A_{1}\epsilon + C_{s}\|\rho_{t=0}\|_{L^{\infty}} \le C_{s}A_{1}\epsilon + C_{s}\eta^{-1}, \yesnumber \label{20260118eq2} \\
\|\rho(t)\|_{H^{\lfloor\frac{s+3}{2}\rfloor}} & \le & C_{s}A_{1}\epsilon + C_{s}\eta^{-1}, \yesnumber \label{20260118eq3}\\
\|\partial_{x}^{\lfloor\frac{s+3}{2}\rfloor}\rho(\tau)\|_{L^{2}} &\le& C_{s} \|\partial_{x}^{s-2}\rho(\tau)\|_{L^{2}} \le C_{s} \|\rho(\tau)-\overline{\rho}\|_{H^{s-2}} \le C_{s} A_{1}\epsilon e^{-\sigma t}, \yesnumber \label{20260118eq4}\\
\|\partial_{x}^{\lfloor\frac{s+3}{2}\rfloor}B\|_{L^{2}} &\le & C_{s}A_{1}\epsilon e^{-\sigma t}. \yesnumber \label{20260118eq5}
\end{IEEEeqnarray*}}
(Note that we have used $s-2 \ge \lfloor\frac{s+3}{2}\rfloor$, which holds for $s \ge 6$.) Applying these inequalities to \eqref{20260118eq1} gives
\begin{IEEEeqnarray*}{rCl}
\IEEEeqnarraymulticol{3}{l}{\|\rho(t)-\overline{\rho}\|_{H^{s}}^{2}+\|B(t)\|_{H^{s}}^{2}} \\
& \le & C_{s}\epsilon^{2} \left (C_{s}A_{1}\epsilon + C_{s}\eta^{-1} \right )^{C_{\gamma}} \exp \left (C_{\gamma, B_{0}, s}\int_{0}^{t} C_{s}A_{1}\epsilon e^{-\sigma \tau} \left (C_{s}A_{1}\epsilon+C_{s}\eta^{-1}\right )^{C_{\gamma, s}}\dd\tau \right ) \\
& \le & C_{\gamma, s}\epsilon^{2}\left (\eta^{-1}+A_{1}\epsilon\right )^{C_{\gamma}}\exp\left ( C_{\gamma, B_{0}, s}\int_{0}^{t} A_{1}\epsilon e^{-\sigma\tau}\left (\eta^{-1}+A_{1}\epsilon\right )^{C_{\gamma, s}}\dd\tau \right ) \\
& \le & A_{2}(\gamma, s)(\eta^{-1}+A_{1}\epsilon)^{A_{3}(\gamma)}\epsilon^{2} \exp \left ( A_{4}(\gamma, B_{0}, s)A_{1}\epsilon \frac{1}{\sigma} \left (\eta^{-1}+A_{1}\epsilon\right )^{A_{5}(\gamma, s)}\right ). \yesnumber \label{20260112eq1}
\end{IEEEeqnarray*}
Now choose $A_{1}=A_{1}(\gamma, B_{0}, \overline{\rho}, s)$ large enough such that \[
A_{2}(2\eta^{-1})^{A_{3}} \exp \left ( A_{4}\eta^{-1}\frac{1}{\sigma}(2\eta^{-1})^{A_{5}}\right ) \le A_{1}^{2},
\] and choose $\epsilon=\epsilon( \gamma, B_{0}, \overline{\rho}, s)$ small enough such that \[
A_{1}\epsilon < \eta^{-1}, \quad 2A_{1}\epsilon < \frac{1}{2}\overline{\rho}.
\] Then, \eqref{20260112eq1} reads as \[
\|\rho(t)-\overline{\rho}\|_{H^{s}}^{2}+\|B(t)\|_{H^{s}}^{2} \le A_{1}^{2}\epsilon^{2},
\] and thus (C1) and (C2) hold.

To prove (C3), observe that from \eqref{20260111eq3} we have
\begin{IEEEeqnarray*}{rCl}
\IEEEeqnarraymulticol{3}{l}{\|\partial_{x}^{s-2}\rho\|_{L^{2}}^{2}} \\
&\le& \|\partial_{x}^{s-2}\rho_{t=0}\|_{L^{2}}^{2}e^{-2\sigma t} \\
&& + e^{-2\sigma t}\int_{0}^{t} e^{2\sigma \tau} C_{\gamma, s}\left ( \|\partial_{x}^{s-1}\rho(\tau)\|_{L^{2}}^{2}+\|\partial_{x}^{s-1}B(\tau)\|_{L^{2}}^{2}\right ) \left ( \|\partial_{x}^{\lfloor\frac{s+1}{2}\rfloor}\rho(\tau)\|_{L^{2}}+\|B(\tau)\|_{H^{\lfloor\frac{s+1}{2}\rfloor}}\right )^{2} \\
&& \qquad\left ( \eta^{-1} + \|\rho(\tau)\|_{H^{\lfloor\frac{s+1}{2}\rfloor}}+\|B(\tau)\|_{H^{\lfloor\frac{s+1}{2}\rfloor}}\right )^{C_{\gamma, s}}\dd\tau \\
& \le & \epsilon^{2}e^{-2\sigma t} + e^{-2\sigma t}\int_{0}^{t}e^{2\sigma \tau}C_{\gamma, s}\left ( \|\rho(\tau)-\overline{\rho}\|_{H^{s}}\|\rho(\tau)-\overline{\rho}\|_{H^{s-2}} + \|B(\tau)\|_{H^{s}}\|B(\tau)\|_{H^{s-2}}\right ) \\
&& \qquad \qquad \qquad \qquad \left ( \|\rho(\tau)-\overline{\rho}\|_{H^{s-2}} + \|B(\tau)\|_{H^{s-2}}\right )^{2} \left ( \eta^{-1} + A_{1}\epsilon \right )^{C_{ \gamma, s}}\dd\tau \\
& \le & \epsilon^{2}e^{-2\sigma t}+e^{-2\sigma t}\int_{0}^{t}C_{\gamma, s} e^{2\sigma \tau} A_{1}^{2}\epsilon^{2}e^{-\sigma \tau} A_{1}^{2}\epsilon^{2}e^{-2\sigma \tau} \eta^{-C_{\gamma, s}}\dd\tau \\
& \le & \epsilon^{2}e^{-2\sigma t}\left ( 1+ \frac{1}{\sigma}A_{6}(\gamma, s)A_{1}^{4}\epsilon^{2}\eta^{-A_{7}(\gamma, s)}\right ).
\end{IEEEeqnarray*}
Thus by Poincar\'e's inequality, \[
\|\rho(t)-\overline{\rho}\|_{H^{s-2}}^{2} \le \epsilon^{2} e^{-2\sigma t} A_{8}(s)\left ( 1+\frac{1}{\sigma}A_{1}^{4}\epsilon^{2}A_{6}(\gamma, s)\eta^{-A_{7}(\gamma, s)}\right ).
\] Now, we could have chosen $A_{1}$ big enough so that \[
A_{1}^{2}>2A_{8},
\] and could have chosen $\epsilon$ small enough so that \[
\frac{1}{\sigma}A_{1}^{4}\epsilon^{2}A_{6}\eta^{-A_{7}} \le 1.
\]
Then, we obtain \[
\|\rho(t)-\overline{\rho}\|_{H^{s-2}} \le A_{1}\epsilon e^{-\sigma t},
\] which is the desired inequality (C3).

Similar argument using \eqref{20260111eq4} instead of \eqref{20260111eq3} gives (C4), possibly after further modifying $A_{1}$ and $\epsilon$.

Hence, (C1)-(C4) follow from (A1)-(A4). We now claim that (A1)-(A4) hold for all $t \in [0, T)$. These inequalities clearly hold at $t=0$, and by continuity, they remain valid on $[0, \delta]$ for some $\delta>0$. Suppose, for the sake of contradiction, that at least one of (A1)-(A4) fails at some time, and let $t_{0}\ge \delta>0$ be the infimum of such times. Continuity then implies that at least one of (A1)-(A4) must become an equality at $t=t_{0}$. However, our previous argument ensures that (C1)-(C4) hold for all $t<T$, in particular at $t=t_{0}$, which preculdes the possibility of equality in (A1)-(A4). This contradiction shows that (A1)-(A4) hold for all $t \in [0, T)$.

Moreover, as $s \ge 6$, (A1) and (A2) show that $\|\rho\|_{H^{4}}$ and $\|B\|_{H^{4}}$ are bounded by $2A_{1}\epsilon$, and thus by Theorem \ref{20260110thm1} we conclude that the maximal existence time $T$ is equal to infinity. Hence the solution $(\rho, B)$ is global in time. Now the decay estimates (A3) and (A4) show that $(\rho(t), B(t)) \to (\overline{\rho}, 0)$ in $H^{s-2}$, and this finishes the proof of Theorem \ref{20260110thm2} for the case $\overline{B}=0$. 

\subsection{Decay of coupled $H^{s}$ norm}\label{subsec3.3}

From now on, assume $\overline{B}\neq 0$.

Define $\zeta_{1}=\zeta_{1}(\gamma, B_{0}, \overline{\rho}, \overline{B})$, $\zeta_{2}=\zeta_{2}(\gamma, B_{0}, \overline{\rho}, \overline{B})$ be two solutions to the quadratic equation \[
\overline{B} \zeta^{2} + \frac{\overline{B}^{2}+B_{0}^{2}-\overline{\rho}^{\gamma}}{\overline{\rho}}\zeta - \overline{\rho}^{\gamma-2}\overline{B}=0,
\] namely, \[
\{\zeta_{1}, \zeta_{2}\} = \frac{-\left (\overline{B}^{2}+B_{0}^{2}-\overline{\rho}^{\gamma}\right ) \pm \sqrt{\left (\overline{B}^{2}+B_{0}^{2}-\overline{\rho}^{\gamma}\right )^{2}+4\overline{B}^{2}\overline{\rho}^{\gamma}}}{2\overline{B}\overline{\rho}}.
\]
Then for $\zeta \in \{\zeta_{1}, \zeta_{2}\}$, we have
{\small \begin{IEEEeqnarray*}{rCl}
\IEEEeqnarraymulticol{3}{l}{\partial_{t}\frac{1}{2}\int_{\TT}\left |\partial_{x}^{s}\left (\zeta \rho + B\right )\right |^{2}} \\
&=&\int_{\TT} \left ( \zeta \partial_{x}^{s}\rho + \partial_{x}^{s}B\right )\left ( \zeta \partial_{x}^{s+2}\left ( \frac{1}{\gamma}\rho^{\gamma}+\frac{1}{2}B^{2}\right )+\partial_{x}^{s+1}\left ( \frac{B^{2}+B_{0}^{2}}{\rho}\partial_{x}B + \rho^{\gamma-2}B\partial_{x}\rho\right )\right ) \\
&=&-\int_{\TT} \left ( \zeta \partial_{x}^{s+1}\rho + \partial_{x}^{s+1}B\right )\left ( \zeta \partial_{x}^{s+1}\left ( \frac{1}{\gamma}\rho^{\gamma}+\frac{1}{2}B^{2}\right )+\partial_{x}^{s}\left ( \frac{B^{2}+B_{0}^{2}}{\rho}\partial_{x}B + \rho^{\gamma-2}B\partial_{x}\rho\right )\right ) \\
&=&\int_{\TT} -\left ( \zeta \partial_{x}^{s+1}\rho + \partial_{x}^{s+1}B\right )\left ( \zeta \rho^{\gamma-1}\partial_{x}^{s+1}\rho + \zeta B\partial_{x}^{s+1}B + \frac{B^{2}+B_{0}^{2}}{\rho}\partial_{x}^{s+1}B + \rho^{\gamma-2}B\partial_{x}^{s+1}\rho\right ) \\
&& + \left ( \zeta \partial_{x}^{s+1}\rho + \partial_{x}^{s+1}B\right )[S] \\
&=& \int_{\TT} -\left ( \zeta \partial_{x}^{s+1}\rho + \partial_{x}^{s+1}B\right )\left ( \zeta \overline{\rho}^{\gamma-1}\partial_{x}^{s+1}\rho + \zeta \overline{B}\partial_{x}^{s+1}B + \frac{\overline{B}^{2}+B_{0}^{2}}{\overline{\rho}}\partial_{x}^{s+1}B + \overline{\rho}^{\gamma-2}\overline{B}\partial_{x}^{s+1}\rho\right ) \\
&& + \left ( \zeta \partial_{x}^{s+1}\rho + \partial_{x}^{s+1}B\right )[S-\zeta(\rho^{\gamma-1}-\overline{\rho}^{\gamma-1})\partial_{x}^{s+1}\rho -\zeta(B-\overline{B})\partial_{x}^{s+1}B-(B^{2}+B_{0}^{2})(\rho^{-1}-\overline{\rho}^{-1})\partial_{x}^{s+1}B \\
&&\qquad\qquad\qquad\qquad\qquad -(B-\overline{B})(B+\overline{B})\overline{\rho}^{-1}\partial_{x}^{s+1}B-\rho^{\gamma-2}(B-\overline{B})\partial_{x}^{s+1}\rho-(\rho^{\gamma-2}-\overline{\rho}^{\gamma-2})\overline{B}\partial_{x}^{s+1}\rho] \\
&=& \int_{\TT}- \left ( \zeta \overline{B} + \frac{\overline{B}^{2}+B_{0}^{2}}{\overline{\rho}}\right )\left ( \zeta \partial_{x}^{s+1}\rho + \partial_{x}^{s+1}B\right )^{2} \\
&& + \left ( \zeta \partial_{x}^{s+1}\rho + \partial_{x}^{s+1}B\right )[S-\zeta(\rho^{\gamma-1}-\overline{\rho}^{\gamma-1})\partial_{x}^{s+1}\rho -\zeta(B-\overline{B})\partial_{x}^{s+1}B-(B^{2}+B_{0}^{2})(\rho^{-1}-\overline{\rho}^{-1})\partial_{x}^{s+1}B \\
&&\qquad\qquad\qquad\qquad \qquad -(B-\overline{B})(B+\overline{B})\overline{\rho}^{-1}\partial_{x}^{s+1}B-\rho^{\gamma-2}(B-\overline{B})\partial_{x}^{s+1}\rho-(\rho^{\gamma-2}-\overline{\rho}^{\gamma-2})\overline{B}\partial_{x}^{s+1}\rho], \yesnumber \label{20260113eq1}
\end{IEEEeqnarray*}}
where $S$ is 
\begin{IEEEeqnarray*}{rCl}
S&=&\left ( \sum_{\substack{k, \ell \ge 0, \\ 1 \le i_{1}, \ldots, i_{k} \le s, \\ 0 \le j_{1}, \ldots, j_{\ell} \le s, \\ i_{1}+\cdots +i_{k}+j_{1}+\cdots +j_{\ell}=s+1, \\ \alpha+k+\frac{\gamma}{2}\ell\in\{\gamma, \frac{3}{2}\gamma-1\}}} C_{\gamma, s, \alpha, i_{1}, \ldots, i_{k}, j_{1}, \ldots, j_{\ell}} \rho^{\alpha}\partial_{x}^{i_{1}}\rho \cdots \partial_{x}^{i_{k}}\rho\partial_{x}^{j_{1}}B\cdots \partial_{x}^{j_{\ell}}B \right .\\
&& \qquad \left .+ B_{0}^{2}\sum_{\substack{k, \ell \ge 0, \\ 1 \le i_{1}, \ldots,  i_{k} \le s, \\ 0 \le j_{1}, \ldots, j_{\ell} \le s, \\ i_{1}+\cdots +i_{k}+j_{1}+\cdots +j_{\ell}=s+1, \\ \alpha+k=-1, \ell=1}}C_{\gamma, s, \alpha, i_{1}, \ldots, i_{k}, j_{1}, \ldots, j_{\ell}}\rho^{\alpha}\partial_{x}^{i_{1}}\rho\cdots \partial_{x}^{i_{k}}\rho\partial_{x}^{j_{1}}B\cdots \partial_{x}^{j_{\ell}}B\right ).
\end{IEEEeqnarray*}
Observe that
\begin{IEEEeqnarray*}{rCl}
\|S\|_{L^{2}} & \le & C_{\gamma, B_{0}, s}\left ( \|\partial_{x}^{s}\rho\|_{L^{2}}+\|\partial_{x}^{s}B\|_{L^{2}}\right )\left ( \|\partial_{x}^{\lfloor\frac{s+3}{2}\rfloor}\rho\|_{L^{2}}+\|\partial_{x}^{\lfloor\frac{s+3}{2}\rfloor}B\|_{L^{2}}\right )\\
&& \left ( \eta^{-1} + \|\rho\|_{H^{\lfloor\frac{s+3}{2}\rfloor}}+\|B\|_{H^{\lfloor \frac{s+3}{2}\rfloor}}\right )^{C_{\gamma, s}}, \\
\|(\rho^{\gamma-1}-\overline{\rho}^{\gamma-1})\partial_{x}^{s+1}\rho\|_{L^{2}} &\le& \|\partial_{x}^{s+1}\rho\|_{L^{2}}\|\partial_{x}(\rho^{\gamma-1})\|_{L^{1}} \\
& \le & C_{\gamma}\|\partial_{x}^{s+1}\rho\|_{L^{2}}\|\partial_{x}\rho\|_{L^{2}}\left ( \eta^{-1}+\|\rho\|_{H^{\lfloor\frac{s+3}{2}\rfloor}}\right )^{C_{\gamma}}, \\
\|(B-\overline{B})\partial_{x}^{s+1}B\|_{L^{2}} &\le& \|\partial_{x}^{s+1}B\|_{L^{2}}\|\partial_{x}B\|_{L^{1}} \\
& \le & C\|\partial_{x}^{s+1}B\|_{L^{2}}\|\partial_{x}B\|_{L^{2}}, \\
\|(B^{2}+B_{0}^{2})(\rho^{-1}-\overline{\rho}^{-1})\partial_{x}^{s+1}B\|_{L^{2}} &\le & C_{B_{0}} \|\partial_{x}^{s+1}B\|_{L^{2}} \|\partial_{x}\rho\|_{L^{2}} \left ( \eta^{-1} + \|\rho\|_{H^{\lfloor\frac{s+3}{2}\rfloor}}+\|B\|_{H^{\lfloor \frac{s+3}{2}\rfloor}}\right )^{C_{\gamma}}, \\
\|(B-\overline{B})(B+\overline{B})\overline{\rho}^{-1}\partial_{x}^{s+1}B\|_{L^{2}} &\le& C \|\partial_{x}^{s+1}B\|_{L^{2}}\|\partial_{x}B\|_{L^{2}}\left ( \eta^{-1} + \|\rho\|_{H^{\lfloor\frac{s+3}{2}\rfloor}}+\|B\|_{H^{\lfloor \frac{s+3}{2}\rfloor}}\right )^{C_{\gamma}}, \\
\|\rho^{\gamma-2}(B-\overline{B})\partial_{x}^{s+1}\rho\|_{L^{2}} &\le& \|\partial_{x}^{s+1}\rho\|_{L^{2}} \|\partial_{x}B\|_{L^{2}} \left ( \eta^{-1} + \|\rho\|_{H^{\lfloor\frac{s+3}{2}\rfloor}}\right )^{C_{\gamma}}, \\
\|(\rho^{\gamma-2}-\overline{\rho}^{\gamma-2})\overline{B}\partial_{x}^{s+1}\rho\|_{L^{2}} &\le& \|\partial_{x}^{s+1}\rho\|_{L^{2}} \|\partial_{x}\rho\|_{L^{2}} \left ( \eta^{-1} + \|\rho\|_{H^{\lfloor\frac{s+3}{2}\rfloor}}+\|B\|_{H^{\lfloor \frac{s+3}{2}\rfloor}}\right )^{C_{\gamma}}.
\end{IEEEeqnarray*}
Here, we have used the following: if $f(x_{0})=0$ for some $x_{0} \in \TT$, then \[
\|f\|_{L^{\infty}} \le \|\partial_{x}f\|_{L^{1}} \le C\|\partial_{x}f\|_{L^{2}},
\] which follows from integration by parts and Jensen's inequality.

Applying these series of inequalities and Young's inequality to \eqref{20260113eq1} gives
\begin{IEEEeqnarray*}{rCl}
\IEEEeqnarraymulticol{3}{l}{\partial_{t}\|\partial_{x}^{s}(\zeta \rho + B)\|_{L^{2}}^{2} + 2\sigma \|\partial_{x}^{s}(\zeta \rho + B)\|_{L^{2}}^{2}} \\
& \le & C_{\gamma, B_{0}, \overline{\rho}, \overline{B}, s} \left (\|\partial_{x}^{s+1}\rho\|_{L^{2}}^{2}+\|\partial_{x}^{s+1}B\|_{L^{2}}^{2}\right )\left ( \|\partial_{x}^{\lfloor\frac{s+3}{2}\rfloor}\rho\|_{L^{2}}+\|\partial_{x}^{\lfloor\frac{s+3}{2}\rfloor}B\|_{L^{2}}\right )^{2} \\
&& \left (\eta^{-1}+\|\rho\|_{H^{\lfloor\frac{s+3}{2}\rfloor}}+\|B\|_{H^{\lfloor\frac{s+3}{2}\rfloor}}\right )^{C_{\gamma, s}}, \yesnumber \label{20260117eq3}
\end{IEEEeqnarray*}
where $\sigma = \sigma(\gamma, B_{0}, \overline{\rho}, \overline{B})>0$ is some positive constant.

\subsection{Proof of Theorem \ref{20260110thm2} for the case $\overline{B}\neq 0$} \label{subsec3.4}

Let $\zeta_{1}=\zeta_{1}(\gamma, B_{0}, \overline{\rho}, \overline{B})$, $\zeta_{2}=\zeta_{2}(\gamma, B_{0}, \overline{\rho}, \overline{B})$, $\sigma = \sigma(\gamma, B_{0}, \overline{\rho}, \overline{B})$ be as in the previous section. Again in this proof, $A_{1}(\cdot), A_{2}(\cdot), \ldots$ denote positive constants whose values remain fixed, and the variables inside the parantheses indicate the parameters on which each constant depends. Let $(\rho, B)$ be a smooth solution to the system \eqref{20260110eq2} such that \[
\|\rho_{t=0}-\overline{\rho}\|_{H^{s}} + \|B_{t=0}-\overline{B}\|_{H^{s}} \le \epsilon, \quad \fint_{\TT}\rho_{t=0}=\overline{\rho}, \quad \fint_{\TT}B_{t=0}=\overline{B}.
\] $\epsilon>0$ will be chosen later. As in Section \ref{subsec3.2}, by letting $\epsilon$ small enough, we may choose $\eta=\eta(\gamma, B_{0}, \overline{\rho}, \overline{B})>0$ such that $\rho, \rho^{-1}, |B| \le \eta^{-1}$. Let $A_{9}>1$ be a constant to be chosen later. Again, it is enough to prove the following four inequalities 
\begin{itemize}
\item[(C1')] $\|\rho(t)-\overline{\rho}\|_{H^{s}} \le A_{9}\epsilon$,
\item[(C2')] $\|B(t)-\overline{B}\|_{H^{s}} \le A_{9}\epsilon$,
\item[(C3')] $\|\zeta_{1}\rho(t)+B(t)-\zeta_{1}\overline{\rho}-\overline{B}\|_{H^{s-2}} \le A_{9}\epsilon e^{-\sigma t}$,
\item[(C4')] $\|\zeta_{2}\rho(t)+B(t)-\zeta_{2}\overline{\rho}-\overline{B}\|_{H^{s-2}} \le A_{9}\epsilon e^{-\sigma t}$,
\end{itemize}
for all time $t$, assuming that
\begin{itemize}
\item[(A1')] $\|\rho(t)-\overline{\rho}\|_{H^{s}} \le 2A_{9}\epsilon$,
\item[(A2')] $\|B(t)-\overline{B}\|_{H^{s}} \le 2A_{9}\epsilon$,
\item[(A3')] $\|\zeta_{1}\rho(t)+B(t)-\zeta_{1}\overline{\rho}-\overline{B}\|_{H^{s-2}} \le 2A_{9}\epsilon e^{-\sigma t}$,
\item[(A4')] $\|\zeta_{2}\rho(t)+B(t)-\zeta_{2}\overline{\rho}-\overline{B}\|_{H^{s-2}} \le 2A_{9}\epsilon e^{-\sigma t}$,
\end{itemize}
for all time $t$. Note that for (A3) and (A4) to hold at $t=0$, $A_{9}$ should be chosen to be greater than some constant depending on $\gamma, B_{0}, \overline{\rho}, \overline{B}$.

By Poincar\'e's inequality and Corollary \ref{20260110cor1},
{\small\begin{IEEEeqnarray*}{rCl}
\IEEEeqnarraymulticol{3}{l}{\|\rho(t)-\overline{\rho}\|_{H^{s}}^{2} + \|B(t)-\overline{B}\|_{H^{s}}^{2}} \\
& \le & C_{s}\left (\|\partial_{x}^{s}\rho(t)\|_{L^{2}}^{2} + \|\partial_{x}^{s}B(t)\|_{L^{2}}^{2}\right ) \\
& \le & C_{s}\left (\|\partial_{x}^{s}\rho_{t=0}\|_{L^{2}}^{2} + \|\partial_{x}^{s}B_{t=0}\|_{L^{2}}^{2}\right ) \left ( \eta^{-1} + \|\rho_{t=0}\|_{H^{\lfloor\frac{s+3}{2}\rfloor}} + \|\rho(t)\|_{H^{\lfloor\frac{s+3}{2}\rfloor}}\right )^{C_{\gamma}} \\
&& \exp \left ( C_{\gamma, B_{0}, s} \int_{0}^{t}  \left ( \|\partial_{x}^{\lfloor\frac{s+3}{2}\rfloor}\rho(\tau)\|_{L^{2}} + \|\partial_{x}^{\lfloor\frac{s+3}{2}\rfloor}B(\tau)\|_{L^{2}}\right ) \left ( \eta^{-1}+\|\rho(\tau)\|_{H^{\lfloor \frac{s+3}{2}\rfloor}}+\|B(\tau)\|_{H^{\lfloor \frac{s+3}{2}\rfloor}}\right )^{C_{\gamma, s}}\dd\tau \right ). \yesnumber \label{20260117eq1}
\end{IEEEeqnarray*}}
Using inequalities analogous to \eqref{20260118eq2}--\eqref{20260118eq5} for \eqref{20260117eq1} gives
{\small \begin{IEEEeqnarray*}{rCl}
\IEEEeqnarraymulticol{3}{l}{\|\rho(t)-\overline{\rho}\|_{H^{s}}^{2} + \|B(t)-\overline{B}\|_{H^{s}}^{2}} \\
& \le & C_{\gamma, s}\epsilon^{2} \left ( \eta^{-1} + A_{9}\epsilon \right )^{C_{\gamma}} \exp \left ( C_{\gamma, B_{0}, s} \int_{0}^{t}\left ( \|\rho(\tau)-\overline{\rho}\|_{H^{s-2}}+\|B(\tau)-\overline{B}\|_{H^{s-2}}\right )\left ( \eta^{-1}+A_{9}\epsilon\right )^{C_{\gamma, s}}\dd\tau\right ) \\
& \le & C_{\gamma, s}\epsilon^{2}\left ( \eta^{-1}+A_{9}\epsilon \right )^{C_{\gamma}}\\
&& \exp \left ( C_{\gamma, B_{0}, \overline{\rho}, \overline{B}, s} \int_{0}^{t} \left ( \|\zeta_{1}\rho(\tau)+B(\tau)-\zeta_{1}\overline{\rho}-\overline{B}\|_{H^{s-2}} + \|\zeta_{2}\rho(\tau)+B(\tau)-\zeta_{2}\overline{\rho}-\overline{B}\|_{H^{s-2}}\right ) \left( \eta^{-1}+A_{9}\epsilon\right )^{C_{\gamma, s}}\dd\tau\right ) \\
& \le & C_{\gamma, s}\epsilon^{2}\left (\eta^{-1}+A_{9}\epsilon\right )^{C_{\gamma}} \exp \left ( C_{\gamma, B_{0}, \overline{\rho}, \overline{B}, s} \int_{0}^{t} A_{9}\epsilon e^{-\sigma \tau}\left (\eta^{-1}+A_{9}\epsilon\right )^{C_{ \gamma, s}}\dd\tau\right ) \\
& \le & A_{10}(\gamma, s)\epsilon^{2}\left (\eta^{-1}+A_{9}\epsilon\right )^{A_{11}(\gamma)}\exp \left ( A_{12}(\gamma, B_{0}, \overline{\rho}, \overline{B}, s)\frac{1}{\sigma}A_{9}\epsilon \left (\eta^{-1}+A_{9}\epsilon\right )^{A_{13}( \gamma, s)}\right ). \yesnumber \label{20260117eq2}
\end{IEEEeqnarray*}}
Now choose $A_{9}=A_{9}(\gamma, B_{0}, \overline{\rho}, \overline{B}, s)$ large enough such that \[
 A_{10}\left (2\eta^{-1}\right )^{A_{11}} \exp \left ( A_{12}\frac{1}{\sigma}\eta^{-1}\left (2\eta^{-1}\right )^{A_{13}} \right ) \le A_{9}^{2},
\]
And choose $\epsilon=\epsilon(\gamma, B_{0}, \overline{\rho}, \overline{B}, s)>0$ small enough such that $A_{9}\epsilon < \eta^{-1}$ and $2A_{9}\epsilon < \frac{1}{2}\overline{\rho}$. Then, \eqref{20260117eq2} reads as \[
\|\rho(t)-\overline{\rho}\|_{H^{s}}^{2}+\|B(t)-\overline{B}\|_{H^{s}}^{2} \le A_{9}^{2}\epsilon^{2}.
\] This proves (C1') and (C2'). To prove (C3') and (C4'), observe that for $\zeta \in \{\zeta_{1}, \zeta_{2}\}$ we have from \eqref{20260117eq3} that
\begin{IEEEeqnarray*}{rCl}
\IEEEeqnarraymulticol{3}{l}{\|\partial_{x}^{s-2}\left (\zeta \rho(t)+B(t)\right )\|_{L^{2}}^{2}} \\
& \le & \|\partial_{x}^{s-2}\left (\zeta \rho_{t=0}+B_{t=0}\right )\|_{L^{2}}^{2}e^{-2\sigma t} \\
&& + e^{-2\sigma t}\int_{0}^{t} e^{2\sigma \tau} C_{\gamma, B_{0}, \overline{\rho}, \overline{B}, s}\left (\|\partial_{x}^{s-1}\rho(\tau)\|_{L^{2}}^{2}+\|\partial_{x}^{s-1}B(\tau)\|_{L^{2}}^{2}\right )\left ( \|\partial_{x}^{\lfloor\frac{s+1}{2}\rfloor}\rho(\tau)\|_{L^{2}}+\|\partial_{x}^{\lfloor\frac{s+1}{2}\rfloor}B(\tau)\|_{L^{2}}\right )^{2} \\
&& \qquad \qquad\left (\eta^{-1}+\|\rho(\tau)\|_{H^{\lfloor\frac{s+1}{2}\rfloor}}+\|B(\tau)\|_{H^{\lfloor\frac{s+1}{2}\rfloor}}\right )^{C_{\gamma, s}}\dd\tau. \yesnumber \label{20260117eq4}
\end{IEEEeqnarray*}
Observe that \[
\|\partial_{x}^{s-2}(\zeta \rho_{t=0}+B_{t=0})\|_{L^{2}} \le (|\zeta|+1) \left (\|\partial_{x}^{s-2}\rho_{t=0}\|_{L^{2}}+\|\partial_{x}^{s-2}B_{t=0}\|_{L^{2}}\right ) \le C_{\gamma, B_{0}, \overline{\rho}, \overline{B}, s}\epsilon,
\] and that 
\begin{IEEEeqnarray*}{rCl}
\|\partial_{x}^{s-1}\rho(\tau)\|_{L^{2}}^{2} & \le & C_{s}\|\rho(\tau)-\overline{\rho}\|_{H^{s-1}}^{2} \\
& \le & C_{s}\|\rho(\tau)-\overline{\rho}\|_{H^{s}}\|\rho(\tau)-\overline{\rho}\|_{H^{s-2}} \\
& \le & C_{\gamma, B_{0}, \overline{\rho}, \overline{B}, s} A_{9}\epsilon \left ( \|\zeta_{1}\rho(\tau)+B(\tau)-\zeta_{1}\overline{\rho}-\overline{B}\|_{H^{s-2}} + \|\zeta_{2}\rho(\tau)+B(\tau)-\zeta_{2}\overline{\rho}-\overline{B}\|_{H^{s-2}}\right ) \\
& \le & C_{\gamma, B_{0}, \overline{\rho}, \overline{B}, s}A_{9}^{2}\epsilon^{2} e^{-\sigma \tau}, \\
\|\partial_{x}^{s-1}B(\tau)\|_{L^{2}}^{2} & \le & C_{\gamma, B_{0}, \overline{\rho}, \overline{B}, s}A_{9}^{2}\epsilon^{2}e^{-\sigma \tau}.
\end{IEEEeqnarray*}
Using these inequalities to \eqref{20260117eq4} gives
\begin{IEEEeqnarray*}{rCl}
\IEEEeqnarraymulticol{3}{l}{\|\partial_{x}^{s-2}(\zeta \rho(t)+B(t))\|_{L^{2}}^{2}} \\
& \le & C_{\gamma, B_{0}, \overline{\rho}, \overline{B}, s}\epsilon^{2} e^{-2\sigma t } + e^{-2\sigma t}\int_{0}^{t}e^{2\sigma \tau} C_{\gamma, B_{0}, \overline{\rho}, \overline{B}, s} A_{9}^{2}\epsilon^{2}e^{-\sigma \tau} A_{9}^{2}\epsilon^{2}e^{-2\sigma \tau} \left ( \eta^{-1}+A_{9}\epsilon\right )^{C_{\gamma, s}}\dd\tau \\
& \le & C_{\gamma, B_{0}, \overline{\rho}, \overline{B}, s}\epsilon^{2}e^{-2\sigma t} + e^{-2\sigma t} \frac{1}{\sigma}C_{\gamma, B_{0}, \overline{\rho}, \overline{B}, s} A_{9}^{4}\epsilon^{4}(2\eta^{-1})^{C_{\gamma, s}}.
\end{IEEEeqnarray*}
Thus by Poincar\'e's inequality, \[
\|\zeta \rho(t) + B(t) - \zeta \overline{\rho}-\overline{B}\|_{H^{s-2}}^{2} \le \epsilon^{2}e^{-2\sigma t}(A_{14}(\gamma, B_{0}, \overline{\rho}, \overline{B}, s) + \frac{1}{\sigma} A_{9}^{4}\epsilon^{2}A_{15}(\gamma, B_{0}, \overline{\rho}, \overline{B}, s)).
\]
Hence, if we have chosen $A_{9}$ large enough so that \[
A_{9}^{2} \ge 2A_{14},
\] and have chosen $\epsilon$ small enough so that \[
\frac{1}{\sigma}A_{9}^{4}\epsilon^{2}A_{15}<A_{14},
\] then we would have \[
\|\zeta \rho(t)+B(t)-\zeta \overline{\rho}-\overline{B}\|_{H^{s-2}} \le A_{9}\epsilon e^{-\sigma t},
\] which is (C3') and (C4'). This completes the proof of Theorem \ref{20260110thm2} for the case $\overline{B}\neq 0$.

\appendix

\section{Deferred calculation from Section \ref{subsec2.2}} \label{appendixA}

In this appendix, we verify that $W(\rho, B)$ and $Z(\rho, B)$ defined in Section \ref{subsec2.2} are $C^{2}$ functions in $(\rho, B)$, and compute their first and second derivatives.

Recall that $1<\gamma<2$ and $B_{0}\neq 0$, and recall also the definition of $f$, $w$, and $W$ from Section \ref{subsec2.2}:
\begin{IEEEeqnarray*}{rCl}
f(\rho, B)&:=& \begin{cases}
\infty & B=0, \rho^{\gamma}\le B_{0}^{2}, \\
\frac{4\rho^{\gamma}}{-\big(B^{2}+B_{0}^{2}-\rho^{\gamma}\big)+\sqrt{\big(B^{2}+B_{0}^{2}-\rho^{\gamma}\big)^{2}+4B^{2}\rho^{\gamma}}}& \text{otherwise.}
\end{cases}, \\
w(\rho, B) &:=& \begin{cases}
(B^{2}+B_{0}^{2}) f^{\frac{2}{2-\gamma}} - \int_{f}^{4}\frac{2B_{0}^{2}}{2-\gamma}\frac{1-\frac{1}{4}s}{\left(1-\frac{1}{2}s\right )^{2}}s^{\frac{2}{2-\gamma}}\dd s & f < \infty, \\
\infty & f = \infty.
\end{cases}, \\
W(\rho, B) &:=& e^{-w}.
\end{IEEEeqnarray*}

We first claim that $(\rho, B) \mapsto f(\rho, B)$ is continuous, from $\RR^{+} \times \RR$ to $\RR \cup \{\infty\}$. It is enough to prove that for any $\rho_{\infty}>0$ satisfying $\rho_{\infty}^{\gamma} \le B_{0}^{2}$, one has \[
\lim_{(\rho, B) \to (\rho_{\infty}, 0)} f(\rho, B)=\infty.
\]
This is true since the denominator in the definition of $f$ is positive and converges to zero as $(\rho, B) \to (\rho_{\infty}, 0)$.

Next we claim that $(\rho, B)\mapsto w(\rho, B)$ is continuous. Continuity on $\{f<\infty\}$ is clear. Hence it is enough to prove that as $(\rho, B) \to (\rho_{\infty}, 0)$ where $f(\rho_{\infty}, 0)=\infty$, one has that \[
w(\rho, B) \to w(\rho_{\infty}, 0)=\infty.
\] Indeed, for $(\rho, B)$ with $f(\rho, B)<\infty$,
\begin{IEEEeqnarray*}{rCl}
w&=& B^{2} f^{\frac{2}{2-\gamma}} + B_{0}^{2}\int_{0}^{f}\frac{2}{2-\gamma}s^{\frac{\gamma}{2-\gamma}}\dd s - \int_{f}^{4}\frac{2B_{0}^{2}}{2-\gamma}\frac{1-\frac{1}{4}s}{\left (1-\frac{1}{2}s\right )^{2}} s^{\frac{2}{2-\gamma}}\dd s \\
&=& B^{2}f^{\frac{2}{2-\gamma}}+4^{\frac{2}{2-\gamma}}B_{0}^{2} + \int_{4}^{f} \frac{2B_{0}^{2}}{2-\gamma} \frac{1}{\left (1-\frac{1}{2}s\right )^{2}}s^{\frac{\gamma}{2-\gamma}} \dd s \\
& \ge & \frac{2B_{0}^{2}}{2-\gamma} \int_{4}^{f}\frac{1}{\left (1-\frac{1}{2}s\right )^{2}}s^{\frac{\gamma}{2-\gamma}}\dd s.
\end{IEEEeqnarray*}
So, as $(\rho, B) \to (\rho_{\infty}, 0)$, we have $f(\rho, B) \to \infty$ due to continuity of $f$, and since we have \[
\int_{4}^{\infty}\frac{1}{\left (1-\frac{1}{2}s\right )^{2}}s^{\frac{\gamma}{2-\gamma}}\dd s =\infty,
\] thanks to $1<\gamma<2$, we obtain $w(\rho, B) \to \infty$ as desired. Hence, $(\rho, B) \to w(\rho, B)$ is continuous from $\RR^{+}\times \RR$ to $\RR \cup \{\infty\}$.

Now we show that $(\rho, B) \mapsto W(\rho, B)$ is $C^{2}$ and provide the computation of its derivatives. First, if $B \neq 0$, then
\begin{IEEEeqnarray*}{rCl}
\partial_{\rho}f &=& \frac{\gamma\rho^{\gamma-1}}{B^{2}}\left [ -1 + \frac{B^{2}-B_{0}^{2}+\rho^{\gamma}}{\sqrt{\left (B^{2}+B_{0}^{2}-\rho^{\gamma}\right )^{2}+4B^{2}\rho^{\gamma}}}\right ] \\
&=& -\frac{4 \gamma \rho^{\gamma-1}B_{0}^{2}}{\left (B^{2}+B_{0}^{2}-\rho^{\gamma}\right )^{2}+4B^{2}\rho^{\gamma}}\frac{1}{1+ \frac{B^{2}-B_{0}^{2}+\rho^{\gamma}}{\sqrt{\left (B^{2}+B_{0}^{2}-\rho^{\gamma}\right )^{2}+4B^{2}\rho^{\gamma}}}}, \\
\partial_{B}f &=& \frac{2}{B}\frac{B^{2}+B_{0}^{2}-\rho^{\gamma}}{\sqrt{\left (B^{2}+B_{0}^{2}-\rho^{\gamma}\right )^{2}+4B^{2}\rho^{\gamma}}}-\frac{2}{B^{3}}\left [ B_{0}^{2}-\rho^{\gamma}+\sqrt{\left (B^{2}+B_{0}^{2}-\rho^{\gamma}\right )^{2}+4B^{2}\rho^{\gamma}}\right ] \\
&=& \frac{1}{B^{3}}\left [ B^{2}-B_{0}^{2}+\rho^{\gamma}-\sqrt{\left (B^{2}+B_{0}^{2}-\rho^{\gamma}\right )^{2}+4B^{2}\rho^{\gamma}}\right ] \left [ B^{2}+B_{0}^{2}-\rho^{\gamma}+\sqrt{\left (B^{2}+B_{0}^{2}-\rho^{\gamma}\right )^{2}+4B^{2}\rho^{\gamma}}\right ] \\
&=& \frac{f}{B}\left [ -1+\frac{B^{2}-B_{0}^{2}+\rho^{\gamma}}{\sqrt{\left (B^{2}+B_{0}^{2}-\rho^{\gamma}\right )^{2}+4B^{2}\rho^{\gamma}}}\right ] \\
&=& -\frac{4 B_{0}^{2}Bf}{\left (B^{2}+B_{0}^{2}-\rho^{\gamma}\right )^{2}+4B^{2}\rho^{\gamma}}\frac{1}{1+ \frac{B^{2}-B_{0}^{2}+\rho^{\gamma}}{\sqrt{\left (B^{2}+B_{0}^{2}-\rho^{\gamma}\right )^{2}+4B^{2}\rho^{\gamma}}}}.
\end{IEEEeqnarray*}

If $B=0$ and $\rho^{\gamma}>B_{0}^{2}$, then we have
\begin{IEEEeqnarray*}{rCl}
\partial_{\rho}f &=& \partial_{\rho}\left ( \frac{2\rho^{\gamma}}{\rho^{\gamma}-B_{0}^{2}} \right )\\
&=& -2B_{0}^{2}\frac{\gamma\rho^{\gamma-1}}{\left (\rho^{\gamma}-B_{0}^{2}\right )^{2}} \\
&=& -\frac{4 \gamma \rho^{\gamma-1}B_{0}^{2}}{\left (B^{2}+B_{0}^{2}-\rho^{\gamma}\right )^{2}+4B^{2}\rho^{\gamma}}\frac{1}{1+ \frac{B^{2}-B_{0}^{2}+\rho^{\gamma}}{\sqrt{\left (B^{2}+B_{0}^{2}-\rho^{\gamma}\right )^{2}+4B^{2}\rho^{\gamma}}}}, \\
\partial_{B}f &=& \lim_{h \to 0}\frac{1}{h}\left ( f(\rho, h)-f(\rho, 0)\right ) \\
&=& \lim_{h \to 0}\partial_{B}f(\rho, h) \quad (\because \text{L'Hospital's rule}) \\
&=& -\lim_{B \to 0}\frac{4 B_{0}^{2}Bf}{\left (B^{2}+B_{0}^{2}-\rho^{\gamma}\right )^{2}+4B^{2}\rho^{\gamma}}\frac{1}{1+ \frac{B^{2}-B_{0}^{2}+\rho^{\gamma}}{\sqrt{\left (B^{2}+B_{0}^{2}-\rho^{\gamma}\right )^{2}+4B^{2}\rho^{\gamma}}}} \\
&=& 0 \\
&=& -\frac{4 B_{0}^{2}Bf}{\left (B^{2}+B_{0}^{2}-\rho^{\gamma}\right )^{2}+4B^{2}\rho^{\gamma}}\frac{1}{1+ \frac{B^{2}-B_{0}^{2}+\rho^{\gamma}}{\sqrt{\left (B^{2}+B_{0}^{2}-\rho^{\gamma}\right )^{2}+4B^{2}\rho^{\gamma}}}}.
\end{IEEEeqnarray*}

Hence, we conclude that for any $(\rho, B)$ satisfying $f(\rho, B)<\infty$, $f$ is continuously differentiable and
\begin{IEEEeqnarray*}{rCl}
\partial_{\rho}f &=& -\frac{4 \gamma \rho^{\gamma-1}B_{0}^{2}}{\left (B^{2}+B_{0}^{2}-\rho^{\gamma}\right )^{2}+4B^{2}\rho^{\gamma}}\frac{1}{1+ \frac{B^{2}-B_{0}^{2}+\rho^{\gamma}}{\sqrt{\left (B^{2}+B_{0}^{2}-\rho^{\gamma}\right )^{2}+4B^{2}\rho^{\gamma}}}} \quad \text{if }f<\infty, \yesnumber \label{20260117eq11} \\
\partial_{B}f &=& -\frac{4 B_{0}^{2}Bf}{\left (B^{2}+B_{0}^{2}-\rho^{\gamma}\right )^{2}+4B^{2}\rho^{\gamma}}\frac{1}{1+ \frac{B^{2}-B_{0}^{2}+\rho^{\gamma}}{\sqrt{\left (B^{2}+B_{0}^{2}-\rho^{\gamma}\right )^{2}+4B^{2}\rho^{\gamma}}}} \quad \text{if }f<\infty. \yesnumber \label{20260117eq12}
\end{IEEEeqnarray*}

Therefore, if $f<\infty$, we have
\begin{IEEEeqnarray*}{rCl}
\partial_{\rho}w &=& (B^{2}+B_{0}^{2})\frac{2}{2-\gamma}f^{\frac{\gamma}{2-\gamma}}\partial_{\rho}f + \frac{2B_{0}^{2}}{2-\gamma}\frac{1-\frac{1}{4}f}{\left (1-\frac{1}{2}f\right )^{2}} f^{\frac{2}{2-\gamma}}\partial_{\rho}f  \\
&=& -\frac{8\gamma}{2-\gamma}f^{\frac{\gamma}{2-\gamma}} \frac{\rho^{\gamma-1}B_{0}^{2}}{\left (B^{2}+B_{0}^{2}-\rho^{\gamma}\right )^{2}+4B^{2}\rho^{\gamma}}\frac{1}{1+ \frac{B^{2}-B_{0}^{2}+\rho^{\gamma}}{\sqrt{\left (B^{2}+B_{0}^{2}-\rho^{\gamma}\right )^{2}+4B^{2}\rho^{\gamma}}}} \left ( B^{2}+\frac{B_{0}^{2}}{\left (1-\frac{1}{2}f\right )^{2}}\right ) \\
&=& -\frac{8\gamma}{2-\gamma}f^{\frac{\gamma}{2-\gamma}} \frac{\rho^{\gamma-1}B_{0}^{2}}{\left (B^{2}+B_{0}^{2}-\rho^{\gamma}\right )^{2}+4B^{2}\rho^{\gamma}}\frac{1}{1+ \frac{B^{2}-B_{0}^{2}+\rho^{\gamma}}{\sqrt{\left (B^{2}+B_{0}^{2}-\rho^{\gamma}\right )^{2}+4B^{2}\rho^{\gamma}}}} \\
&& \phantom{+}\left ( B^{2}+B_{0}^{2}\left ( \frac{-(B^{2}-B_{0}^{2}+\rho^{\gamma})-\sqrt{\left (B^{2}+B_{0}^{2}-\rho^{\gamma}\right )^{2}+4B^{2}\rho^{\gamma}}}{2B_{0}^{2}}\right )^{2}\right ) \\
&=& -\frac{4\gamma}{2-\gamma}\rho^{\gamma-1}f^{\frac{\gamma}{2-\gamma}} \frac{4B_{0}^{2}B^{2}+(B^{2}-B_{0}^{2}+\rho^{\gamma})^{2}+(B^{2}-B_{0}^{2}+\rho^{\gamma})\sqrt{\left (B^{2}+B_{0}^{2}-\rho^{\gamma}\right )^{2}+4B^{2}\rho^{\gamma}}}{(B^{2}+B_{0}^{2}-\rho^{\gamma})^{2}+4B^{2}\rho^{\gamma}+(B^{2}-B_{0}^{2}+\rho^{\gamma})\sqrt{\left (B^{2}+B_{0}^{2}-\rho^{\gamma}\right )^{2}+4B^{2}\rho^{\gamma}}} \\
&=& -\frac{4\gamma}{2-\gamma}\rho^{\gamma-1}f^{\frac{\gamma}{2-\gamma}}, \yesnumber \label{20260117eq13} \\
\partial_{B}w &=& (B^{2}+B_{0}^{2})\frac{2}{2-\gamma}f^{\frac{\gamma}{2-\gamma}}\partial_{B}f + \frac{2B_{0}^{2}}{2-\gamma}\frac{1-\frac{1}{4}f}{\left (1-\frac{1}{2}f\right )^{2}}f^{\frac{2}{2-\gamma}}\partial_{B}f + 2B f^{\frac{2}{2-\gamma}} \\
&=& -\frac{4\gamma}{2-\gamma}\rho^{\gamma-1}f^{\frac{\gamma}{2-\gamma}}\frac{Bf}{\gamma \rho^{\gamma-1}} + 2Bf^{\frac{2}{2-\gamma}},  \\
&=& -\frac{2\gamma}{2-\gamma}Bf^{\frac{2}{2-\gamma}}, \yesnumber \label{20260117eq14} \\
\partial_{\rho}W &=& \exp(-w)\frac{4\gamma}{2-\gamma}\rho^{\gamma-1}f^{\frac{\gamma}{2-\gamma}},  \\
\partial_{B}W &=& \exp(-w)\frac{2\gamma}{2-\gamma}Bf^{\frac{2}{2-\gamma}}. 
\end{IEEEeqnarray*}

Now observe that $\partial_{\rho}W$ and $\partial_{B}W$ converges to zero as $(\rho, B)$ approaches to a point where $f=\infty$. This is because \[
\exp(w) \ge C\exp\left ( C\int_{4}^{f}\frac{s^{\frac{\gamma}{2-\gamma}}}{\left (1-\frac{1}{2}s\right )^{2}}\dd s\right ) \ge C\exp\left ( C f^{\frac{2\gamma-2}{2-\gamma}}\right )
\] as $f \to \infty$. (The last inequality holds since $1<\gamma<2$.) Hence by L'Hospital's rule, we conclude that \begin{IEEEeqnarray*}{rCl}
\partial_{\rho}W &=& \begin{cases}
\exp(-w)\frac{4\gamma}{2-\gamma}\rho^{\gamma-1}f^{\frac{\gamma}{2-\gamma}} & f<\infty, \\
0 & f=\infty.
\end{cases}, \yesnumber \label{20260117eq15} \\
\partial_{B}W &=& \begin{cases}
\exp(-w)\frac{2\gamma}{2-\gamma}Bf^{\frac{2}{2-\gamma}} & f<\infty, \\
0 & f=\infty.
\end{cases}\yesnumber \label{20260117eq16}
\end{IEEEeqnarray*}

Differentiating it once again, we obtain that for any $(\rho, B)$ with $f(\rho, B)<\infty$,
\begin{IEEEeqnarray*}{rCl}
\partial_{\rho\rho}W &=& \frac{4\gamma}{2-\gamma}\exp(-w) \left ( (\gamma-1)\rho^{\gamma-2}f^{\frac{\gamma}{2-\gamma}} + \rho^{\gamma-1}\frac{\gamma}{2-\gamma}f^{\frac{2\gamma-2}{2-\gamma}}\partial_{\rho}f - \rho^{\gamma-1}f^{\frac{\gamma}{2-\gamma}}\partial_{\rho}w\right ) \\
&=& \frac{4\gamma}{2-\gamma}\exp(-w)f^{\frac{\gamma}{2-\gamma}}\rho^{\gamma-1}\\
&& \left ( (\gamma-1)\rho^{-1} - \frac{\gamma}{2-\gamma}f^{-1}\frac{4 \gamma \rho^{\gamma-1}B_{0}^{2}}{\left (B^{2}+B_{0}^{2}-\rho^{\gamma}\right )^{2}+4B^{2}\rho^{\gamma}}\frac{1}{1+ \frac{B^{2}-B_{0}^{2}+\rho^{\gamma}}{\sqrt{\left (B^{2}+B_{0}^{2}-\rho^{\gamma}\right )^{2}+4B^{2}\rho^{\gamma}}}} + \frac{4\gamma}{2-\gamma}\rho^{\gamma-1}f^{\frac{\gamma}{2-\gamma}}\right ) \\
&=& \frac{4\gamma}{2-\gamma}\exp(-w)f^{\frac{\gamma}{2-\gamma}}\rho^{\gamma-1}\\
&& \left ( (\gamma-1)\rho^{-1} - \frac{\gamma}{2-\gamma}f^{-1}\frac{4 \gamma \rho^{\gamma-1}B_{0}^{2}}{\left (B^{2}+B_{0}^{2}-\rho^{\gamma}\right )^{2}+4B^{2}\rho^{\gamma}}\frac{1- \frac{B^{2}-B_{0}^{2}+\rho^{\gamma}}{\sqrt{\left (B^{2}+B_{0}^{2}-\rho^{\gamma}\right )^{2}+4B^{2}\rho^{\gamma}}}}{\frac{4B^{2}\rho^{\gamma}}{\left (B^{2}+B_{0}^{2}-\rho^{\gamma}\right )^{2}+4B^{2}\rho^{\gamma}}} + \frac{4\gamma}{2-\gamma}\rho^{\gamma-1}f^{\frac{\gamma}{2-\gamma}}\right ), \\
\partial_{\rho B}W &=& \frac{4\gamma}{2-\gamma}\exp(-w) \left ( \rho^{\gamma-1}\frac{\gamma}{2-\gamma}f^{\frac{2\gamma-2}{2-\gamma}}\partial_{B}f - \rho^{\gamma-1}f^{\frac{\gamma}{2-\gamma}}\partial_{B}w\right ) \\
&=& \frac{8\gamma^{2}}{\left (2-\gamma\right )^{2}}\exp(-w)f^{\frac{\gamma}{2-\gamma}}\rho^{\gamma-1}B\left ( -\frac{2 B_{0}^{2}}{\left (B^{2}+B_{0}^{2}-\rho^{\gamma}\right )^{2}+4B^{2}\rho^{\gamma}}\frac{1}{1+ \frac{B^{2}-B_{0}^{2}+\rho^{\gamma}}{\sqrt{\left (B^{2}+B_{0}^{2}-\rho^{\gamma}\right )^{2}+4B^{2}\rho^{\gamma}}}}+f^{\frac{2}{2-\gamma}}\right ), \\
\partial_{BB}W &=& \frac{2\gamma}{2-\gamma}\exp(-w)\left ( f^{\frac{2}{2-\gamma}} + \frac{2}{2-\gamma}Bf^{\frac{\gamma}{2-\gamma}}\partial_{B}f - Bf^{\frac{2}{2-\gamma}}\partial_{B}w\right ) \\
&=& \frac{2\gamma}{2-\gamma}\exp(-w)f^{\frac{2}{2-\gamma}}\\
&& \left ( 1 - \frac{2}{2-\gamma}\frac{4 B_{0}^{2}B^{2}}{\left (B^{2}+B_{0}^{2}-\rho^{\gamma}\right )^{2}+4B^{2}\rho^{\gamma}}\frac{1}{1+ \frac{B^{2}-B_{0}^{2}+\rho^{\gamma}}{\sqrt{\left (B^{2}+B_{0}^{2}-\rho^{\gamma}\right )^{2}+4B^{2}\rho^{\gamma}}}} + \frac{2\gamma}{2-\gamma}B^{2}f^{\frac{2}{2-\gamma}}\right ).
\end{IEEEeqnarray*}
It is routine to check that all these second derivatives converge to zero as $(\rho, B)$ converges to $(\rho_{\infty}, 0)$ where $\rho_{\infty}^{\gamma} \le B_{0}^{2}$. Indeed, the only hard part is to check that \[
\exp(-w)f^{\frac{2}{2-\gamma}}\frac{1}{\left (B^{2}+B_{0}^{2}-\rho^{\gamma}\right )^{2}+4B^{2}\rho^{\gamma}}\frac{1}{1+ \frac{B^{2}-B_{0}^{2}+\rho^{\gamma}}{\sqrt{\left (B^{2}+B_{0}^{2}-\rho^{\gamma}\right )^{2}+4B^{2}\rho^{\gamma}}}} \to 0
\] as $(\rho, B) \to (\rho_{\infty}, 0)$ where $\rho_{\infty}^{\gamma}=B_{0}^{2}$, and this follows because
\begin{IEEEeqnarray*}{rCl}
\IEEEeqnarraymulticol{3}{l}{\frac{1}{\left (B^{2}+B_{0}^{2}-\rho^{\gamma}\right )^{2}+4B^{2}\rho^{\gamma}}\frac{1}{1+ \frac{B^{2}-B_{0}^{2}+\rho^{\gamma}}{\sqrt{\left (B^{2}+B_{0}^{2}-\rho^{\gamma}\right )^{2}+4B^{2}\rho^{\gamma}}}}} \\
& \le & \frac{1}{\left (B^{2}+B_{0}^{2}-\rho^{\gamma}\right )^{2}+4B^{2}\rho^{\gamma}}\frac{1}{1+ \frac{-B^{2}-B_{0}^{2}+\rho^{\gamma}}{\sqrt{\left (B^{2}+B_{0}^{2}-\rho^{\gamma}\right )^{2}+4B^{2}\rho^{\gamma}}}} \\
&=& \frac{f}{4\rho^{\gamma}\sqrt{\left (B^{2}+B_{0}^{2}-\rho^{\gamma}\right )^{2}+4B^{2}\rho^{\gamma}}} \\
&\le& \frac{f^{2}}{4\rho^{\gamma}\sqrt{\left (B^{2}+B_{0}^{2}-\rho^{\gamma}\right )^{2}+4B^{2}\rho^{\gamma}}\frac{B_{0}^{2}-\rho^{\gamma}+\sqrt{\left (B^{2}+B_{0}^{2}-\rho^{\gamma}\right )^{2}+4B^{2}\rho^{\gamma}}}{B^{2}}} \\
& = & \frac{f^{2}\left ( \sqrt{\left (B_{0}^{2}-\rho^{\gamma}\right )^{2} + 2B^{2}(B_{0}^{2}+\rho^{\gamma})+B^{4}} - (B_{0}^{2}-\rho^{\gamma})\right )}{4\rho^{\gamma}\sqrt{\left (B^{2}+B_{0}^{2}-\rho^{\gamma}\right )^{2}+4B^{2}\rho^{\gamma}}\left ( B^{2}+2(B_{0}^{2}+\rho^{\gamma})\right )} \\
& \le & \frac{f^{2}}{4\rho^{\gamma}B_{0}^{2}} \ll f^{-\frac{2}{2-\gamma}}\exp(w).
\end{IEEEeqnarray*}
Hence, again by L'Hospital's rule, we conclude that \begin{IEEEeqnarray*}{rCl}
\partial_{\rho\rho}W &=& \begin{cases}
\frac{4\gamma}{2-\gamma}\exp(-w) \left ( (\gamma-1)\rho^{\gamma-2}f^{\frac{\gamma}{2-\gamma}} + \rho^{\gamma-1}\frac{\gamma}{2-\gamma}f^{\frac{2\gamma-2}{2-\gamma}}\partial_{\rho}f - \rho^{\gamma-1}f^{\frac{\gamma}{2-\gamma}}\partial_{\rho}w\right ) & \text{if }f<\infty, \\
0 & \text{if }f=\infty.
\end{cases}, \yesnumber \label{20260117eq17} \\
\partial_{\rho B}W &=& \begin{cases}
\frac{4\gamma}{2-\gamma}\exp(-w) \left ( \rho^{\gamma-1}\frac{\gamma}{2-\gamma}f^{\frac{2\gamma-2}{2-\gamma}}\partial_{B}f - \rho^{\gamma-1}f^{\frac{\gamma}{2-\gamma}}\partial_{B}w\right ) & \text{if }f<\infty, \\
0 & \text{if }f=\infty.
\end{cases}, \yesnumber \label{20260117eq18} \\
\partial_{BB}W &=& \begin{cases}
\frac{2\gamma}{2-\gamma}\exp(-w)\left ( f^{\frac{2}{2-\gamma}} + \frac{2}{2-\gamma}Bf^{\frac{\gamma}{2-\gamma}}\partial_{B}f - Bf^{\frac{2}{2-\gamma}}\partial_{B}w\right ) & \text{if }f<\infty, \\
0 & \text{if }f=\infty.
\end{cases}. \yesnumber \label{20260117eq19}
\end{IEEEeqnarray*} This completes the proof that $(\rho, B) \mapsto W(\rho, B)$ is $C^{2}$.

Similar computation shows that $(\rho, B) \mapsto g, z$ are continuous from $\RR^{+}\times \RR$ to $\RR \cup \{\infty\}$, and that $(\rho, B) \mapsto Z$ is $C^{2}$. The derivatives are given by
\begin{IEEEeqnarray*}{rCl}
\partial_{\rho}g &=& \frac{4B_{0}^{2}\gamma \rho^{\gamma-1}}{\big(B^{2}+B_{0}^{2}-\rho^{\gamma}\big)^{2}+4B^{2}\rho^{\gamma}}\frac{1}{1-\frac{B^{2}-B_{0}^{2}+\rho^{\gamma}}{\sqrt{\big( B^{2}+B_{0}^{2}-\rho^{\gamma}\big)^{2}+4B^{2}\rho^{\gamma}}}} \qquad \text{if }g<\infty, \yesnumber \label{20260117eq20}\\
\partial_{B}g &=& 
-\frac{4B_{0}^{2}Bg}{\big(B^{2}+B_{0}^{2}-\rho^{\gamma}\big)^{2}+4B^{2}\rho^{\gamma}}\frac{1}{1-\frac{B^{2}-B_{0}^{2}+\rho^{\gamma}}{\sqrt{\big( B^{2}+B_{0}^{2}-\rho^{\gamma}\big)^{2}+4B^{2}\rho^{\gamma}}}}   \qquad \text{if }g<\infty, \yesnumber \label{20260117eq21}\\
\partial_{\rho}z &=& \frac{4\gamma}{2-\gamma}g^{\frac{\gamma}{2-\gamma}}\rho^{\gamma-1} \qquad \text{if }g<\infty, \yesnumber \label{20260117eq22}\\
\partial_{B}z &=& -\frac{2\gamma}{2-\gamma}g^{\frac{2}{2-\gamma}}B \qquad \text{if }g<\infty, \yesnumber \label{20260117eq23}\\
\partial_{\rho}Z &=& \begin{cases}
-\exp(-z)\frac{4\gamma}{2-\gamma}g^{\frac{\gamma}{2-\gamma}}\rho^{\gamma-1} & \text{if }g<\infty, \\
0 & \text{if }g=\infty.
\end{cases}, \yesnumber \label{20260117eq24}\\
\partial_{B}Z &=& \begin{cases} \exp(-z)\frac{2\gamma}{2-\gamma}g^{\frac{2}{2-\gamma}}B & \text{if }g<\infty, \\
0 & \text{if }g=\infty.
\end{cases}, \yesnumber \label{20260117eq25}\\
\partial_{\rho\rho}Z &=& \begin{cases} -\frac{4\gamma}{2-\gamma}\exp(-z)\left ( (\gamma-1)g^{\frac{\gamma}{2-\gamma}}\rho^{\gamma-2} + \frac{\gamma}{2-\gamma}\rho^{\gamma-1}g^{\frac{2\gamma-2}{2-\gamma}}\partial_{\rho}g - \rho^{\gamma-1}g^{\frac{\gamma}{2-\gamma}}\partial_{\rho}z \right ) & \text{if }g<\infty, \\
0 & \text{if }g=\infty.
\end{cases}, \yesnumber \label{20260117eq26}\\
\partial_{\rho B}Z &=& \begin{cases}-\frac{4\gamma}{2-\gamma}\exp(-z)\left ( \frac{\gamma}{2-\gamma}\rho^{\gamma-1}g^{\frac{2\gamma-2}{2-\gamma}}\partial_{B}g - \rho^{\gamma-1}g^{\frac{\gamma}{2-\gamma}}\partial_{B}z \right ) & \text{if }g<\infty, \\
0 & \text{if }g=\infty. \end{cases}, \yesnumber \label{20260117eq27}\\
\partial_{BB}Z &=& \begin{cases}\frac{2\gamma}{2-\gamma}\exp(-z)\left ( g^{\frac{2}{2-\gamma}} + \frac{2}{2-\gamma}g^{\frac{\gamma}{2-\gamma}}B\partial_{B}g - g^{\frac{2}{2-\gamma}}B\partial_{B}z\right ) & \text{if }g<\infty, \\
0 & \text{if }g=\infty. \end{cases}.\yesnumber \label{20260117eq28}
\end{IEEEeqnarray*}

\bibliographystyle{plain}
\bibliography{references}

\end{document}